\documentclass[a4paper,reqno,10pt]{amsart}

 \usepackage{pstricks}
 \usepackage{epsfig}

\usepackage{amsmath}
\usepackage{amsthm}
\usepackage{amssymb}
\usepackage{amscd} 

\usepackage{bbm} 
\usepackage{mathrsfs} 

\usepackage[applemac]{inputenc} 

\usepackage{verbatim} 

\usepackage{color}
\definecolor{purple}{rgb}{0.9,0.01,0.3}

\swapnumbers 

\newtheoremstyle{exercise} 
  {3pt} 
  {3pt} 
  {\scriptsize\rmfamily} 
  {
\parindent} 
  {\rmfamily\scshape} 
  {.} 
  {.5em} 
  {} 

\newtheoremstyle{newplain}
  {5pt}
  {5pt}
  {\itshape}
  {}
  {\rmfamily\scshape}
  {. ---}
  {.5em}
  {}

\newtheoremstyle{newremark}
  {5pt}
  {5pt}
  {\rmfamily}
  {}
  {\rmfamily\scshape}
  {. ---}
  {.5em}
  {}

\theoremstyle{newplain}
\newtheorem*{Theorem*}{Theorem} 

\theoremstyle{newplain}
\newtheorem{Theorem}{Theorem}
\newtheorem{Lemma}[Theorem]{Lemma}
\newtheorem{Corollary}[Theorem]{Corollary}
\newtheorem{Proposition}[Theorem]{Proposition}
\newtheorem{Conjecture}[Theorem]{Conjecture}
\newtheorem{Definition}[Theorem]{Definition}

\theoremstyle{newremark}
\newtheorem{Empty}[Theorem]{}
\newtheorem{Remark}[Theorem]{Remark}

\newtheorem{Claim}[Theorem]{Claim}
\newtheorem{Scholium}[Theorem]{Scholium}

\theoremstyle{exercise}

\numberwithin{Theorem}{section}
\numberwithin{Exercise}{section}

\newcommand{\N}{\mathbb{N}} 
\newcommand{\R}{\mathbb{R}} 

\newcommand{\Z}{\mathbb{Z}} 

\newcommand{\calC}{\mathscr{C}}

\newcommand{\calE}{\mathscr{E}}
\newcommand{\calF}{\mathscr{F}}

\newcommand{\calH}{\mathscr{H}}

\newcommand{\calK}{\mathscr{K}}
\newcommand{\calL}{\mathscr{L}}
\newcommand{\calM}{\mathscr{M}}

\newcommand{\calP}{\mathscr{P}}

\newcommand{\calR}{\mathscr{R}}

\newcommand{\balpha}{\boldsymbol{\alpha}}

\newcommand{\boldeta}{\boldsymbol{\eta}}

\newcommand{\bC}{\mathbf{C}}

\newcommand{\bG}{\mathbf{G}}

\newcommand{\bL}{\mathbf{L}}

\DeclareMathOperator{\rmap}{\mathrm{ap}} 
\DeclareMathOperator{\rmbdry}{\mathrm{bdry}} 
\DeclareMathOperator{\rmcard}{\mathrm{card}} 
\DeclareMathOperator{\rmclos}{\mathrm{clos}}

\DeclareMathOperator{\rmdiam}{\mathrm{diam}} 

\DeclareMathOperator{\rmdim}{\mathrm{dim}} 
\DeclareMathOperator{\rmdist}{\mathrm{dist}} 

\DeclareMathOperator{\rmim}{\mathrm{im}} 
\DeclareMathOperator{\rmint}{\mathrm{int}} 

\DeclareMathOperator{\rmLip}{\mathrm{Lip}} 

\DeclareMathOperator{\rmosc}{\mathrm{osc}} 
\DeclareMathOperator{\rmreg}{\mathrm{reg}}
\DeclareMathOperator{\rmsing}{\mathrm{sing}} 
\DeclareMathOperator{\rmspan}{\mathrm{span}} 
\DeclareMathOperator{\rmsupp}{\mathrm{supp}} 
\DeclareMathOperator{\rmTan}{\mathrm{Tan}} 

\newcommand{\blseg}{\pmb{[}}
\newcommand{\brseg}{\pmb{]}}

\newcommand{\hel} {
\hskip2.5pt{\vrule height7pt width.5pt depth0pt}
\hskip-.2pt\vbox{\hrule height.5pt width7pt depth0pt}
\, }

\newcommand{\restr}{\hel}

\newcommand{\cone} {
\times \negthickspace \negthickspace \negthickspace \times \medspace }

\newcommand{\lseg}{\boldsymbol{[}\!\boldsymbol{[}}
\newcommand{\rseg}{\boldsymbol{]}\!\boldsymbol{]}}

\newcommand{\cqfd} {
\renewcommand{\qedsymbol}{$\blacksquare$}
\qed
\renewcommand{\qedsymbol}{$\square$} }

\def\XXint#1#2#3{{%
\setbox0=\hbox{$#1{#2#3}{\int}$}
\vcenter{\hbox{$#2#3$}}\kern-.5\wd0}}

\newcommand{\veps}{\varepsilon}
\newcommand{\bveps}{\boldsymbol{\veps}}
\newcommand{\vphi}{\varphi}

\newcommand{\la}{\langle}
\newcommand{\ra}{\rangle}

\renewcommand{\leq}{\leqslant}
\renewcommand{\geq}{\geqslant}
\renewcommand{\subset}{\subseteq}
\renewcommand{\supset}{\supseteq}

\begin{document}


\title[Weighted length]{On sets minimizing their weighted length in uniformly convex separable Banach spaces}
\author{Thierry De Pauw}
\email{depauw@math.jussieu.fr} 
\author{Antoine Lemenant}
\email{lemenant@ljll.univ-paris-diderot.fr} 
\author{Vincent Millot}
\email{millot@math.jussieu.fr} 

\thanks{This work has been partially supported by the Agence Nationale de la Recherche, through the project ANR-12-BS01-0014-01 GEOMETRYA and ANR 10-JCJC 0106 AMAM.}


\begin{abstract}
We study existence and partial regularity relative to the weighted Steiner problem in Banach spaces. We show $C^1$ regularity almost everywhere for almost minimizing sets in uniformly rotund Banach spaces whose modulus of uniform convexity verifies a Dini growth condition.
\end{abstract}

\maketitle

\tableofcontents



\section{Introduction}

This paper contributes to the study of one dimensional geometric variational problems in an ambient Banach space $X$. We address both existence and partial regularity issues. The paradigmatic {\em weighted Steiner problem} is
\begin{equation*}
(\calP) \begin{cases}
\text{minimize } \int_C w \,d\calH^1 \\
\text{among compact connected sets $C \subset X$ containing $F$} \,.
\end{cases}
\end{equation*}
Here $\calH^1$ denotes the one dimensional Hausdorff measure (relative to the metric of~$X$), $w : X \to (0,+\infty]$ is a weight, and $F$ is a finite set implementing the boundary condition.
\par 
Assuming that problem $(\calP)$ admits finite energy competing sets, we prove existence of a minimizer 
in case $X$ is the dual of a separable Banach space, and $w$ is weakly* lower semicontinuous and bounded away from zero, 
Theorem \ref{ExistenceTH}. Ideas on how to circumvent the lack of compactness that ensues from $X$ being possibly infinite dimensional go back to M. Gromov, \cite{GRO.83}, and have been implemented by L.~Ambrosio and B. Kirchheim \cite{AMB.KIR.00} in the context of metric currents, as well as by L.~Ambrosio and P. Tilli \cite{AMBROSIO.TILLI} in the context of the Steiner problem (with $w \equiv 1$). The novelty here is to allow for a varying weight $w$; the relevant lower semicontinuity of the weighted length is in Theorem \ref{compactness}.
\par 
In studying the regularity of a minimizer $C$ of problem $(\calP)$, we regard $C$ as a member of the larger class of {\em almost minimizing} sets. Our definition is less restrictive than that of F.J. Almgren \cite{ALM.76} who first introduced the concept. A gauge is a nondecreasing function $\xi : \R^+ \setminus \{0\} \to \R^+$ such that $\xi(0+)=0$. We say a compact connected set $C \subset X$ of finite length is $(\xi,r_0)$ almost minimizing in an open set $U \subset X$ whenever the following holds: For every $x \in C \cap U$, every $0 < r \leq r_0$ such that $B(x,r) \subset U$, and every compact connected $C' \subset X$ with
\begin{equation*}
C \setminus B(x,r) = C' \setminus B(x,r)
\end{equation*}
one has
\begin{equation*}
\calH^1(C \cap B(x,r)) \leq (1+\xi(r)) \calH^1(C' \cap B(x,r)) \,.
\end{equation*}
One easily checks that if $C$ is a minimizer of $(\calP)$ then it is $(\xi,\infty)$ almost minimizing in $U = X \setminus F$, where $\xi$ is (related to) the oscillation of the weight $w$, Theorem~\ref{lenghtminimiz}. For instance if $w$ is H\"older continuous of exponent $\alpha$ then $\xi(r)$ behaves asymptotically like $r^\alpha$ near $r=0$.
\par 
In order to appreciate the hypotheses of our regularity results, we now make elementary observations. In case $\rmcard F = 2$ and $w$ is bounded from above and from below by positive constants, each minimizer $C$ of $(\calP)$ is a minimizing geodesic curve $\Gamma$ with respect to the conformal metric induced by $w$, with endpoints those of $F$, Theorem \ref{toporeg}. Since $\calH^1(\Gamma) < \infty$ we infer that $\Gamma$ is a Lipschitz curve. In general not much more regularity seems to ensue from the minimizing property of $\Gamma$. Indeed in the plane $X = \ell_\infty^2$ with $w\equiv1$, every 1-Lipschitz graph over one of the coordinate axes is length minimizing, as the reader will happily check. However if $X$ is a rotund\footnote{or strictly convex} Banach space, then $\Gamma$ must be a straight line segment.  Finally, in case $w$ is merely H\"older continuous the Euler-Lagrange equation for geodesics cannot be written in the classical or even weak sense, and our regularity results do not seem to entail from ODE or PDE arguments, even when the ambient space $X = \ell_2^2$ is the Euclidean plane.
\par 
In Section \ref{sec.4} we report on some properties of $(\xi,r_0)$ almost minimizing sets $C$ in general Banach spaces $X$. It is convenient -- but not always necessary -- to assume that the gauge $\xi$ verifies a Dini growth condition, specifically that
\begin{equation*}
\zeta(r) = \int_0^r \frac{\xi(\rho)}{\rho} \,d \calL^1(\rho) < \infty \,,
\end{equation*}
for each $r > 0$. We show that for each $x \in C \cap U$ the weighted density ratio
\begin{equation*}
\exp[\zeta(r)] \frac{\calH^1(C \cap B(x,r))}{2r}
\end{equation*}
is a nondecreasing function of $0 < r \leq \min \{r_0,\rmdist(x,X \setminus U)\}$, Theorem \ref{monotonicity}. Its limit as $r \downarrow 0$, denoted $\Theta^1(\calH^1 \hel C , x)$, verifies the following dichotomy:
\begin{equation*}
\text{{\bf Either }} \Theta^1(\calH^1 \hel C , x ) = 1 \quad
\text{{\bf Or }}  \Theta^1(\calH^1 \hel C , x) \geq 3/2 \,,
\end{equation*}
Corollary \ref{density.dichotomy}. We then establish that $\rmreg(C) := C \cap U \cap \{ x : \Theta^1(\calH^1 \hel C , x) = 1 \}$ is relatively open in $C \cap U$ and that for each $x \in \rmreg(C)$ and every $\delta > 0$ there exists $0 < r < \delta$ such that $C \cap B(x,r)$ is a Lipschitz curve that intersects $\rmbdry B(x,r)$ exactly in its two endpoints, Theorems \ref{lip.reg} and \ref{lip.reg.bis}.
\par 
In Section \ref{sec.5} we improve on the regularity of $\rmreg(C)$. Assume the ambient Banach space $X$ is uniformly rotund\footnote{or uniformly convex} and let $\delta_X(\veps)$ denote its modulus of uniform rotundity, see \ref{UR}. Let $x \in \rmreg(C)$. We assume $r_j \downarrow 0$ and each $C \cap B(x,r_j)$ is a Lipschitz curve $\Gamma_j$ with endpoints $x_j^-$ and $x_j^+$ on $\rmbdry B(x,r_j)$. We let $L_j$ be the affine line containing $x_j^-$ and $x_j^+$. We want to show that $\Gamma_j$ does not wander too far away from $L_j$, i.e. we seek for an upper bound of $\max_{z \in \Gamma_j} \rmdist(z,L_j)$. Suppose this maximum equals $h_jr_j$ and is achieved at $z \in \Gamma_j$. The triangle inequality implies $\calH^1(\Gamma_j) \geq \|z-x_j^-\| + \|x_j^+-z\|$. As $X$ is uniformly rotund, the latter is quantitatively larger than the length of the straight line segment joining $x_j^-$ and $x_j^+$, specifically
\begin{equation*}
\|z-x_j^-\| + \|x_j^+-z\| \geq \|x_j^+-x_j^-\| \left( 1 + \delta_X(Ch_j) \right) \,,
\end{equation*}
Theorem \ref{excess.length}. On the other hand, the almost minimizing property of $C$ says that
\begin{equation*}
\calH^1(\Gamma_j) \leq (1 + \xi(r_j)) \|x_j^+-x_j^-\| \,.
\end{equation*}
It now becomes clear that $h_j$ cannot be too large, in fact
\begin{equation*}
h_j \leq C (\delta_X^{-1} \circ \xi)(r_j) \,,
\end{equation*}
which in turns yields the Hausdorff distance estimate
\begin{equation*}
\rmdist_\calH(\Gamma_j , L_j \cap B(x,r_j)) \leq C' (\delta_X^{-1} \circ \xi)(r_j) \,.
\end{equation*}
Upon noticing that the {\em good} radii $r_j$ can be chosen in near geometric progression, we infer that the sequence of affine secant lines $\{L_j\}$ is Cauchy provided
\begin{equation}
\label{eq.intro.1}
\sum_{j=1}^\infty (\delta_X^{-1} \circ \xi)(2^{-j}) < \infty \,.
\end{equation}
The fact that the relevant inequalities are also locally uniform in $x$ then yields our main $C^1$ regularity Theorem \ref{mainregth} under the assumption that $\delta_X$ and $\xi$ verify the Dini growth condition \eqref{eq.intro.1}.  
\par 
In case $\xi(r) \cong r^\alpha$, a change of variable shows that \eqref{eq.intro.1} in fact involves solely $\delta_X$, namely it is equivalent to asking that
\begin{equation}
\label{eq.intro.2}
\sum_{j=1}^\infty \delta_X^{-1}(2^{-j}) < \infty \,.
\end{equation}
The condition is met for instance by all $\bL_p$ spaces, $1 < p < \infty$, as shown by the Clarkson inequalities. The case when $X=\ell_2^n$ is a finite dimensional Euclidean space has been worked out for instance in \cite{MOR.94} (see also \cite[Section 12]{d3} and \cite{MEI.09}).

 In Section \ref{sec.6} we apply our existence and regularity results to quasihyperbolic geodesics for instance in $\bL_p$ spaces. It is perhaps worth noticing that even in the finite dimensional setting $X = \ell_p^n$, $2 < p < \infty$, the problem is not ``elliptic'', or rather the metric is not Finslerian, as the smooth unit sphere $S_{\ell_p^n}$ has vanishing curvature at $\pm e_1,\ldots, \pm e_n$. In fact, in case $X$ is finite dimensional and the unit sphere $S_X$ is $C^\infty$ smooth, \eqref{eq.intro.2} may be understood as a condition on the order of vanishing of
\begin{equation*}
f_v : T_v S_X \to \R : h \mapsto \|v+h\|-1 \,,
\end{equation*}
$v \in S_X$. With this in mind, we show in Section \ref{sec.8} how to completely dispense with \eqref{eq.intro.2} in case $\rmdim X = 2$, and the norm of $X$ is rotund and $C^2$. The relevant regularity result Theorem \ref{87} states that $\rmreg(C)$ is made of {\em differentiable} curves (not necessarily $C^1$) provided $C$ is $(\xi,r_0)$ almost minimizing and $\sqrt{\xi}$ is Dini. In order to prove this we localize the modulus of continuity $\delta_X(v;\veps)$ relative to each direction $v \in S_X$. We then consider the subset $G = S_X \cap \{ v : \partial^2_{h,h}f_v(0) > 0 \}$. We observe it is relatively open in $S_X$, and its complement $S_X \setminus G$ is nowhere dense because the norm is rotund, i.e. the unit circle $S_X$ contains no line segment. Furthermore, if $v \in G$ then $\delta_X(v;\veps) \geq c(v)\veps^2$, the best case scenario for regularity. To prove the differentiability of $\rmreg(C)$ at $x \in \rmreg(C)$ we need only to establish that the set of tangent lines $\rmTan(C,x)$ is a singleton. This set is connected, according to D. Preiss, \cite{PRE.87}. Thus either $L \in \rmTan(C,x) \cap G \neq \emptyset$ and we can run the regularity proof of Section \ref{sec.5} ``in a cone about $L$'', or $\rmTan(C,x) \subset S_X \setminus G$ and therefore $\rmTan(C,x)$ is a singleton.

\section{Preliminaries}

\begin{Empty}[Metric spaces]
In a metric space $(E,d)$ we define the open and closed $r$-neighborhoods of a subset $A \subset E$ by the relations
\begin{equation*}
\begin{split}
U(A,r) & = E \cap \{ y : \rmdist(y,A) < r \} \\
B(A,r) & = E \cap \{ y : \rmdist(y,A) \leq r \}
\end{split}
\end{equation*}
where $\rmdist(y,A) = \inf \{ d(y,x) : x \in A \}$. If $A=\{x\}$ is a singleton, these are the usual open and closed balls $U(x,r)$ and $B(x,r)$. The interior, closure and boundary of a subset $S \subset E$ are respectively denoted by $\rmint S$, $\rmclos S$ and $\rmbdry S$.
\end{Empty}

\begin{Empty}[The ambient Banach space $X$]
Throughout this paper $X$ denotes a Banach space with $\rmdim X \geq 2$. We do not merely care about the isomorphic type of $X$, but also about the specific given norm. Changing the norm for an equivalent one affects the corresponding Hausdorff measure, and therefore also the solutions of the variational problems we are interested in, as well as their regularity theory. Various collections of further requirements about $X$ are made in distinct sections. Specifically:
\begin{enumerate}
\item[(3)] In Section \ref{sec.3}, $X$ is the dual of a separable space;
\item[(4)] In Section \ref{sec.4}, $X$ is an arbitrary separable Banach space;
\item[(5)] In section \ref{sec.5}, $X$ is uniformly rotund, and the main result \ref{mainregth} applies when $\delta_X^{-1}$ verifies a Dini growth condition, $\delta_X$ being the modulus of uniform rotundity of $X$;
\item[(6)] In section \ref{sec.6}, $X$ is as in section \ref{sec.5};
\item[(7)] In section \ref{sec.8}, $X$ is a finite dimensional (uniformly) rotund space with $C^2$ smooth norm, and the main result \ref{87} also assumes that $\rmdim X = 2$.
\end{enumerate}
\end{Empty}

\begin{Empty}[Hausdorff distance]
In a metric space $(E,d)$ we define the {\em Hausdorff distance} between two closed sets $A_1,A_2 \subset E$ as
\begin{equation*}
\rmdist_\calH(A_1,A_2) = \inf \{ r > 0 : A_1 \subset B(A_2,r) \text{ and } A_2 \subset B(A_1,r) \} \,.
\end{equation*}
If $(E,d)$ is compact then the {\em Blaschke selection principle} asserts that $(\calK(E),\rmdist_\calH)$ is a compact metric space, where $\calK(E)$ denotes the collection of nonempty compact subsets of $E$. It is easily seen that 
\begin{enumerate}
\item[(I)] If $\lim_n \rmdist_\calH(A_n,A)=0$ and $x \in A$ then there is a sequence $\{x_n\}$ in $E$ such that $x_n \in A_n$, $n=1,2,\ldots$, and $\lim_n d(x_n,x)=0$;
\item[(II)] If $\lim_n \rmdist_\calH(A_n,A)=0$, $x \in E$ and $\{x_n\}$ is a sequence in $E$ such that $x_n \in A_n$, $n=1,2,\ldots$, and $\lim_n d(x_n,x)=0$, then $x \in A$.
\end{enumerate}
\end{Empty}

\begin{Empty}[Hausdorff measure]
Given a metric space $(E,d)$ we will consider the {\em 1 dimensional Hausdorff outer measure} $\calH^1$ defined for subsets $A \subset E$ by the following formulas:
\begin{multline*}
\calH^1_{(\delta)}(A) = \inf \bigg\{ \sum_{i \in I} \rmdiam A_i : \{A_i\}_{i \in I} \text{ is a finite or countable family of}\\ \text{subsets of $E$ such that } A \subset \cup_{i \in I} A_i \text{ and } \rmdiam A_i \leq \delta \text{ for all } i \in I \bigg\}
\end{multline*}
corresponding to each $0 < \delta \leq \infty$, and
\begin{equation*}
\calH^1(A) = \sup_{\delta} \calH^1_{(\delta)}(A) \,.
\end{equation*}
All Borel subsets of $E$ are $\calH^1$ measurable in the sense of Caratheodory, and the definition of $\calH^1(A)$ remains unchanged if we restrict to {\em closed} (resp. {\em open}) covers $\{A_i\}_{i \in I}$ in the definition of $\calH^1_{(\delta)}(A)$. If we want to insist about the underlying metric space we will write $\calH^1_E$ instead of $\calH^1$. It is useful to note that if $F \subset E$ is considered as a metric space $(F, d \restriction F \times F)$ then $\calH^1_E(F) = \calH^1_F(F)$. Finally, if $E$ is a normed linear space and $a,b \in E$, we define the line segment with endpoints $a,b$ by $\blseg a,b \brseg = E \cap \{ a + t(b-a) : 0 \leq t \leq 1 \}$, and one checks that $\calH^1( \blseg a,b \brseg ) = \|b-a\|$.
\end{Empty}

\begin{Empty}[A covering theorem]
\label{25}
Given a metric space $(E,d)$ and $C \subset E$, we define the {\em enlargement} of $C$ as
\begin{equation*}
\hat{C} = B(C, 2 (\rmdiam C)) = E \cap \{ x : \rmdist(x,C) \leq 2(\rmdiam C) \} \,.
\end{equation*}
In particular $\widehat{B(x,r)} \subset B(x,5r)$, $x \in E$, $r > 0$.
\par 
A {\em Vitali cover} of $A \subset E$ is a collection $\calC$ of closed subsets of $E$ with the following property: {\em For every $x \in A$ and every $\delta > 0$ there exists $C \in \calC$ such that $x \in C$ and $\rmdiam C < \delta$}. It follows from \cite[2.8.6]{GMT} that $\calC$ admits a disjointed subcollection $\calC^*$ with the following property: {\em For every finite $\calF \subset \calC^*$ one has
\begin{equation*}
A \setminus \cup \calF \subset \cup \{ \hat{C} : C \in \calC^* \setminus \calF \} \,.
\end{equation*}
}
\end{Empty}

\begin{Empty}[Comparing measures]
\label{26}
Given a metric space $(E,d)$ and a {\em finite Borel measure} $\mu$ on $E$, we define at each $x \in E$ the following generalized upper density:
\begin{equation*}
\widetilde{\Theta}^1(\mu,x) = \lim_{\delta \to 0^+} \sup \left\{ \frac{\mu(C)}{\rmdiam C} : x \in C \subset E, \;C \text{ is closed, and } 0 < \rmdiam C < \delta \right\}\,.
\end{equation*}
\par 
{\em
If $A \subset E$ is Borel, $0 < t < \infty$, and $\widetilde{\Theta}^1(\mu,x) \geq t$ for every $x \in A$, then $\mu(A) \geq t \calH^1(A)$.
}
\par 
In order to prove this we fix $\delta > 0$, $\veps > 0$, and we choose an open set $U \subset E$ containing $A$ such that $\mu(U) \leq \veps + \mu(A)$. Our assumption guarantees that with each $x \in A$ and $i\in\{1,2,\ldots\}$ large enough we can associate a closed set $C_{x,i} \subset U$ such that $x \in C_{x,i}$, $0 < \rmdiam C_{x,i} < i^{-1} \delta$, and $\mu(C_{x,i}) \geq t (1-\veps) (\rmdiam C_{x,i})$. We extract a disjointed subfamily $\{C_j\}_{j \in J}$ of $\{C_{x,i} : x \in A, i\in\{1,2,\ldots\}, C_{x,i} \subset U\}$ according to~\ref{25}. Since for each finite $F \subset J$ one has
\begin{equation*}
\sum_{j \in F} \rmdiam \widehat{C_j} \leq 5 \sum_{j\in F} \rmdiam C_j \leq 5 (1-\veps)^{-1}t^{-1} \mu(E) < \infty \,,
\end{equation*}
and since $\rmdiam C_j > 0$ for every $j \in J$, we infer that $J$ is at most countable. Thus we as well assume $J = \N$ and
we may select $k$ large enough for
\begin{equation*}
\sum_{j=k+1}^\infty \rmdiam \widehat{C_j} \leq \veps \,.
\end{equation*}
Thus,
\begin{multline*}
\calH^1_{(\delta)}(A) \leq \sum_{j=1}^k \rmdiam C_j + \sum_{j=k+1}^\infty \rmdiam \widehat{C_j} \leq (1-\veps)^{-1}t^{-1} \sum_{j=1}^k \mu(C_j) + \veps \\
\leq (1-\veps)^{-1}t^{-1}(\veps + \mu(A)) + \veps \,.
\end{multline*}
Letting $\veps \to 0$ and $\delta \to 0$ completes the proof.
\end{Empty}

\begin{Empty}[The multiplicity function and Eilenberg's inequality]
\label{eilenberg}
Here we consider a {\em complete separable} metric space $(E,d)$ and a Borel function $f : E \to \R$.  We recall that $f(A)$ is $\calL^1$ measurable whenever $A \subset E$ is Borel, see e.g. \cite[Lemma 8.6.1 and Corollary 8.4.3]{COHN}. It thus follows as in \cite[2.10.10]{GMT} that the {\em multiplicity function}
\begin{equation*}
\R \to \N \cup \{ \infty \} : r \mapsto \rmcard( A \cap f^{-1}\{r\} )
\end{equation*}
is $\calL^1$ measurable. 
\par 
Assuming furthermore that $f$ be Lipschitz, the {\em Eilenberg's inequality} \cite[2.10.25]{GMT} states that
\begin{equation*}
\int_\R \rmcard( A \cap f^{-1} \{r \} )\, d\calL^1(r) \leq (\rmLip f) \calH^1(A) \,.
\end{equation*}
We will often apply these two results to the case when $f(x) = d(x,x_0)$, $x_0 \in X$.
\end{Empty}

\begin{Empty}[Curves]
\label{prelim.curves}
A {\em curve} in a metric space $(E,d)$ is a topological line segment, i.e. a set $\Gamma \subset E$ of the type $\Gamma = \gamma([a,b])$ where $a < b$ are real numbers and $\gamma : [a,b] \to E$ is an injective continuous map. We call $\gamma(a)$ and $\gamma(b)$ the {\em endpoints} of $\Gamma$, and we write $\mathring{\Gamma}:=\Gamma\setminus\{\gamma(a),\gamma(b)\}$. If $x \in \Gamma$ is not an endpoint then $\Gamma \setminus \{x\}$ has two components. If $\calH^1(\Gamma) < \infty$ then there exists an injective $\gamma' : [0,\calH^1(\Gamma)] \to E$ such that $\rmLip \gamma' \leq 1$ and $\rmim \gamma' = \Gamma$ as well. If $S \subset E$ is compact connected, and $\calH^1(S) < \infty$, then for each distinct $x,x' \in S$ there exists a curve contained in $S$ whose endpoints are $x$ and $x'$ (see e.g. \cite[4.4.7]{AMBROSIO.TILLI}).
\end{Empty}

\begin{Empty}[Gauges and Dini Gauges]
\label{gauges}
Given an interval $I = \R \cap \{ r : 0 < r \leq b \}$, $0 < b < \infty$, a {\em gauge on $I$} is a nondecreasing function
\begin{equation*}
\xi : I \to \R+
\end{equation*}
such that $\lim_{r \to 0+} \xi(r) = 0$. We often omit to specify the interval $I$ when it is clearly determined by the context. We say that a gauge $\xi$ on $I$ {\em is a Dini gauge} provided
\begin{equation*}
\zeta(r) := \int_0^r \frac{\xi(\rho)}{\rho} \,d\calL^1(\rho) < \infty \,,
\end{equation*}
$r \in I$, and we call $\zeta$ the {\em mean slope} of $\xi$. Notice $\zeta$ is a gauge as well. 
\par 
The following are useful examples of gauges. If $\xi(r) \leq  a r^\alpha$, $a > 0$, $0 < \alpha \leq 1$, we call $\xi$ a {\em geometric gauge} and we easily check that it is Dini with $\zeta(r) = \alpha^{-1} \xi(r)$. As another class of examples we consider the gauges $\xi(r) = a | \log r |^{-1-\alpha}$, $0 < r < 1$, corresponding to $\alpha > 0$ and $a > 0$. We call these {\em log-geometric gauges} and we check they are Dini as well, with $\zeta(r) = \alpha^{-1} a | \log r |^{-\alpha}$. The gauge $\xi(r) = | \log r|^{-1}$, $0 < r < 1$, however, is {\em not} Dini.
\par 
Let $\beta > 1$. Define $I_j = [\beta^{-(j+1)},\beta^{-j}]$, $j \in \N$. For any gauge $\xi$ in $I$ and any $I_j \subset I$ one has
\begin{multline*}
\left( \frac{\beta-1}{\beta} \right) \xi(\beta^{-(j+1)}) \leq \calL^1(I_j) \left( \inf_{\rho \in I_j} \frac{\xi(\rho)}{\rho} \right) \\
\leq \int_{I_j} \frac{\xi(\rho)}{\rho}\,d\calL^1(\rho) \\
\leq \calL^1(I_j) \left( \sup_{\rho \in I_j} \frac{\xi(\rho)}{\rho} \right) \leq (\beta-1) \xi(\beta^{-j}) \,.
\end{multline*}
Thus the appropriate comparison tests imply that {\em $\xi$ is Dini if and only if}
\begin{equation*}
\sum_{\substack{j=0 \\ I_j \subset I}}^\infty \xi(\beta^{-j}) < \infty \,.
\end{equation*}
Furthermore,
\begin{equation*}
\sum_{j=k}^\infty \xi(\beta^{-j}) \leq \left( \frac{\beta}{\beta-1} \right) \zeta(\beta^{-(k-1)}) \,,
\end{equation*}
whenever $k$ is sufficiently large for $I_{k-1} \subset I$. Given $r$ such that $\beta^2r \in I$ and choosing $j$ such that $r \in I_j$, this also implies that
\begin{equation*}
\xi(r) \leq \xi(\beta^{-j}) \leq \left( \frac{\beta}{\beta-1} \right) \zeta(\beta^{-(j-1)}) \leq \left( \frac{\beta}{\beta-1} \right) \zeta(\beta^2 r) \,.
\end{equation*}
\end{Empty}

\begin{Empty}[Uniformly rotund spaces]
\label{UR}
We recall that a Banach space $X$ is called {\em uniformly rotund}\footnote{or {\em uniformly convex}} (abbreviated {\em UR}) whenever the following holds. For every $\veps>0$ there exists $\delta >0$ such that for every $x, y\in B_X$,
\begin{equation}
\label{eqq.1}
\|x-y\|\geq \varepsilon \Rightarrow \left \| \frac{x+y}{2}\right\|\leq 1-\delta\,.
\end{equation}
Notice that, corresponding to a fixed $\veps > 0$, the set of those $0 < \delta < 1$ for which \eqref{eqq.1} holds is a closed interval. Thus there exists a nondecreasing choice $\veps \mapsto \delta(\veps)$ for which \eqref{eqq.1} is valid. In fact, given an arbitrary Banach space $X$ and $0 < \veps \leq 2$, we put
\begin{equation*}
\delta_X(\veps) = \inf \left\{ 1 - \left \| \frac{x+y}{2} \right\| : x,y \in X , \max\{\|x\|,\|y\|\} \leq 1 \, \text{ and } \|x-y\|\geq\veps \right\} \,.
\end{equation*}
It is most obvious that $\delta_X$ is a gauge. One notices that $X$ is uniformly rotund if and only if $\delta_X(\veps) > 0$ for every $0 < \veps \leq 2$. In this case $\delta_X$ is called the {\em modulus of uniform rotundity} of $X$.
\par 
We also define
\begin{equation*}
\delta_X^{-1}(t) = \sup \{ \veps > 0 : \delta_X(\veps) \leq t \}
\end{equation*}
and we readily infer that $\delta_X(\veps) \leq t$ implies $\veps \leq \delta_X^{-1}(t)$ for all $\veps > 0$ and all $t > 0$. The gauge $\delta_X^{-1}$, particularly its growth, pertains to the regularity theory of Section~\ref{sec.5}.

\begin{Remark}\label{stupidremark} 
In the definition of $\delta_X(\varepsilon)$ one may require that $\|x\|=\|y\|=1$ instead of $\max\{\|x\|,\|y\|\}\leq 1$. This leads to an equivalent definition of rotundity. 
 \footnote{{\it Definition 1.e.1 (Vol. II Chap. 1 Paragraph e) in \cite{LINDENSTRAUSS.TZAFRIRI.II}}.}
\end{Remark}

\end{Empty}

\section{Existence}
\label{sec.3}

\begin{Empty}[Local hypotheses about the ambient Banach space]
In this section $X$
is the dual of a separable Banach space, with norm $\|\cdot\|$. Its closed unit ball $B_X$ equipped with the restriction of the weak* topology of $X$ is a compact separated topological space. It is metrizable as well, owing to the separability of a predual of $X$. We let $d^*$ denote any metric on $B_X$ compatible with its weak* topology, for instance
\begin{equation*}
d^*(x_1,x_2)=\sum_{n} 2^{-n}|\langle y_n,  x_1-x_2\rangle|\;,
\end{equation*}
where $y_1,y_2,\ldots$ is a dense sequence of the unit ball of some predual of $X$. Notice that $d^*(x_1,x_2) \leq \|x_1-x_2\|$.  In the compact metric space $(B_X,d^*)$ we denote the corresponding Hausdorff distance as $\rmdist_\calH^*$.
 We consider two metrizable topologies on $B_X$: that induced by the norm of $X$, and that induced by the weak* topology of $X$. When we refer to {\em closed} (resp. {\em compact}) subsets $C \subset B_X$ we always mean strongly closed (resp. compact), i.e. with respect to the norm topology of $X$, and we use the terminology {\em weakly* closed} (resp. {\em weakly* compact}) otherwise.
\end{Empty}


\begin{Lemma} 
\label{Ahlfors} 
Let $(E,d)$ be a metric space. 
\begin{enumerate}
\item[(A)] If $C \subset E$ is connected, $x \in C$ and $0 < r \leq \rmdiam C$, it follows that 
\begin{equation*}
\calH^1(C \cap B(x,r)) \geq r \,;
\end{equation*}
\item[(B)] If $\Gamma$ if a curve in $E$ with endpoints $a$ and $b$ then $\calH^1(\Gamma) \geq d(a,b)$.
\end{enumerate}
\end{Lemma}

\begin{proof} 
(A) There is no restriction to assume $C$ is nonempty.
Given $x \in C$ we consider the Lipschitz function $u : E \to \R : y \mapsto d(y,x)$. Since $C$ is connected so is $u(C)$, and $r \in \rmclos u(C)$ whenever $0 < r \leq  \rmdiam C$. As $\rmLip u \leq 1$ we infer that
\begin{equation*}
r = \calH^1_{\R}([0,r [ ) = \calH^1_{\R}(u(B(x,r) \cap C)) \leq \calH^1_E(B(x,r) \cap C)  \,.
\end{equation*}
(B) follows from (A) on letting $C = \Gamma$, $x=a$, and $r=d(a,b)$.
\end{proof}


\begin{Lemma}
\label{subsequence}
Every sequence $\{C_n\}$ of nonempty compact subsets of $B_X$ contains a subsequence $\{C_{k(n)}\}$ such that $\rmdist_{\calH}^*(C_{k(n)},C) \to 0$ as $n \to \infty$ for some nonempty closed set $C \subset B_X$.
\end{Lemma}

\begin{proof}
Upon noticing that each $C_n$ is weakly* compact, this becomes a consequence of the Blaschke selection principle applied to the compact metric space $(B_X,d^*)$, and the fact that a weakly* compact set $C$ is closed.
\end{proof}

\begin{Theorem}[Compactness and lower semicontinuity]
\label{compactness}
Assume that
\begin{enumerate}
\item[(A)]
$\{C_n\}$  is a sequence of nonempty compact connected subsets of $B_X$;
\item[(B)]
$\rmdist_\calH^*(C_n,C)\to 0$ for some nonempty closed subset $C$ of  $B_X$;
\item[(C)]
$w : B_X \to (0,+\infty]$ 
is weakly* lower semicontinuous and 
\begin{equation*}
\sup_n \int_{C_n} w \,d\calH^1 < \infty \,.
\end{equation*}
\end{enumerate}
It follows that
\begin{enumerate}
\item[(D)] $C$ is compact and connected;
\item[(E)] $\int_C w \,d\calH^1\leq \liminf_n \int_{C_n} w \, d\calH^1$;
\item[(F)]
$F\subset C$ whenever $F\subset C_n$ for every $n=1,2,\ldots$. 
\end{enumerate}
\end{Theorem}

\begin{Remark}
If the function $w$ fails to be weakly* lower semicontinuous, conclusion (E) does not need to hold, as the following counterexample shows.  Denote by $\{e_k\}_{k=1}^\infty$ the canonical orthonormal basis of $X=\ell_2$, and define $w:X\to [1,2]$ by $w(x):=\max\{1,2-8\rmdist(x, \rmspan\{e_1\}) \}$. Then consider the sequence $\{C_n\}\subset B_X$ of compact connected sets $C_n:=\gamma_n([0,1])$ where  
$$\gamma_n(t):=\begin{cases}
t e_n & \text{for $0\leq t\leq 1/8$}\,,\\[5pt]
\frac{1}{8} e_n +(t-\frac{1}{8})e_1  & \text{for $1/8< t\leq 7/8$}\,,\\[5pt]
(1-t)e_n + \frac{3}{4}e_1 & \text{for $7/8< t\leq 1$}\,.
\end{cases}
$$
One easily checks that \ref{compactness} (B) holds with  $C=[0,\frac{3}{4}e_1]$. On the other hand we have $\int_{C_n} w \, d\calH^1=\frac{9}{8}$ for every $n=1,2,\ldots$, while $\int_{C} w \, d\calH^1=\frac{3}{2}>\frac{9}{8}$. 
\end{Remark}

\begin{proof}
Conclusion (F) is a trivial consequence of assumption (B). If $\inf_n \rmdiam C_n = 0$ then $\rmdiam^* C = \lim_n \rmdiam^* C_n \leq \liminf_n \rmdiam C_n = 0$, thus $C$ is a singleton and there is nothing to prove. We henceforth assume that $a:=\inf_n \rmdiam C_n >0$. The weak* compactness of $B_X$ together with the nonvanishing and weak* lower semicontinuity of $w$ guarantee that $\eta := \inf_{B_X} w > 0$. Therefore $\calH^1(C_n) \leq \eta^{-1} \int_{C_n} w \,d\calH^1$, $n=1,2,\ldots$, and it ensues from (C) that $b:=\sup_n\calH^1(C_n) < \infty$. 
\par 
We claim that the sequence of metric spaces $\{C_n\}$ is equicompact. Indeed given $r > 0$, $n=1,2,\ldots,$ and $x_1,\ldots,x_{\kappa_n}$ in $C_n$ which are pairwise a distance at least $2r$ apart, it follows from Lemma \ref{Ahlfors} that
\begin{equation*}
\kappa_n r \leq \sum_{k=1}^{\kappa_n} \calH^1(C_n \cap B(x_k,r)) \leq \calH^1(C_n) \leq b \,,
\end{equation*}
whence $\kappa_n$ is bounded independently of $n$. It follows from the Gromov compactness Theorem
 (see e.g. \cite[4.5.7]{AMBROSIO.TILLI}) that there exists a compact metric space $(Z,d_Z)$, a subsequence of $\{C_n\}$ which we still denote as $\{C_n\}$, and isometric embeddings $i_n:(C_n,\|\cdot\|)\to (Z,d_Z)$ such that $D_n:=i_n(C_n)$ converge in Hausdorff distance in $Z$ to some compact set $D\subset Z$. 
\par 
We now consider the mappings $j_n:=i_n^{-1}: (D_n,d_Z)\to (C_n,\|\cdot\|)$. 
We claim that, restricting to a subsequence of $\{D_n\}$ if necessary (still denoted by $\{D_n\}$), {\em there exists a 1-Lipschitz map $j:(D,d_Z)\to (B_X,\|\cdot\|)$ with the following property: For any sequence $\{z_{k(n)}\}$ in $Z$ satisfying $z_{k(n)}\in D_{k(n)}$, $n=1,2,\ldots$, and $z_{k(n)}\to z \in D$, we have $d^*(j_{k(n)}(z_{k(n)}), j(z))\to 0$.}
In order to prove this we consider the graphs of~$j_n$,
$$G_n= (Z \times B_X) \cap \{(z,j_n(z)) : z\in D_n \}\,. $$
According to the Blaschke selection principle $\{G_n\}$ subconverges in Hausdorff distance, in the compact metric space $(Z,d_Z)\times (B_X,d^*)$, to some  compact set $G$. One readily checks that the projection of $G$ on $Z$ equals~$D$. In addition, we observe that for any pair $(z_1,x_1),(z_2,x_2)\in G$ we can find  $(z^n_1,j_n(z_1^n)),(z^n_2,j_n(z_2^n))\in G_n$ such that   $\{(z_k^n,j_n(z_k^n))\}_n$ converges to $(z_k,x_k)$ in $(Z,d_Z)\times (B_X,d_w)$, $k=1,2$, as $n\to\infty$.  Referring to the weak* lower semicontinuity of $\|\cdot\|$, and to the fact that $j_n$ is an isometry, we infer that 
$$\|x_1-x_2\|\leq \liminf_{n\to\infty} \|j_n(z_1^n)-j_n(z_2^n)\|=  \liminf_{n\to\infty} d_Z(z_1^n, z_2^n)=d_Z(z_1,z_2)\,. $$
Consequently $G$ is the graph in $Z\times B_X$ of a 1-Lipschitz map $j:(D,d_Z)\to(B_X,\|\cdot\|)$, i.e. 
$G=\{(z,j(z)) : z\in D)\}$.
In order to complete the proof of our claim, we need to establish the asserted property of $j$. We consider a sequence $\{z_{k(n)}\}$ in $Z$ such that $z_{k(n)} \in D_{k(n)}$, $n=1,2,\ldots$, and $d_Z(z_{k(n)},z) \to 0$ for some $z \in D$. Any subsequence of $\{j_{k(n)}(z_{k(n)})\}$ contains a subsequence itself converging weakly* to some $x \in B_X$. The Hausdorff distance convergence of $\{G_n\}$ to $G$ then implies that $(z,x) \in G$, i.e. $x=j(z)$. Since this is independent of the original subsequence, the conclusion follows.
\par 
We now establish that $C=j(D)$, starting with the inclusion $C \subset j(D)$. Given $x\in C$ we choose $x_n\in C_n$, $n=1,2,\ldots$, such that $d^*(x_n,x)\to 0$. Letting $z_n:=i_n(x_n)$, $n=1,2,\ldots$ we infer from the compactness of $Z$ that a suitable subsequence $\{z_{k(n)}\}$ of $\{z_n\}$ converges to some $z \in Z$. The Hausdorff convergence of $\{D_n\}$ to $D$ implies that $z \in D$, and in turn the claim of the preceding paragraph implies that $\lim_n d^*(x_{k(n)},j(z))=\lim_n d^*(j_{k(n)}(z_{k(n)}),j(z))=0$, thus $x = j(z)$. The other way round, given $z\in D$ we choose $z_n\in D_n$, $n=1,2,\ldots$, such that $d_Z(z_n,z) \to 0$. The claim of the preceding paragraph implies that $d^*(j_n(z_n),j(z))\to 0$. Since $j_n(z_n)\in C_n$ we conclude that $j(z)\in C$.
\par 
The connectedness of $D$ follows from that of each $D_n$, the Hausdorff convergence of $\{D_n\}$ to $D$ and the relation $\beta := \sup_n \calH^1(D_n) < \infty$, in the following fashion. Given $z_1,z_2 \in D$ and $n=1,2,\ldots$, we choose a curve $\Gamma_n \subset D_n$ with endpoints $z_1,z_2$. Since $\calH^1(\Gamma_n) \leq \beta$ we may select a parametrization $\gamma_n : [0,1] \to Z$ of $\Gamma_n$ so that $\rmLip \gamma_n \leq \beta$. It follows from the Arzela-Ascoli Theorem and the compactness of $Z$ that some subsequence of $\{\gamma_n\}$ converges uniformly to some Lipschitz $\gamma : [0,1] \to Z$. One readily checks that $\Gamma = \rmim \gamma$ is a curve in $D$ with endpoints $z_1$ and $z_2$. Conclusion (D) follows at once from the equality $C=j(D)$.
\par 
We now turn to proving conclusion (E). There is no restriction to assume that $\liminf_n \int_{C_n} w \, d\calH^1 < \infty$ and, extracting a subsequence of $\{D_n\}$ in the first place, we may also assume that this limit inferior is a limit:
\begin{equation*}
\liminf_n \int_{C_n} w \,d \calH^1 = \lim_n \int_{C_n} w \,d\calH^1 \,.
\end{equation*}
We also notice that there is no restriction to assume $\rmdiam D > 0$, for if $\rmdiam D = 0$ then $\calH^1(C) \leq \calH^1_Z(D) = 0$, because $\rmLip j \leq 1$, and (E) is trivially verified.
\par 
With each $n=1,2,\ldots$ we associate a finite Borel measure $\mu_n$ on $Z$ by the formula
\begin{equation*}
\mu_n(B) = \int_{B \cap D_n} w(j_n(z)) \,d\calH^1_Z(z) \,,
\end{equation*}
$B \subset Z$ Borel. Since $j_n$ is an isometry we observe that
\begin{equation}
\label{eq.2}
\mu_n(Z) = \int_{D_n} w(j_n(z))\,d\calH^1_Z(z) = \int_{C_n} w(x)\, d\calH^1(x) \,.
\end{equation}
Thus $\{\mu_n\}$ is bounded in $C(Z)^*$ and it follows from the Banach-Alaoglu and Riesz-Markov Theorems that some subsequence, still denoted $\{\mu_n\}$, converges weakly* in $C(Z)^*$ to a finite Borel measure $\mu$. We establish now that
\begin{equation}
\label{eq.1}
\widetilde{\Theta}^1(\mu,z) \geq w(j(z))
\end{equation}
for every $z \in D$. 
\par 
Fix $z \in D$, $0 < r < \rmdiam D$ and $\veps > 0$. Choose $z_n \in D_n$, $n=1,2,\ldots$, so that $d_Z(z_n,z) \to 0$. Choose next $z'_n \in D_n \setminus B(z,r)$ and a curve $\Gamma_n \subset D_n$ with endpoints $z_n$ and $z'_n$, according to \ref{prelim.curves}. 
If $n$ is sufficiently large then $B(z_n,r/3) \subset B(z,r)$, and arguing as in Lemma \ref{Ahlfors} we infer the existence of $\tilde{z}_n \in \Gamma_n \cap \rmbdry B(z_n,r/3)$. We let $\Gamma_n^i$, $i=1,2$, denote the two components of $\Gamma_n \setminus \{\tilde{z}_n\}$. Upon noticing that $\rmdiam \Gamma_n^i \geq r/3$, $i=1,2$, we infer from Lemma \ref{Ahlfors} that $\calH^1_Z(\Gamma_n^i \cap B(\tilde{z}_n,r/3)) \geq r/3$, $i=1,2$, and therefore
\begin{equation} 
\label{eq.3}
\calH^1_Z(D_n \cap B(\tilde{z}_n,r/3)) \geq \frac{2r}{3} \,.
\end{equation}
Considering a subsequence if necessary we may assume that $d_Z(\tilde{z}_n,\tilde{z}) \to 0$ for some $\tilde{z} \in Z$. Now we abbreviate $\rho = r/3 + 2\veps$ where $\veps > 0$ is chosen sufficiently small for $(1-\veps) \rho \leq r/3$, and we further consider only integers $n$ so large that $d_Z(z_n,z) < \veps$ and $d_Z(\tilde{z}_n,\tilde{z}) < \veps$. One then readily checks that $z \in B(\tilde{z},\rho)$ and that $B(\tilde{z}_n,r/3) \subset B(\tilde{z},\rho)$. It follows from the latter and \eqref{eq.3} that
\begin{multline}
\label{eq.4}
(1-\veps)(\rmdiam B(\tilde{z},\rho)) \leq 
(1-\veps) 2 \rho \leq \frac{2r}{3} \leq \calH^1_Z(D_n \cap B(\tilde{z}_n,r/3))\\ \leq \calH^1_Z(D_n \cap B(\tilde{z},\rho)) \,.
\end{multline}
We now make the {\em additional assumptions that $w$ be Lipschitz} (with respect to the norm $\|\cdot\|$ of $B_X$) and we observe that for every $\zeta \in B(\tilde{z},\rho)$ one has
\begin{multline*}
w(j_n(\zeta))\geq w(j_n(z_n)) -(\rmLip w) \|j_n(\zeta)-j_n(z_n)\| \\ = w(j_n(z_n)) -(\rmLip w) d_Z(\zeta,z_n)
\geq w(j_n(z_n)) -2 (\rmLip w) \rho \,.
\end{multline*}
It follows from \eqref{eq.4} and the above that
\begin{equation*}
\begin{split}
\mu_n(B(\tilde{z},\rho)) & = \int_{B(\tilde{z},\rho)} w(j_n(\zeta))\, d\calH^1_Z(\zeta) \\
& \geq \left( \inf_{\zeta \in B(\tilde{z},\rho)} w(j_n(\zeta)) \right) \calH^1_Z(D_n \cap B(\tilde{z},\rho)) \\
& \geq \big(w(j_n(z_n)) -2 (\rmLip w) \rho \big) (1-\veps) (\rmdiam B(\tilde{z},\rho)) \,,
\end{split}
\end{equation*}
for $n$ sufficiently large. Letting $n \to \infty$ in the above and referring to Portmanteau's Theorem, the weak* lower semicontinuity of $w$, and $\lim_n d^*(j_n(z_n),j(z)) = 0$, we infer that
\begin{equation*}
\mu(B(\tilde{z},\rho)) \geq \limsup_n \mu_n(B(\tilde{z},\rho)) \geq \big(w(j(z)) -2 (\rmLip w) \rho \big) (1-\veps) (\rmdiam B(\tilde{z},\rho)) \,.
\end{equation*}
Letting $\veps \to 0$ and $\rho \leq r \to 0$ we obtain \eqref{eq.1}.
\par 
Since \eqref{eq.1} holds for every $z \in D$ we infer from \ref{26} that for every $0 < t < 1$ and every $k \in \Z$,
\begin{equation*}
\mu(D^k) \geq t^k \calH^1_Z(D^k) \geq t \int_{D^k} \widetilde{\Theta}^1(\mu,z)\, d\calH^1_Z(z) 
\end{equation*}
where
\begin{equation*}
D^k = D \cap \{ z : t^{k-1} \geq \widetilde{\Theta}^1(\mu,z) > t^k \} \,.
\end{equation*}
Since we have not discussed the measurability of $\widetilde{\Theta}^1(\mu,\cdot)$ we refer to \cite[2.4.10 and 2.4.3(2)]{GMT} for the next estimate. Summing over $k \in \Z$ and letting $t \to 1^-$ yields
\begin{equation*}
\mu(Z) \geq \mu(D) \geq \int^*_D \widetilde{\Theta}^1(\mu,z)\, d\calH^1_Z(z) \geq \int_D w(j(z))\,  d\calH^1_Z(z) \,.
\end{equation*} 
Next we infer from the surjectivity of $j$ and the inequality $\rmLip j \leq 1$ that $\calH^1 \hel C \leq j_*(\calH^1_Z \hel D)$. Thus
\begin{equation*}
\int_C w \, d \calH^1 \leq \int_C w \,d \big[ j_* (\calH^1_Z \hel D) \big] = \int_D (w \circ j) \, d\calH^1_Z \,.
\end{equation*}
It then follows from \eqref{eq.2} that
\begin{multline*}
\int_C w \, d\calH^1  \leq \int_D (w \circ j) \, d\calH^1_Z 
\leq \mu(Z)
 = \lim_n \mu_n(Z) \\
 = \lim_n \int_{D_n} (w \circ j_n)\, d\calH^1_Z 
 = \lim_n \int_{C_n} w \, d \calH^1 \,.
\end{multline*}
This completes the proof in case $w$ is Lipschitz. It thus remains only to remove that assumption. To this end we introduce the Yosida approximations $w_k$ of $w$, defined by the relation
\begin{equation*}
w_k(x) = \inf \left\{ w(y) + k \|y-x\| : y \in B_X \right\} \,,
\end{equation*}
$x \in B_X$, $k=1,2,\ldots$. We easily check that the sequence $\{w_k\}$ is nondecreasing and converges everywhere to $w$, and that each $w_k$ is both Lipschitz and weakly* lower semicontinuous. Therefore,
\begin{equation*}
\int_C w_k\, d\calH^1 \leq \liminf_n \int_{C_n} w_k \, d\calH^1 \leq \liminf_n \int_{C_n} w \,d\calH^1 \,,
\end{equation*}
for each $k=1,2,\ldots$, and the conclusion follows from the Monotone Convergence Theorem.
\end{proof}

We now consider a nonempty finite set $F \subset X$ and a weakly* lower semicontinuous function
\begin{equation*}
w : X \to (0,+\infty]
\end{equation*}
such that $\inf_X w > 0$. 
We let $\calC_F$ denote the collection of {\em connected compact} sets $C \subset X$ such that $F \subset C$. With each $C \in \calC_F$ we associate the {\em weighted length}
\begin{equation*}
\calL_w(C) = \int_C w\, d \calH^1 \,.
\end{equation*}
We consider the variational problem
\begin{equation*}
(\calP_{F,w}) \begin{cases}
\text{minimize } \calL_w(C) \\
\text{among } C \in \calC_F \,,
\end{cases}
\end{equation*}
assuming that $\inf (\calP_{F,w}) < \infty$. Note that this finiteness assumption holds for instance if $w$ is bounded on the convex hull $K$ of $F$.
Indeed if $F = \{x_0,x_1,\ldots,x_\kappa\}$ we let $C_0 = \cup_{k=1}^\kappa \blseg x_0,x_k \brseg$, so that $C_0 \in \calC_F$ and $\calL_w(C_0) \leq (\sup_K w)\sum_{k=1}^\kappa \|x_k-x_0\|$.

\begin{Theorem}[Existence]
\label{ExistenceTH}
Whenever $F$ and $w$ are as above, the variational problem $(\calP_{F,w})$ admits at least one solution.
\end{Theorem}

\begin{proof}
We apply the direct method of calculus of variations.
Define $\beta := \inf (\calP_{F,w})$ and let $\{C_n\}$ be a minimizing sequence such that $\calL_w(C_n) \leq 1 + \beta$, $n=1,2,\ldots$. Given $n$, let $x \in C_n$ be such that $\|x-x_0\| = \max_{y \in C_n} \|y-x_0\|$, where $x_0 \in F$. Let $\Gamma$ be a curve in $C_n$ with endpoints $x$ and $x_0$. It follows that
\begin{equation*}
1 + \beta \geq \calL_w(C_n) \geq \int_{\Gamma_n} w \,d\calH^1 \geq (\inf_X w) \calH^1(\Gamma_n) \geq (\inf_X w) \|x-x_0\| \,.
\end{equation*}
Therefore $C_n \subset B(x_0,R)$ where $R = (1+\beta)(\inf_X w)^{-1}$ and the conclusion follows from Theorem \ref{compactness} applied with $B(x_0,R)$.
\end{proof}


We end this section by showing that the minimizers of problem $(\calP_{F,w})$ are {\em almost minimizing} in a sense to be defined momentarily, and the remaining part of the paper will be devoted to studying the regularity properties of these (more general) almost minimizing sets.

\begin{Empty}[Almost minimizing sets]
\label{defalmmin} 
Given a gauge $\xi$, an open set $U \subset X$, and $r_0 > 0$, we say a compact connected set $C \subset X$ is {\em $(\xi,r_0)$ almost minimizing in $U$} provided $\calH^1(C) < \infty$ and the following holds: For every $x \in C \cap U$, every $0 < r \leq r_0$ such that $B(x,r) \subset U$, and every compact connected set $C' \subset X$ with 
\begin{equation*}
C' \setminus B(x,r) = C \setminus B(x,r)
\end{equation*}
one has
\begin{equation}
\calH^1(C\cap B(x,r))\leq (1+\xi(r))\calH^1(C'\cap B(x,r) ) \,. \label{almostmin}
\end{equation}
A set $C'$ as above is called a {\em competitor} for $C$ in the ball $B(x,r)$.
\end{Empty}

Given an open set $U \subset X$, a function $w : U \to \R$, and $r > 0$, we recall that the {\em oscillation} of $w$ at scale $r > 0$ is defined by
\begin{equation*}
\rmosc(w,r) = \sup \{ | w(x_1) - w(x_2) | : x_1,x_2 \in U \text{ and } \|x_1-x_2\| \leq r \} \,.
\end{equation*}
Thus $\lim_{r \to 0^+} \rmosc(w,r)=0$ if and only if $w$ is uniformly continuous.

\begin{Theorem} 
\label{lenghtminimiz} 
Assume that $F \subset X$ is a nonempty finite set, that $w:X\to [a,b]$ (where $0 < a < b < \infty$) is uniformly continuous, and that the variational problem $(\calP_{F,w})$ admits a minimizer $C$. It follows that $C$ is $(\xi,\infty)$ almost minimizing in $X \setminus F$, relative to the gauge
$$\xi(r)=\rmosc(w,r)\left(\frac{a+b}{a^2}\right).$$
\end{Theorem} 

\begin{proof} 
Notice that $\xi$ is indeed a gauge since $w$ is both bounded and uniformly continuous. Define $U = X \setminus F$, and fix $x$ and $r$ such that $x \in C$ and $B(x,r) \subset U$. We abbreviate $B = B(x,r)$ and we observe that for each competitor $C'$ in $B$ one has
\begin{eqnarray*}
a \calH^1(C\cap B)\leq \int_{C\cap B} w(x) \,d\calH^1(x) \leq \int_{C'\cap B} w(x)\, d\calH^1(x) \leq  b\calH^1(C'\cap B), \label{comparaison1}
\end{eqnarray*}
as well as
\begin{equation*}
\begin{split}
(w(x_0)-\rmosc(w,r))\calH^1(C\cap B)  &\leq \int_{C\cap B} w(x) \,d\calH^1(x) \\
&\leq  \int_{C'\cap B} w(x) \,d\calH^1(x) \\
&\leq  (w(x_0)+\rmosc(w,r))\calH^1(C'\cap B) .\label{comparaison2}
\end{split}
\end{equation*}
Therefore,
\begin{equation*}
\begin{split}
w(x_0)\calH^1(C\cap B) &\leq (w(x_0)+\rmosc(w,r))\calH^1(C'\cap B) + \rmosc(w,r)\calH^1(C\cap B)  \\
&\leq  (w(x_0)+\rmosc(w,r))\calH^1(C'\cap B) + \rmosc(w,r) \frac{b}{a}\calH^1(C'\cap B) \,,
\end{split}
\end{equation*}
and the conclusion follows upon dividing by $w(x_0)\geq a >0$.
\end{proof}

Since we are considering 1 dimensional geometric variational problems, it is worth pointing out the easy local topological regularity of minimizers.

\begin{Theorem}
\label{toporeg}
Assume that $\rmcard F = 2$ and that $w : X \to \R^+ \setminus \{0\}$ is uniformly continuous. It follows that every minimizer of problem $(\calP_{F,w})$ is a curve $\Gamma$ with endpoints those of $F$, and that $\Theta^1(\calH^1 \hel \Gamma,x) = 1$ for each $x \in \mathring{\Gamma}$.
\end{Theorem}

\begin{proof}
If $C$ is a minimizer then it contains a curve $\Gamma$ with endpoints those of $F$, according to \ref{prelim.curves}. It follows that $\calL_w(C \setminus \Gamma) = 0$, and in turn $\calH^1(C \setminus \Gamma) = 0$. From this we infer that in fact $C \setminus \Gamma = \emptyset$, for if $x \in C \setminus \Gamma$ and $r > 0$ is so that $B(x,r) \cap \Gamma = \emptyset$ then $0 < \calH^1(C \cap B(x,r)) \leq \calH^1(C \setminus \Gamma)$, according to \ref{Ahlfors}, a contradiction.
\par 
Let $x_0 \in \mathring{\Gamma}$ and $r > 0$ so that $B(x_0,r) \cap F = \emptyset$. We choose an arclength parametrization $\gamma : [a,b] \to X$ of $\Gamma$, and $a < t_0 < b$ such that $x_0=\gamma(t_0)$. Define
\begin{gather*}
t_-:= \inf \{ t \leq t_0 : \gamma(t) \in \rmbdry B(x_0,r) \},\\
t_+:= \sup \{ t \geq t_0 : \gamma(t) \in \rmbdry B(x_0,r) \} .
\end{gather*}
We create a competitor for the problem $(\calP_{F,w})$ as follows:
\begin{equation*}
C' = \gamma([a,t_-]) \cup \blseg \gamma(t_-) , x_0 \brseg \cup \blseg x_0,\gamma(t_+) \brseg \cup \gamma([t_+,b]) \,.
\end{equation*}
From the relation $\calL_w(C) \leq \calL_w(C')$ we obtain
\begin{equation*}
\int_{\gamma([t_{-},t_{+}])} w \,d \calH ^1
\leq \int_{\blseg\gamma(t_-),x_0 \brseg \cup \blseg x_0,\gamma(t_+) \brseg } w \,d \calH ^1 \,.
\end{equation*}
It entails from the definition of $t_+$ and $t_-$ that $\Gamma \cap B(x_0,r) \subset \gamma([t_-,t_+]))$, thus in fact
\begin{equation*}
\int_{\Gamma \cap B(x_0,r)} w \,d \calH ^1
\leq \int_{\blseg\gamma(t_-),x_0 \brseg \cup \blseg x_0,\gamma(t_+) \brseg } w \,d \calH ^1 \,.
\end{equation*}
We next infer from the uniform continuity of $w$ that 
\begin{equation*}
(w(x_0) - \rmosc(w;r)) \calH^1(\Gamma \cap B(x_0,r)) \leq \int_{\Gamma \cap B(x_0,r)} w \,d\calH^1 \,,
\end{equation*}
as well as
\begin{equation*}
\int_{\blseg\gamma(t_-),x_0 \brseg \cup \blseg x_0,\gamma(t_+) \brseg } w \,d \calH ^1 \leq (w(x_0) + \rmosc(w;r)) 2r \,.
\end{equation*}
Therefore, if $r>0$ is sufficiently small then,
\begin{equation*}
\frac{\calH^1(\Gamma \cap B(x_0,r))}{2r} \leq \frac{ w(x_0) + \rmosc(w;r)}{w(x_0) - \rmosc(w;r)} \,.
\end{equation*}
Letting $r \to 0^+$ we infer that $\Theta^{*\,1}(\calH^1 \hel \Gamma , x_0) \leq 1$. The reverse inequality $\Theta^1_*(\calH^1 \hel \Gamma,x_0) \geq 1$ is proved in \ref{density.first}(C). 
\end{proof}

\section{Almost minimizing sets in arbitrary Banach spaces}
\label{sec.4}

We establish the basic discrepancy between regular and singular points of almost minimizing sets. 

\begin{Empty}[Local hypothesis about the ambient Banach space]
In this section $X$ denotes a separable Banach space.
\end{Empty}

\begin{Scholium}
We will repeatedly use (without mention) the following observation. If $B \subset X$ is a closed ball of radius $r>0$ and $\Gamma$ is a curve in $X$ with endpoints $a$ and $b$ so that $a \not\in B$ and $b \in B$, then $\Gamma \cap \rmbdry B \neq \emptyset$. This is because if $\gamma : [0,1] \to X$ parametrizes $\Gamma$ so that $f(0)=a$ and $f(1)=b$, and if $x$ is the center of the ball $B$, then $f(t) = \| \gamma(t) - x \|$ is continuous and $f(1) \leq r < f(0)$. In fact, there is the smallest parameter $t^*$ such that $\gamma(t^*) \in \rmbdry B$. Thus the subcurve $\Gamma'$ of $\Gamma$ with endpoints $a$ and $\gamma(t^*)$ is so that $\mathring{\Gamma'} \cap B = \emptyset$.
\end{Scholium}

\begin{Proposition}
\label{prop.comp}
Assume that:
\begin{enumerate}
\item[(A)] $C \subset X$ is compact and connected, $\calH^1(C) < \infty$, $x \in C$, $r  > 0$, and $B = B(x,r)$;
\item[(B)] $N \in \N \setminus\{0\}$, $\rmcard C \cap \rmbdry B = N$, and
\begin{equation*}
C \cap \rmbdry B = \{ x_1,\ldots,x_N\} \,.
\end{equation*}
\end{enumerate}
If $x_0 \in B$ then
\begin{equation*}
C' = (C \setminus B) \cup \left( \cup_{n=1}^N \blseg x_0,x_n \brseg \right)
\end{equation*}
is a competitor for $C$ in $B$. In particular, if $N=1$ then
\begin{equation*}
C' = C \setminus \rmint B
\end{equation*}
is a competitor for $C$ in $B$.
\end{Proposition}

\begin{proof}
Since $C'$ is the union of $C \setminus \rmint B$ and finitely line segments, it is compact. We now show that any pair of $a,b \in C'$ is connected by a curve contained in $C'$. If $a$ and $b$ both belong to $B$, they are related connected by a curve in $C' \cap B$. We now assume one or both of $a$ and $b$ does not belong to $B$. Recalling \ref{prelim.curves}, we select a curve $\Gamma \subset C$ with endpoints $a$ and $b$. If $\mathring{\Gamma} \cap B = \emptyset$ we are done. Assume $a \not \in B$ and choose $n \in \{1,\ldots,N\}$ such that $x_n$ is closest to $a$ along $\Gamma$, and denote $\Gamma_{a,x_n}$ the corresponding subcurve of $\Gamma$, so that $\Gamma_{a,x_n} \subset C'$. If $b \in B$ then $b$ can be joinded to $x_n$ in $C' \cap B$ by a curve $\Gamma_{b,x_n}$, and $\Gamma_{a,x_n} \cup \Gamma_{x_n,b} \subset C'$ is a curve with endpoints $a$ and $b$. If instead $b \not \in B$ then let $x_m$, $m \in \{1,\ldots,N\}$ be closest to $b$ along $\Gamma$, and denote $\Gamma_{b,x_m}$ the corresponding subcurve of $\Gamma$, so that $\Gamma_{b,x_m} \subset C'$. Finally, choose a curve $\Gamma_{x_n,x_m}$ contained in $C' \cap B$ with endpoints $x_n$ and $x_m$ and notice that $\Gamma_{a,x_n} \cup \Gamma_{x_n,x_m} \cup \Gamma_{x_m,b} \subset C'$ is a curve with endpoints $a$ and $b$.
\end{proof}

\begin{Theorem}
\label{local.topology}
Assume that:
\begin{enumerate}
\item[(A)] $C \subset X$ is compact and connected, $U \subset X$ is open, $0 < r < r_0$, $x \in C$, $B(x,r) \subset U$, $\xi$ is a gauge;
\item[(B)] $C$ is $(\xi,r_0)$ almost minimizing in $U$;
\item[(C)] $r < \min(\rmdiam C,r_0)$.
\end{enumerate}
The following hold:
\begin{enumerate}
\item[(D)] $\rmcard (C \cap \rmbdry B(x,r)) \geq 2$;
\item[(E)] If $\rmcard( C \cap \rmbdry B(x,r)) = 2$ then $C \cap B(x,r)$ contains a Lipschitz curve whose endpoints are $\{x_0,x_1\} = C \cap \rmbdry B(x,r)$;
\item[(F)] If $\rmcard( C \cap \rmbdry B(x,r)) = 2$ and $r$ is a point of $\calL^1$ approximate continuity of $\rho \mapsto \rmcard (C \cap \rmbdry B(x,\rho))$ then for every $0 < \veps < 1$ there exists $(1-\veps) r < \rho < r$ such that the following dichotomy holds:
\begin{enumerate}
\item[{\bf either}] $C \cap B(x,\rho)$ is a Lipschitz curve and $C \cap \rmbdry B(x,\rho)$ consists of its endpoints;
\item[{\bf or}] $C \cap B(x,r)$ contains three Lipschitz curves $\Gamma_1,\Gamma_2,\Gamma_3$ whose intersection is a singleton $\{\tilde{x}\} = \Gamma_1 \cap \Gamma_2 \cap \Gamma_3$, and $\tilde{x}\in \rmint B(x,r)$ is an endpoint of each $\Gamma_1, \Gamma_2$ and $\Gamma_3$.
\end{enumerate}
\end{enumerate}
\end{Theorem}

\begin{Remark}
Several comments are in order.
\begin{enumerate}
\item[(A)] The ``temporary'' conclusion (E) does not assert that $C \cap B(x,r)$ {\em is} a Lipschitz curve $\Gamma$, but merely that it contains such $\Gamma$ whose endpoints are on $\rmbdry B(x,r)$; in particular it is not claimed that $x \in \Gamma$.
\item[(B)] The first alternative of conclusion (F), however, states that $C \cap B(x,\rho)$ {\em is} a Lipschitz curve $\Gamma$, and that $C \cap \rmbdry B(x,\rho)$ consists of the endpoints of $\Gamma$, for $\rho$ close to $r$.
\item[(C)] The function $\rho \mapsto \rmcard(C \cap B(x,\rho))$ is $\calL^1$ measurable on $\R^+$, recall \ref{eilenberg}, and hence approximately continuous $\calL^1$ almost everywhere, see \cite[2.9.12 and 2.9.13]{GMT}.
\end{enumerate}
\end{Remark}

\begin{proof}[Proof of Theorem \ref{local.topology}]
We abbreviate $B = B(x,r)$ and we start by proving (D). Since $x \in C \cap B$ and $r < \rmdiam C$ we infer that $C \cap \rmbdry B$ is not empty. Assuming $C \cap \rmbdry B$ is a singleton, we infer from \ref{prop.comp} (applied with $x_0=x_1$) that $C' = C \setminus B$ is a competitor for $C$ in $B$. Now since $C$ is almost minimizing we would have $\calH^1(C \cap B) \leq (1 + \xi(r)) \calH^1(C' \cap B) = 0$, in contradiction with \ref{Ahlfors}.
\par 
We now turn to proving (E). Let $C \cap \rmbdry B = \{x_0,x_1\}$.
We will show that $C$ contains a Lipschitz curve $\Gamma$ with endpoints $x_0$ and $x_1$, whose interior $\mathring{\Gamma} = \Gamma \setminus \{x_0,x_1\}$ is contained in $\rmint B$. Recalling \ref{prelim.curves} we infer that there exists a Lipschitz curve $\Gamma \subset C$ with endpoints $x_0$ and $x_1$. If $\mathring{\Gamma} \cap \rmint B \neq \emptyset$ then $\mathring{\Gamma} \subset \rmint B$ for otherwise $\rmcard (C \cap \rmbdry B) \geq 3$, a contradiction. Thus (E)  will be established if we rule out the case $\mathring{\Gamma} \subset X \setminus B$. 

\begin{center}
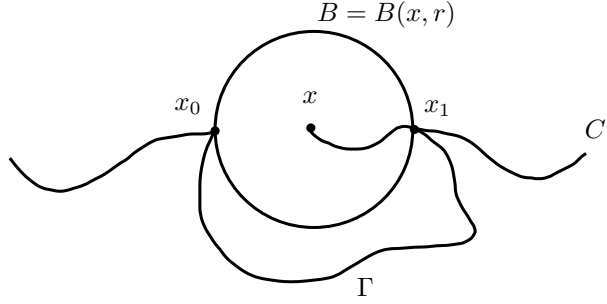
\begin{figure}[!h]
\scalebox{1} 
{
\begin{pspicture}(0,-2.0)(8.04,2.0)
\pscircle[linewidth=0.04,dimen=outer](4.0,0.28){1.32}
\psdots[dotsize=0.12](3.96,0.3)
\usefont{T1}{ptm}{m}{n}
\rput(4.96,1.805){$B=B(x,r)$}
\usefont{T1}{ptm}{m}{n}
\rput(3.95,0.705){$x$}
\usefont{T1}{ptm}{m}{n}
\rput(2.35,0.585){$x_0$}
\usefont{T1}{ptm}{m}{n}
\rput(5.63,0.565){$x_1$}
\psdots[dotsize=0.12](5.32,0.28)
\usefont{T1}{ptm}{m}{n}
\rput(7.7,0.285){$C$}
\usefont{T1}{ptm}{m}{n}
\rput(4.68,-1.815){$\Gamma$}
\psdots[dotsize=0.12](2.7,0.26)
\pscustom[linewidth=0.04]
{
\newpath
\moveto(0.0,-0.1)
\lineto(0.03,-0.14)
\curveto(0.045,-0.16)(0.09,-0.215)(0.12,-0.25)
\curveto(0.15,-0.285)(0.21,-0.345)(0.24,-0.37)
\curveto(0.27,-0.395)(0.33,-0.44)(0.36,-0.46)
\curveto(0.39,-0.48)(0.45,-0.505)(0.48,-0.51)
\curveto(0.51,-0.515)(0.56,-0.525)(0.58,-0.53)
\curveto(0.6,-0.535)(0.645,-0.54)(0.67,-0.54)
\curveto(0.695,-0.54)(0.75,-0.53)(0.78,-0.52)
\curveto(0.81,-0.51)(0.875,-0.49)(0.91,-0.48)
\curveto(0.945,-0.47)(1.005,-0.445)(1.03,-0.43)
\curveto(1.055,-0.415)(1.1,-0.39)(1.12,-0.38)
\curveto(1.14,-0.37)(1.18,-0.35)(1.2,-0.34)
\curveto(1.22,-0.33)(1.275,-0.295)(1.31,-0.27)
\curveto(1.345,-0.245)(1.4,-0.205)(1.42,-0.19)
\curveto(1.44,-0.175)(1.48,-0.15)(1.5,-0.14)
\curveto(1.52,-0.13)(1.56,-0.1)(1.58,-0.08)
\curveto(1.6,-0.06)(1.65,-0.02)(1.68,0.0)
\curveto(1.71,0.02)(1.78,0.065)(1.82,0.09)
\curveto(1.86,0.115)(1.925,0.145)(1.95,0.15)
\curveto(1.975,0.155)(2.035,0.165)(2.07,0.17)
\curveto(2.105,0.175)(2.17,0.185)(2.2,0.19)
\curveto(2.23,0.195)(2.285,0.205)(2.31,0.21)
\curveto(2.335,0.215)(2.39,0.22)(2.42,0.22)
\curveto(2.45,0.22)(2.505,0.22)(2.53,0.22)
\curveto(2.555,0.22)(2.6,0.225)(2.62,0.23)
\curveto(2.64,0.235)(2.665,0.24)(2.67,0.24)
\curveto(2.675,0.24)(2.67,0.215)(2.66,0.19)
\curveto(2.65,0.165)(2.625,0.11)(2.61,0.08)
\curveto(2.595,0.05)(2.57,-0.02)(2.56,-0.06)
\curveto(2.55,-0.1)(2.535,-0.165)(2.53,-0.19)
\curveto(2.525,-0.215)(2.515,-0.275)(2.51,-0.31)
\curveto(2.505,-0.345)(2.5,-0.42)(2.5,-0.46)
\curveto(2.5,-0.5)(2.5,-0.58)(2.5,-0.62)
\curveto(2.5,-0.66)(2.5,-0.735)(2.5,-0.77)
\curveto(2.5,-0.805)(2.515,-0.885)(2.53,-0.93)
\curveto(2.545,-0.975)(2.58,-1.05)(2.6,-1.08)
\curveto(2.62,-1.11)(2.655,-1.165)(2.67,-1.19)
\curveto(2.685,-1.215)(2.715,-1.255)(2.73,-1.27)
\curveto(2.745,-1.285)(2.78,-1.33)(2.8,-1.36)
\curveto(2.82,-1.39)(2.865,-1.44)(2.89,-1.46)
\curveto(2.915,-1.48)(2.965,-1.52)(2.99,-1.54)
\curveto(3.015,-1.56)(3.075,-1.6)(3.11,-1.62)
\curveto(3.145,-1.64)(3.215,-1.67)(3.25,-1.68)
\curveto(3.285,-1.69)(3.355,-1.705)(3.39,-1.71)
\curveto(3.425,-1.715)(3.495,-1.725)(3.53,-1.73)
\curveto(3.565,-1.735)(3.65,-1.74)(3.7,-1.74)
\curveto(3.75,-1.74)(3.845,-1.735)(3.89,-1.73)
\curveto(3.935,-1.725)(4.025,-1.715)(4.07,-1.71)
\curveto(4.115,-1.705)(4.195,-1.69)(4.23,-1.68)
\curveto(4.265,-1.67)(4.33,-1.64)(4.36,-1.62)
\curveto(4.39,-1.6)(4.455,-1.56)(4.49,-1.54)
\curveto(4.525,-1.52)(4.6,-1.475)(4.64,-1.45)
\curveto(4.68,-1.425)(4.75,-1.385)(4.78,-1.37)
\curveto(4.81,-1.355)(4.875,-1.33)(4.91,-1.32)
\curveto(4.945,-1.31)(5.015,-1.3)(5.05,-1.3)
\curveto(5.085,-1.3)(5.17,-1.295)(5.22,-1.29)
\curveto(5.27,-1.285)(5.36,-1.28)(5.4,-1.28)
\curveto(5.44,-1.28)(5.515,-1.275)(5.55,-1.27)
\curveto(5.585,-1.265)(5.65,-1.26)(5.68,-1.26)
\curveto(5.71,-1.26)(5.775,-1.255)(5.81,-1.25)
\curveto(5.845,-1.245)(5.91,-1.235)(5.94,-1.23)
\curveto(5.97,-1.225)(6.03,-1.195)(6.06,-1.17)
\curveto(6.09,-1.145)(6.12,-1.095)(6.12,-1.07)
\curveto(6.12,-1.045)(6.105,-1.0)(6.09,-0.98)
\curveto(6.075,-0.96)(6.05,-0.92)(6.04,-0.9)
\curveto(6.03,-0.88)(6.005,-0.84)(5.99,-0.82)
\curveto(5.975,-0.8)(5.945,-0.75)(5.93,-0.72)
\curveto(5.915,-0.69)(5.89,-0.635)(5.88,-0.61)
\curveto(5.87,-0.585)(5.86,-0.535)(5.86,-0.51)
\curveto(5.86,-0.485)(5.855,-0.43)(5.85,-0.4)
\curveto(5.845,-0.37)(5.835,-0.31)(5.83,-0.28)
\curveto(5.825,-0.25)(5.81,-0.195)(5.8,-0.17)
\curveto(5.79,-0.145)(5.77,-0.1)(5.76,-0.08)
\curveto(5.75,-0.06)(5.73,-0.025)(5.72,-0.01)
\curveto(5.71,0.0050)(5.68,0.035)(5.66,0.05)
\curveto(5.64,0.065)(5.6,0.095)(5.58,0.11)
\curveto(5.56,0.125)(5.525,0.155)(5.51,0.17)
\curveto(5.495,0.185)(5.465,0.21)(5.45,0.22)
\curveto(5.435,0.23)(5.4,0.25)(5.38,0.26)
\curveto(5.36,0.27)(5.32,0.29)(5.3,0.3)
\curveto(5.28,0.31)(5.225,0.32)(5.19,0.32)
\curveto(5.155,0.32)(5.1,0.305)(5.08,0.29)
\curveto(5.06,0.275)(5.015,0.24)(4.99,0.22)
\curveto(4.965,0.2)(4.925,0.165)(4.91,0.15)
\curveto(4.895,0.135)(4.86,0.11)(4.84,0.1)
\curveto(4.82,0.09)(4.78,0.07)(4.76,0.06)
\curveto(4.74,0.05)(4.695,0.035)(4.67,0.03)
\curveto(4.645,0.025)(4.595,0.02)(4.57,0.02)
\curveto(4.545,0.02)(4.495,0.02)(4.47,0.02)
\curveto(4.445,0.02)(4.4,0.025)(4.38,0.03)
\curveto(4.36,0.035)(4.32,0.045)(4.3,0.05)
\curveto(4.28,0.055)(4.245,0.07)(4.23,0.08)
\curveto(4.215,0.09)(4.18,0.105)(4.16,0.11)
\curveto(4.14,0.115)(4.105,0.13)(4.09,0.14)
\curveto(4.075,0.15)(4.04,0.175)(4.02,0.19)
\curveto(4.0,0.205)(3.97,0.235)(3.96,0.25)
\curveto(3.95,0.265)(3.935,0.285)(3.92,0.3)
}
\pscustom[linewidth=0.04]
{
\newpath
\moveto(5.3,0.3)
\lineto(5.35,0.29)
\curveto(5.375,0.285)(5.425,0.275)(5.45,0.27)
\curveto(5.475,0.265)(5.525,0.26)(5.55,0.26)
\curveto(5.575,0.26)(5.62,0.255)(5.64,0.25)
\curveto(5.66,0.245)(5.7,0.235)(5.72,0.23)
\curveto(5.74,0.225)(5.79,0.215)(5.82,0.21)
\curveto(5.85,0.205)(5.905,0.19)(5.93,0.18)
\curveto(5.955,0.17)(6.005,0.145)(6.03,0.13)
\curveto(6.055,0.115)(6.105,0.08)(6.13,0.06)
\curveto(6.155,0.04)(6.2,0.0)(6.22,-0.02)
\curveto(6.24,-0.04)(6.285,-0.08)(6.31,-0.1)
\curveto(6.335,-0.12)(6.385,-0.16)(6.41,-0.18)
\curveto(6.435,-0.2)(6.51,-0.245)(6.56,-0.27)
\curveto(6.61,-0.295)(6.685,-0.33)(6.71,-0.34)
\curveto(6.735,-0.35)(6.785,-0.36)(6.81,-0.36)
\curveto(6.835,-0.36)(6.88,-0.365)(6.9,-0.37)
\curveto(6.92,-0.375)(6.975,-0.375)(7.01,-0.37)
\curveto(7.045,-0.365)(7.105,-0.345)(7.13,-0.33)
\curveto(7.155,-0.315)(7.2,-0.29)(7.22,-0.28)
\curveto(7.24,-0.27)(7.28,-0.25)(7.3,-0.24)
\curveto(7.32,-0.23)(7.36,-0.21)(7.38,-0.2)
\curveto(7.4,-0.19)(7.435,-0.165)(7.45,-0.15)
\curveto(7.465,-0.135)(7.495,-0.105)(7.51,-0.09)
\curveto(7.525,-0.075)(7.55,-0.055)(7.58,-0.04)
}
\end{pspicture} 
}
\caption{A case to rule out in proving (E).}
\end{figure}
\end{center}

Assuming if possible that $\mathring{\Gamma} \subset X \setminus B$ we verify that
\begin{equation*}
C' = (C \setminus B) \cup \{ x_0,x_1 \}
\end{equation*}
is a competitor. It is indeed easy to check that $C'$ is compact and we now show that it is connected. Given $a,b \in C'$ we will find a curve $\Gamma' \subset C'$ with endpoints $a$ and $b$.
According to \ref{prelim.curves}, there exists a curve $\Gamma'' \subset C$ with endpoints $a$ and $b$. If $\Gamma'' \cap \rmint B = \emptyset$ we let $\Gamma' = \Gamma''$ and we are done. Otherwise $\Gamma''$ contains one of $x_0$ and $x_1$, and hence also both. We denote by $\Gamma''_0$ the subcurve of $\Gamma''$ with endpoints $a$ and (say) $x_0$, by $\Gamma''_1$ the subcurve of $\Gamma''$ with endpoints $x_1$ and $b$, and we put $\Gamma''_B = \Gamma'' \cap B$. Thus $\Gamma'' = \Gamma''_0 \cup \Gamma''_B \cup \Gamma''_1$ and we define a new curve $\Gamma' \subset C'$ corresponding to $\Gamma''_0 \cup \Gamma \cup \Gamma''_1$. This completes the proof that $C'$ is a competitor. Now since $\calH^1(C' \cap B) = 0$ and $C$ is almost minimizing, we infer that $\calH^1(C \cap B) = 0$. Together with hypothesis (C), this contradicts Lemma \ref{Ahlfors}. Thus conclusion (E)  is established.
\par 
It remains to prove (F). Let $B_r = B(x,r)$, $\{x_0,x_1\} = C \cap \rmbdry B(x,r)$, and let $\Gamma \subset C$ denote a Lipschitz curve with endpoints $x_0$ and $x_1$, and $\mathring{\Gamma} \subset \rmint B_r$, whose existence results from conclusion (E). The $\calL^1$ approximate continuity of $\rho \mapsto \rmcard(C \cap \rmbdry B(x,\rho))$ at $\rho=r$ implies the existence of an increasing sequence $\{ \rho_k \}$ with limit $r$ and such that $\rmcard (C \cap \rmbdry B(x,\rho_k)) =2$ for every $k$. Choose $\rho = \rho_k$ with $k$ large enough for $(1-\veps)r < \rho < r$. Taking $k$ even larger we may assume that $\rmcard(\Gamma \cap B(x,\rho))\geq 2$ since $\mathring{\Gamma} \subset \rmint B_r$ and $x_0, x_1\in \rmbdry B_r$ are the endpoints of $\Gamma$. The curve $\Gamma$ being a subset of $C$ we must have $\rmcard(\Gamma \cap B(x,\rho))= 2$.
Abbreviate $B_\rho = B(x,\rho)$. If $C \cap B_\rho = \Gamma \cap B_\rho$ then the first branch of the dichotomy occurs and the proof is finished. Otherwise let $\{y_0,y_1\} = C \cap \rmbdry B_\rho= \Gamma \cap \rmbdry B_\rho$. 
Choose $y \in (C \cap B_\rho) \setminus \Gamma$. Thus $C$ contains a curve $\Gamma'$ with endpoints $y$ and $y_0$. Let $y'$ be the first point on $\Gamma'$ (starting from $y$) that belongs to $\Gamma$. Clearly $y' \in B_\rho$. Whether $y' \in \rmint B_\rho$ or $y' \in \{y_0,y_1\}$, one checks that $C$ contains three nontrivial Lipschitz curves $\Gamma_1,\Gamma_2,\Gamma_3$ whose intersection is $\{y'\}$, two of which are subcurves of $\Gamma$, the other one being a subcurve of $\Gamma'$.
\end{proof}

We now state the basic discrepancy regarding the density of points of almost minimizing sets.

\begin{Theorem}
\label{density.first}
Assume that
\begin{enumerate}
\item[(A)] $C \subset X$ is compact and connected, $U \subset X$ is open, $\xi$ is a gauge, $r_0 > 0$, $x \in C \cap U$;
\item[(B)] $C$ is $(\xi,r_0)$ almost minimizing in $U$.
\end{enumerate}
The following hold.
\begin{enumerate}
\item[(C)] $\Theta^1_*(\calH^1 \hel C ,x) \geq 1$;
\item[(D)] One of the following occurs: {\bf either}
\begin{equation*}
\rmap\lim_{r \to 0^+} \frac{\calH^1(C \cap B(x,r))}{2r} = 1 \,,
\end{equation*}
{\bf or}
\begin{equation*}
\Theta^1_*(\calH^1 \hel C,x) \geq 3/2 \,.
\end{equation*}
\end{enumerate}
\end{Theorem}

\begin{proof}
Given $r \geq \rho > 0$ sufficiently small, it follows from Theorem \ref{local.topology} (D) that $\rmcard (C \cap \rmbdry B(x,\rho)) \geq 2$. Thus
\begin{equation*}
2 r \leq \int_0^r \rmcard (C \cap \rmbdry B(x,\rho)) d\calL^1(\rho) \leq \calH^1( C \cap B(x,r))
\end{equation*}
according to Eilenberg's inequality, recall \ref{eilenberg}. Conclusion (C) readily follows.
\par 
In view of (C), conclusion (D) will be established as soon as we show that the alternative holds with the first condition replaced by the formally weaker
\begin{equation}
\label{eq.14-bis}
\rmap\limsup_{r \to 0^+} \frac{\calH^1(C \cap B(x,r))}{2r} \leq 1 \,.
\end{equation}
We define an $\calL^1$ measurable set $G = \R \cap \{ r  > 0 : \rmcard (C \cap \rmbdry B(x,r)) = 2 \}$, and $\vartheta(r) = r^{-1} \calL^1( G \cap [0,r])$. We choose $0 < r'_0 \leq r_0$ small enough for $B(x,r'_0) \subset U$ and $\xi(r_0') < 1/4$. We claim that if $r \in G \cap [0,r'_0]$ then
\begin{equation}
\label{eq.12}
\frac{\calH^1(C \cap B(x,r))}{2r} \leq 1 + \xi(r)
\end{equation}
and
\begin{equation}
\label{eq.13}
\vartheta(r) \geq 1 - 2 \xi(r) \,.
\end{equation}
In order to prove \eqref{eq.12} we recall that our assumption $\rmcard(C \cap \rmbdry B(x,r)) = 2$ implies $C' = (C \setminus B(x,r)) \cup ( \blseg x_0,x \brseg] \cup \blseg x,x_1 \brseg)$ (where $\{x_0,x_1\} = C \cap \rmbdry B(x,r)$) is a competitor, according to \ref{prop.comp}. The desired inequality thus ensues from the almost minimizing property of $C$. In order to establish \eqref{eq.13} we refer to Theorem \ref{local.topology}(D), to Eilenberg's inequality \ref{eilenberg}, and to \eqref{eq.12}:
\begin{multline*}
2 \calL^1( G \cap [0,r] ) + 3 \calL^1( [0,r] \setminus G) \\ \leq \int_0^r \rmcard ( C \cap \rmbdry B(x,\rho)) d\calL^1(\rho) \\ \leq \calH^1(C \cap B(x,r)) \leq (1 + \xi(r)) 2 r \,.
\end{multline*}
In other words,
\begin{equation*}
2 r \vartheta(r) + 3r (1 - \vartheta(r)) \leq (1 + \xi(r)) 2 r \,,
\end{equation*}
from which \eqref{eq.13} readily follows. Still assuming that $r \in G \cap [0,r'_0]$, the bound $\xi(r) < 1/4$ together with \eqref{eq.13} shows there exists another $\hat{r} \in G$ with $r/4 \leq \hat{r} \leq 3 r / 4$. Iterating this observation, we infer from the hypothesis $G \cap [0,r'_0] \neq \emptyset$ the existence of a sequence $\{ r_k \}$ in $G\cap[0,r'_0]$ such that $\lim_k r_k = 0$ and $1\leq r_k / r_{k+1} \leq 4$. Now if $r_{k+1} \leq r \leq r_{k}$ then
\begin{equation*}
\frac{\calL^1( [0,r] \setminus G)}{r} \leq 
\frac{\calL^1( [0,r_k] \setminus G)}{r_{k+1}} \leq 
\frac{ 4 \calL^1( [0,r_k] \setminus G)}{r_k} \leq
8 \xi(r_k) \,,
\end{equation*}
according to \eqref{eq.13}. Therefore
\begin{equation*}
\lim_{r \to 0^+} \frac{\calL^1( [0,r] \cap G)}{r} = 1 \,.
\end{equation*}
Furthermore it follows from \eqref{eq.12} that
\begin{equation*}
\limsup_{\substack{r \to 0^+ \\ r \in G}} \frac{\calH^1(C \cap B(x,r))}{2r} \leq 1 \,.
\end{equation*}
It is now clear that if $G \cap [0,r'_0] \neq \emptyset$ then \eqref{eq.14-bis} holds. If instead $G \cap [0,r'_0] = \emptyset$ then Eilenberg's inequality implies that for each $0 < r \leq r'_0$ one has
\begin{equation*}
3 r \leq \int_0^r \rmcard(C \cap B(x,\rho)) d\calL^1(\rho) \leq \calH^1(C \cap B(x,r)) \,.
\end{equation*}
In particular $\Theta^1_*(\calH^1 \hel C,x) \geq 3/2$.
\end{proof}

In the remaining part of this section we will obtain better information under the assumption that the gauge $\xi$ is Dini.

\begin{Theorem}[Almost monotonicity]
\label{monotonicity}
Assume that:
\begin{enumerate}
\item[(A)] $C \subset X$ is compact and connected, $U \subset X$ is open, $r_0 > 0$;
\item[(B)] $\xi$ is a Dini gauge with mean slope $\zeta$;
\item[(C)] $C$ is $(\xi,r_0)$ almost minimizing in $U$;
\end{enumerate}
It follows that for every $x \in C \cap U$ the function
\begin{equation*}
(0, \min\{r_0,\rmdist(x,\rmbdry U)\}) \to \R^+ : r \mapsto \exp[\zeta(r)]\frac{\calH^1(C \cap B(x,r))}{2r}
\end{equation*}
is nondecreasing.
\end{Theorem}

\begin{proof}
Fix $x \in C \cap U$ and let $r(x) := \rmdist(x , \rmbdry U)$. We define $N(\rho) = \rmcard(C \cap \rmbdry B(x,\rho)) \in \N \cup \{ \infty \}$ for $0 < \rho < r(x)$, and
\begin{equation*}
\vphi(r) = \calH^1(C \cap B(x,r)) \,,
\end{equation*}
for $0 < r < r(x)$. Notice that $\vphi>0$ according to \ref{Ahlfors}.

We infer from Eilenberg's inequality \ref{eilenberg} (applied to $A = C \cap (B(x,b) \setminus B(x,a))$ that
\begin{equation*}
\int_a^b N(\rho) d\calL^1(\rho) \leq \vphi(b) - \vphi(a)
\end{equation*}
for every $0 < a < b < r(x)$. In particular $N$ is almost everywhere finite and 
\begin{equation}
\label{eqq.2}
N(\rho) \leq \vphi'(\rho)
\end{equation}
at those $\rho$ which are Lebesgue points of $N$ and at which $\vphi$ is differentiable. Since $\vphi$ is nondecreasing this occurs almost everywhere. We next select $0 < \rho < r(x)$ such that $N(\rho) < \infty$ and we define
\begin{equation*}
C'=(C\setminus B(x,\rho)) \cup \left( \bigcup_{y\in C\cap \rmbdry B(x,\rho)}\blseg x,y \brseg \right) \,.
\end{equation*}
It follows from \ref{prop.comp} that $C'$ is a competitor for $C$ in $B(x,\rho)$. Assuming also that \eqref{eqq.2} holds, the almost minimizing property of $C$ yields
\begin{multline*}
\vphi(\rho) = \calH^1(C \cap \rmbdry B(x,\rho)) \leq (1 + \xi(\rho)) \calH^1(C' \cap B(x,\rho)) \\ = (1 + \xi(\rho)) N(\rho) \rho \leq (1 + \xi(\rho)) \vphi'(\rho) \rho \,.
\end{multline*}
It follows from the above that
\begin{equation*}
\frac{d}{d\rho} \log \vphi(\rho) = \frac{\vphi'(\rho)}{\vphi(\rho)} \geq \frac{1}{(1 + \xi(\rho))\rho} \geq \frac{1 - \xi(\rho)}{\rho}= \frac{d}{d\rho} \log \bigg( \rho \exp [ - \zeta(\rho) ] \bigg) \,.
\end{equation*}
Since this inequality occurs almost everywhere and $\vphi$ is nondecreasing, we infer upon integrating each member that
\begin{equation*}
\log \vphi(r_2) - \log \vphi(r_1) \geq \log \bigg( r_2 \exp[ - \zeta(r_2) ] \bigg) - \log \bigg( r_1 \exp[ - \zeta(r_1) ] \bigg) \,.
\end{equation*}
for every $0 < r_1 < r_2 < r(x)$. Our conclusion now easily follows.
\end{proof}

\begin{Corollary}
\label{density.dichotomy}
At each $x \in C \cap U$ the density
\begin{equation*}
\Theta^1(\calH^1 \hel C , x) = \lim_{r \to 0^+} \frac{\calH^1(C \cap B(x,r))}{2r}
\end{equation*}
exists, and either $\Theta^1(\calH^1 \hel C , x) = 1$ or $\Theta^1(\calH^1 \hel C , x ) \geq 3/2$.
\end{Corollary}

\begin{Theorem}[Lipschitz regularity]
\label{lip.reg}
Assume that:
\begin{enumerate}
\item[(A)] $C \subset X$ is compact and connected, $U \subset X$ is open, $0 < r < r_0$, $x \in C$, $B(x,r) \subset U$, $0 < \tau \leq 1/6$;
\item[(B)] $\xi$ is a Dini gauge with mean slope $\zeta$;
\item[(C)] $C$ is $(\xi,r_0)$ almost minimizing in $U$;
\item[(D)] $\exp[\zeta(r)] \leq 1 + \tau/4$;
\item[(E)] $\xi(r) \leq \tau/4$;
\item[(F)] $\rmcard (C \cap \rmbdry B(x,r)) = 2$.
\end{enumerate}
It follows that there exists $\tau r /2 \leq \rho \leq \tau r$ and a Lipschitz curve $\Gamma$ such that $C \cap B(x,\rho) = \Gamma$ and $C \cap \rmbdry B(x,\rho)$ consists of the two endpoints of $\Gamma$.
\end{Theorem}

\begin{proof}
We start arguing as in the proof of Theorem \ref{density.first}. Letting $G$ and $\vartheta$ be defined as in that proof, we infer from our hypotheses (E) and (F) that $\vartheta(r) \geq 1 - \tau/2$, see \eqref{eq.13}, and we infer from our hypotheses (D), (E), (F), and the inequality $\tau \leq 1/6$ that
\begin{equation}
\label{eq.14}
\exp[\zeta(r)] \calH^1(C \cap B(x,r)) \leq (1 + \tau/4)^2 2r \leq (1 + \tau) 2 r \,,
\end{equation}
see \eqref{eq.12}. Thus there exists $\rho \in G \cap (\tau r /2 , \tau r)$ which is a point of $\calL^1$ approximate continuity of $\rho \mapsto \rmcard (C \cap \rmbdry B(x,\rho))$. It then follows from Theorem \ref{local.topology}(F) that our conclusion will be established provided we rule out the second alternative in that conclusion. Assume if possible that such $\tilde{x} \in C \cap B(x,\rho)$ exists. The Eilenberg inequality easily implies that $\Theta^1(\calH^1 \hel C , \tilde{x}) \geq 3/2$. The following contradiction ensues:
\begin{equation*}
\begin{split}
\frac{3}{2} & \leq \Theta^1(\calH^1 \hel C , \tilde{x} ) \\
& \leq \exp[\zeta((1-\tau)r)] \frac{ \calH^1(C \cap B( \tilde{x} , (1-\tau) r))}{ 2 (1-\tau)r } \\
& \leq \exp[\zeta((1-\tau)r)] \frac{ \calH^1(C \cap B( x , \|x-\tilde{x}\|+(1-\tau) r))}{ 2 (1-\tau)r } \\
& \leq \exp[\zeta((1-\tau)r)] \frac{ \calH^1(C \cap B( x , \|x-\tilde{x}\|+(1-\tau) r))}{ 2 (\|x-\tilde{x}\|+(1-\tau) r) } \left(\frac{ \|x-\tilde{x}\|+(1-\tau) r}{ (1-\tau)r} \right)\\
\intertext{which, according to $\|x-\tilde{x}\| \leq \rho \leq \tau r$ and Theorem \ref{monotonicity}, is bounded by}
& \leq \exp[\zeta(r)] \frac{\calH^1(C \cap B(x,r))}{2r} \left( 1 + \frac{\tau}{1-\tau} \right)\\
\intertext{which, according to \eqref{eq.14}, is bounded by}
& \leq (1 + \tau) \left( 1 + \frac{\tau}{1-\tau} \right) \\
& < \frac{3}{2}
\end{split}
\end{equation*}
since $\tau \leq 1/6$.
\end{proof}

\begin{Definition}
\label{def.reg}
Let $C \subset X$ and $x \in C$. We say that:
\begin{enumerate}
\item[(1)] $x$ is a {\em regular point} of $C$ if for each $\delta > 0$ there exists $0 < r < \delta$ such that $C \cap B(x,r)$ is a Lipschitz curve $\Gamma$ and $C \cap \rmbdry B(x,r)$ consists of the two endpoints of $\Gamma$;
\item[(2)] $x$ is a {\em singular point} of $C$ if it is not a regular point of $C$.
\end{enumerate}
The {\em set of regular points} of $C$ is denoted $\rmreg(C)$, and the {\em set of singular points} of $C$ is denoted $\rmsing(C) = C \setminus \rmreg(C)$.
\end{Definition}

\begin{Theorem}
\label{lip.reg.bis}
Assume that:
\begin{enumerate}
\item[(A)] $C \subset X$ is compact and connected, $U \subset X$ is open, $ r_0 > 0$;
\item[(B)] $\xi$ is a Dini gauge;
\item[(C)] $C$ is $(\xi,r_0)$ almost minimizing in $U$.
\end{enumerate}
It follows that $U \cap \rmreg(C) = U \cap \{ x : \Theta^1(\calH^1 \hel C , x) = 1 \}$, that $U \cap \rmsing(C)$ is relatively closed in $U \cap C$, and that $\calH^1(U \cap \rmsing(C)) = 0$.
\end{Theorem}

\begin{proof}
Let $x \in U$. We first show that if $\Theta^1(\calH^1 \hel C,x) = 1$ then $x$ is a regular point of $C$. Since $C$ is closed, we infer that $x \in C$. If $r' > 0$ is sufficiently small then $2r' \leq \calH^1(C \cap B(x,r')) < 3r'$; the first inequality follows as in the proof of Theorem \ref{density.first}(C), whereas the second results from our assumption. Therefore there exists $0 < r < r'$ such that $\rmcard(C \cap \rmbdry B(x,r)) = 2$, according to Eilenberg's inequality. One can of course assume that $r'$ is small enough for hypotheses (D) and (E) of Theorem \ref{lip.reg} to be verified as well. It then follows from that Theorem that $C \cap B(x,r'')$ is indeed a Lipschitz curve, for some $0 < r'' < r'$. Since $r'$ is arbitrarily small, we conclude that $x$ is a regular point of $C$.
\par 
We now assume that $x \in C$ is a regular point of $C$ and we will establish that $\Theta^1(\calH^1 \hel C ,x)=1$. By definition, there are $r > 0$ arbitrarily small such that, in particular, $\rmcard(C \cap \rmbdry B(x,r)) = 2$. If we denote by $x_{0,r}$ and $x_{1,r}$ the corresponding two intersection points, then
\begin{equation*}
C' = \left( C \setminus B(x,r) \right) \cup \left( \blseg x_{0,r},x \brseg \cup \blseg x,x_{1,r} \brseg \right)
\end{equation*}
is a competitor for $C$ in $B(x,r)$, according to \ref{prop.comp}, and therefore $\calH^1(C \cap B(x,r)) \leq (1+\xi(r)) 2r$ according to the almost minimizing property of $C$. If $r$ is chosen small enough for $1 + \xi(r) < (3/2)\exp[-\zeta(r)]$ then $\Theta^1(\calH^1 \hel C ,x) < 3/2$ according to \ref{monotonicity}. In turn, it follows from \ref{density.dichotomy} that $\Theta^1(\calH^1 \hel C , x ) = 1$.
\par 
We turn to proving the relative closedness of $\rmsing(C)$ in $U$. We first observe that the function
\begin{equation*}
\Theta : U \to \R : x \mapsto \Theta^1(\calH^1 \hel C , x)
\end{equation*}
is upper semicontinuous. Indeed, according to \ref{monotonicity}, $\Theta(x) = \inf_{ r > 0} \Theta_r(x)$, where
\begin{equation*}
\Theta_r : U \to \R : x \mapsto \exp[-\zeta(r)] \frac{\calH^1(C \cap B(x,r))}{2r} \,.
\end{equation*}
It then suffices to note that $U \to \R : x \mapsto (\calH^1 \hel C) (B(x,r))$ is uper semicontinuous, for each $r > 0$. Finally, $U \cap \rmsing(C) = U \cap \{ x : \Theta^1(\calH^1 \hel C , x) \geq 3/2\}$, according to \ref{density.dichotomy}, and the proof is complete.
\par
To conclude, since $C$ is rectifiable  \cite[Theorem 4.4.8.]{AMBROSIO.TILLI} it follows from \cite{KIR.94} that $\Theta^1( \calH^1 \hel C,x) = 1$ for $\calH^1$ almost every $x \in C$. Hence $\calH^1(U \cap \rmsing(C))=0$.
\end{proof}

\section{The excess of length of a nonstraight path and a regularity theorem}
\label{sec.5}

\begin{Empty}[Local hypothesis about the ambient Banach space]
In this section $X$ denotes a uniformly rotund Banach space, with modulus of uniform rotundity $\delta_X$, recall \ref{UR}.
\end{Empty}

If $x_0,x_1$ and $z$ are the three vertices of a nondegenerate triangle in a {\em Hilbert space}, then the length of the broken line from $x_0$ to $x_1$ passing through $z$ is substantially larger than the length of the straight path from $x_0$ to $x_1$, specifically 
\begin{equation*}
\|x_0-z\| + \|z-x_1\| \geq \|x_0-x_1\| \sqrt{ 1 + \frac{h^2}{\max\{\|x_0-z\|^2,\|z-x_1\|^2\}}} \,, 
\end{equation*}
where $h = \rmdist (z , L)$, $L = x_0 + \rmspan \{x_1-x_0 \}$. This ensues from the Pythagorean Theorem and from the observation that among all such triangles with same height $h$, the isoceles triangle has the shortest perimeter.
\par 
If $X$ is an arbitrary Banach space, the collection of inequalities
\begin{equation*}
\|x_0-z\| + \|z-x_1\| \geq \|x_0-x_1\| \left( 1 + \delta_2 \left( \frac{h}{\max\{\|x_0-z\|,\|z-x_1\|\}} \right) \right) \,,
\end{equation*}
for some gauge $\delta_2$, is equivalent to the uniform convexity of $X$, see \cite[Lemma IV.1.5]{DEVILLE.GODEFROY.ZIZLER}. Next comes an ersatz of the Pythagorean Theorem that makes this observation quantitative, showing that $\delta_2$ and the modulus of uniform convexity of $X$ have the same asymptotic behavior.

\begin{Proposition}
\label{excess.length}
Assume that $X$ is a uniformly convex Banach space with modulus of uniform convexity $\delta_X$, and that $x_0,x_1,z$ are the vertices of a nondegenerate triangle. It follows that
\begin{equation*}
\|x_0-z\| + \|z-x_1\| \geq \|x_0-x_1\| \left( 1 + \delta_X \left( \frac{ \rmdist(z, x_0 + \rmspan \{x_1-x_0\}) } { 2 (\|x_0-z\| + \|z-x_1\| ) } \right) \right)\,.
\end{equation*}
\end{Proposition}

\begin{proof}
We let $S_X = \rmbdry B(0,1)$ denote the unit sphere. Referring to the Hahn-Banach Theorem, with each unit vector $v \in S_X$ we associate a closed linear subspace $H_v \subset X$ of codimension 1 such that $B_X$ lies entirely on one side of $v+H_v$. We observe that $v \not\in H_v$ by necessity, and that there is no restriction to assume that $H_{-v} = H_v$.
\par 
We denote by $V$ the 2 dimensional affine subspace of $X$ containing $x_0,x_1$ and $z$, and $V_0$ the corresponding linear subspace. We define $v = \|x_1-x_0\|^{-1}(x_1-x_0)$, $L_v = V_0 \cap H_v$, and $L'_v = \rmspan\{v\}$. As $L_v$ and $L'_v$ are nonparallel, the following defines $y'$:
\begin{equation*}
\{y'\} = (x_0 + L'_v) \cap (z + L_v) \,.
\end{equation*}
We now distinguish between the cases when $y'$ lies on $x_0+L'_v$ between $x_0$ and $x_1$, or not.
\par 
{\em First case : $y' \in \blseg x_0,x_1 \brseg$.} We define $\rho_0 = \|x_0-y'\|$, $B_0 = V \cap B(x_0,\rho_0)$, and we let $S_0$ be the boundary of $B_0$ relative to $V$. Our choice of $H_v$ guarantess that $B_0$ lies, in $V$, on one side of $y' + L_v$. Therefore, among the two points of which $S_0 \cap (x_0 + \rmspan \{z-x_0 \})$ consists, one belongs to $\blseg x_0,z \brseg$. We denote it as $y$. 

\begin{center}
\begin{figure}[!h]
\scalebox{1} 
{
\begin{pspicture}(0,-2.669416)(11.373463,3.2067313)
\rput{-45.0}(0.9744069,2.1613634){\psellipse[linewidth=0.04,dimen=outer](3.0962,-0.09553148)(2.735,1.643688)}
\psdots[dotsize=0.12](3.0534627,-0.06826858)
\psline[linewidth=0.03cm](0.5934628,-0.06826858)(9.373463,-0.08826858)
\psline[linewidth=0.03cm](1.2734629,-1.1082686)(8.033463,2.8317313)
\psline[linewidth=0.03cm](3.7334628,3.1917315)(5.8734627,-1.9882686)
\psdots[dotsize=0.12](4.693463,0.89173144)
\psdots[dotsize=0.12](4.493463,0.7917314)
\psdots[dotsize=0.12](5.073463,-0.08826858)
\psdots[dotsize=0.12](7.033463,-0.08826858)
\usefont{T1}{ptm}{m}{n}
\rput(7.1234627,0.19673142){$x_1$}
\usefont{T1}{ptm}{m}{n}
\rput(5.3334627,0.13673142){$y'$}
\usefont{T1}{ptm}{m}{n}
\rput(2.9834628,0.2967314){$x_0$}
\usefont{T1}{ptm}{m}{n}
\rput(1.5134628,2.4367313){$S_0$}
\usefont{T1}{ptm}{m}{n}
\rput(1.9734628,1.4567314){$B_0$}
\usefont{T1}{ptm}{m}{n}
\rput(4.8534627,1.2167314){$z$}
\usefont{T1}{ptm}{m}{n}
\rput(4.1234627,0.8567314){$y$}
\usefont{T1}{ptm}{m}{n}
\rput(6.403463,-2.3232687){$z+L_v$}
\usefont{T1}{ptm}{m}{n}
\rput(10.463463,-0.02326858){$x_0+L_v'$}
\psline[linewidth=0.03cm,arrowsize=0.05291667cm 2.0,arrowlength=1.4,arrowinset=0.4,tbarsize=0.07055555cm 5.0]{<-|}(7.8534627,-0.7082686)(7.213463,-0.7082686)
\usefont{T1}{ptm}{m}{n}
\rput(7.523463,-0.9632686){$v$}
\end{pspicture} 
}
\caption{Situation of the first case.}
\end{figure}
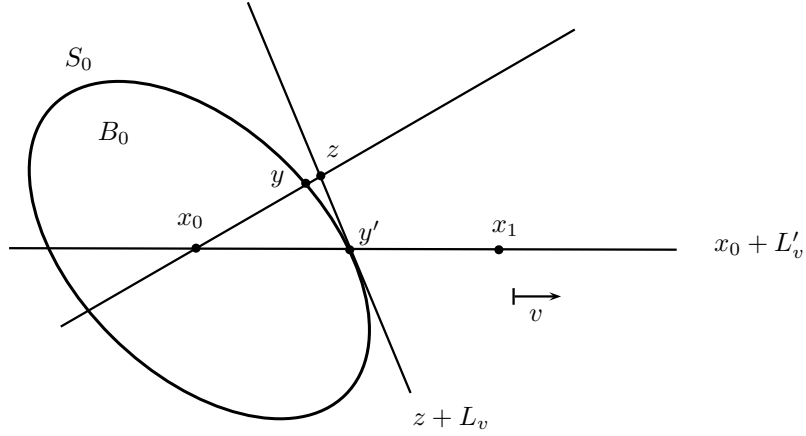
\end{center}

We now abbreviate $\veps = \rho_0^{-1} \|y'-y\|$ and we infer from the uniform convexity of $X$ that
\begin{equation}
\label{eq.5.bis}
\left\| \frac{y'+y}{2} \right\| = (1-\beta) \rho_0 \leq (1- \delta_X(\veps)) \rho_0 \,,
\end{equation}
where the equality defines $\beta$ (thus $\beta \geq \delta_X(\veps)$).
\par 
We claim that
\begin{equation}
\label{eq.5}
\|y-z\| \geq \beta \rho_0 \,.
\end{equation}
In other words, we are comparing the lengths of the line segments $\blseg y,z \brseg$ and $\blseg (y'+y)/2 , w \brseg$, where $w$ is at the intersection of $x_0 + \rmspan \{ (y'+y)/2 - x_0\}$ and $S_0$, in the triangle of vertices $y',y$ and $z$ (see Figure 3 just after). 

With each $s \in \R$ we associate $z_s = y' + s(z-y')$. 
If we denote by $P : V \to V$ the projection onto the line $y' + \rmspan \{y'-x_0\}$, parallel to the line $y' + \rmspan \{z-y'\}$, then the maps $f_s = \left( P \restriction_{\blseg x_0,z_s \brseg} \right)^{-1}$ are affine bijections from $\blseg x_0,y' \brseg$ to $\blseg x_0,z_s \brseg$.
We note that the convex function $s \mapsto \|z_s-x_0\|$ has a mimimum at $s=0$ -- according to our choice of $H_v$ --, and therefore is nondecreasing on the interval $s \geq 0$. It follows that there exists a nondecreasing function $s \mapsto \lambda(s)$, $s \geq 0$, such that
\begin{equation*}
\calH^1( f_s(I)) = \lambda(s) \calH^1(I)
\end{equation*}
for each interval $I \subset \blseg x_0,y' \brseg$. We now choose $0 \leq s_1 \leq s_2$ such that $(y'+y)/2 \in \blseg x_0,z_{s_1} \brseg$ and $z = z_{s_2}$. It is easily seen that
\begin{equation*}
P \left( \blseg y,z \brseg \right) \supset P \left( \blseg (y'+y)/2,z_{s_1} \brseg \right) \supset P \left( \blseg (y'+y)/2 , w \brseg \right) \,
\end{equation*}
and the the proof of \eqref{eq.5} follows. 
\begin{center}
\begin{figure}[!h]\label{figurebluered}
\scalebox{1} 
{
\begin{pspicture}(0,-3.455)(9.395,3.455)
\definecolor{color1097}{rgb}{1.0,0.2,0.2}
\definecolor{color1099}{rgb}{0.4,0.4,1.0}
\definecolor{color1151}{rgb}{0.4,0.4,0.4}
\psline[linewidth=0.03cm](1.1,-1.24)(9.38,-1.28)
\psline[linewidth=0.03cm](1.1,-1.22)(7.34,2.2)
\psdots[dotsize=0.12](1.1,-1.24)
\usefont{T1}{ptm}{m}{n}
\rput(0.45,-1.255){$x_0$}
\psbezier[linewidth=0.03](3.36,2.44)(5.32,1.74)(6.66,0.44)(6.86,-3.32)
\psline[linewidth=0.03cm](5.52,3.44)(7.1,-3.44)
\usefont{T1}{ptm}{m}{n}
\rput(6.12,1.825){$z$}
\usefont{T1}{ptm}{m}{n}
\rput(6.96,-1.035){$y'$}
\usefont{T1}{ptm}{m}{n}
\rput(5.39,1.525){$y$}
\psdots[dotsize=0.12](5.98,1.46)
\psdots[dotsize=0.12](6.62,-1.28)
\psline[linewidth=0.02cm](1.08,-1.26)(7.94,0.48)
\psline[linewidth=0.02cm](5.44,1.14)(6.6,-1.3)
\psdots[dotsize=0.12](5.42,1.16)
\psline[linewidth=0.05cm,linecolor=color1097](5.98,-0.02)(6.14,0.02)
\psline[linewidth=0.05cm,linecolor=color1099](5.48,1.2)(5.92,1.44)
\psline[linewidth=0.05cm,linecolor=color1099](5.46,1.16)(5.94,1.42)
\psline[linewidth=0.05cm,linecolor=color1097](5.98,0.0)(6.14,0.04)
\usefont{T1}{ptm}{m}{n}
\rput(4.27,-1.615){$\rho_0$}
\psline[linewidth=0.03cm,linecolor=color1151,arrowsize=0.05291667cm 2.0,arrowlength=1.4,arrowinset=0.4]{<-}(5.88,0.04)(2.58,1.0)
\usefont{T1}{ptm}{m}{n}
\rput(2.36,1.345){$\beta\rho_0$}
\usefont{T1}{ptm}{m}{n}
\rput(3.8,2.605){$S_0$}
\end{pspicture} 
}
\caption{Illustration of \eqref{eq.5}: the blue segment is larger than the tiny red one.}
\end{figure}
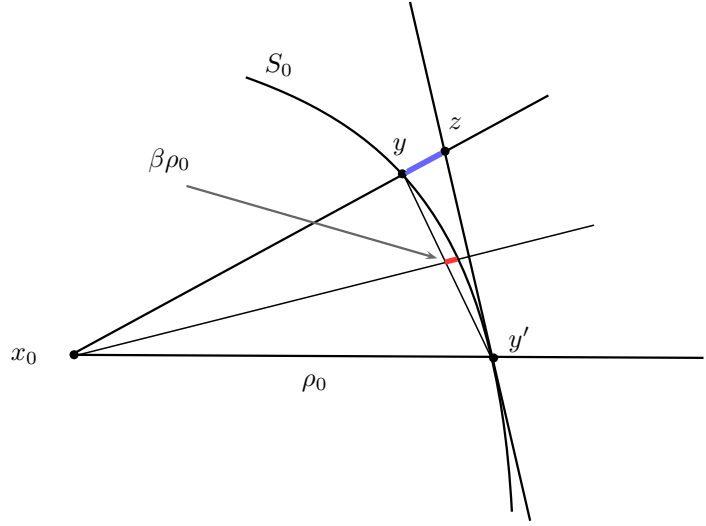
\end{center}
\par 
It ensues from \eqref{eq.5} that
\begin{equation}
\label{eq.7}
\begin{split}
\|x_0-z\| & = \|x_0-y\| + \|y-z\| \\
& \geq (1 + \beta) \|x_0-y\| \\
& \geq (1 + \delta_X(\veps)) \|x_0-y \| \,.
\end{split}
\end{equation}
We are now going to establish that
\begin{equation}
\label{eq.6}
\|x_0-z\| \geq \left( 1 + \delta_X \left( \frac{ \rmdist(z , x_0 + \rmspan \{x_1 - x_0 \})}{2 \|x_0-z \|} \right) \right) \|x_0-y'\| \,.
\end{equation}
Either $\|x_0-z\| \geq 2 \|x_0-y\|$, in which case \eqref{eq.6} readily holds since $\delta_X(\eta) \leq 1$ for each $0 < \eta \leq 2$, or $\|x_0-z\| < 2 \|x_0-y\|$. In the latter case,
\begin{equation*}
\begin{split}
\rmdist( z, x_0 + \rmspan \{x_1-x_0\}) & = \frac{\|z-x_0\|}{\|y-x_0\|} \rmdist(y , x_0 + \rmspan \{x_1-x_0\}) \\
& \leq 2 \rmdist(y , x_0 + \rmspan \{x_1-x_0\}) \\
& \leq 2 \|y-y'\| \,,
\end{split}
\end{equation*}
therefore
\begin{equation*}
\veps = \frac{\|y-y'\|}{\rho_0} \geq \frac{ \rmdist( z, x_0 + \rmspan \{x_1-x_0\}) }{ 2 \|x_0-y\| }
\geq \frac{ \rmdist( z, x_0 + \rmspan \{x_1-x_0\}) }{ 2 \|x_0-z\| } \,. 
\end{equation*}
Thus \eqref{eq.6} follows from \eqref{eq.7}, because $\delta_X$ is nondecreasing.
\par 
We now repeat the same argument with the vertex $x_1$ playing the role of $x_0$. Since $L_{-v} = L_v$, we observe that the new point $y'$ coincides with the one found previously. The analogous calculations therefore yield
\begin{equation}
\label{eq.8}
\|x_1-z\| \geq \left( 1 + \delta_X \left( \frac{ \rmdist(z , x_0 + \rmspan \{x_1 - x_0 \})}{2 \|x_1-z \|} \right) \right) \|x_1-y'\| \,.
\end{equation}
Upon summing inequalities \eqref{eq.6} and \eqref{eq.8} we obtain
\begin{equation}
\label{eq.9}
\|x_0-z\| + \|z-x_1\| \geq \left( 1 + \delta_X \left( \frac{ \rmdist( z , x_0 + \rmspan \{x_1-x_0 \} ) }{ 2 \max \{ \|x_0-z\| , \|z-x_1\| \} } \right) \right) \|x_0-x_1\| \,,
\end{equation}
which proves the proposition in this case.
\par 
{\em Second case : $y' \not \in \blseg x_0,x_1 \brseg$.} We start by observing that we may as well assume
\begin{equation}
\label{eq.10}
\lambda := \min \{ \|x_0-z\| , \|z-x_1\| \} < \|x_0-x_1\|
\end{equation}
for otherwise the conclusion
\begin{multline*}
\|x_0-z\| + \|z-x_1\| \geq 2 \|x_0-x_1\| \\ \geq \left( 1 + \delta_X \left( \frac{ \rmdist( z , x_0 + \rmspan \{x_1-x_0 \} ) }{ 2 ( \|x_0-z\| + \|z-x_1\| ) } \right) \right) \|x_0-x_1\|
\end{multline*}
readily follows from the trivial inequality $0 < \delta_X(\eta) \leq 1$, $0 < \eta \leq 2$. We define $c = (x_0+x_1)/2$, $r = \|x_0-x_1\|/2$, and we let
\begin{equation*}
\{x_0' , x_1' \} = (x_0 + L'_v) \cap \rmbdry B(c,r + \lambda) \,.
\end{equation*}
It follows from our choice of $\lambda$, \eqref{eq.10}, and the definition of $L_v$ that the point $y'$ defined by
\begin{equation*}
\{y'\} = (x_0 + L'_v) \cap (z + L_v )
\end{equation*}
belongs to the line segment $\blseg x_0' , x_1' \brseg$. Therefore the {\em first case} of this proof applies to the triangle with vertices $x_0',x_1'$ and $z$. Accordingly,
\begin{equation}
\label{eq.11}
\|x_0' - z \| + \|z-x_1' \| \geq (1 + \delta' ) \|x_0' - x_1' \|
\end{equation}
where
\begin{equation*}
\delta' = \delta_X \left( \frac{ \rmdist ( z , x_0 + \rmspan \{x_1-x_0\} ) }{ 2 \max \{ \|x_0' - z \| , \| z-x_1' \| \} } \right) \,.
\end{equation*}
We notice that
\begin{multline*}
\|x_0'-z\| + \|z-x_1'\| \leq \|x_0-z\| + \|z-x_1\| + 2 \lambda \\ \leq \|x_0-z\| + \|z-x_1\| + 2 (1+ \delta')\lambda
\end{multline*}
as well as
\begin{equation*}
\|x_0'-x_1'\| = \|x_0-x_1\| + 2 \lambda \,.
\end{equation*}
Plugging these inequalities in \eqref{eq.11} yields
\begin{equation*}
\|x_0-z\| + \|z-x_1\| \geq (1 + \delta') \|x_0-x_1\| \,,
\end{equation*}
and it remains to observe that
\begin{equation*}
\delta' \geq \delta_X \left( \frac{ \rmdist ( z , x_0 + \rmspan \{x_1-x_0\} ) }{ 2 ( \|x_0 - z \| + \| z-x_1\| )} \right)
\end{equation*}
because
\begin{equation*}
\max \{ \|x_0'-z\| , \|z-x'_1\| \} \leq \max \{ \|x_0-z \| , \|z-x_1\| \} + \lambda = \|x_0-z\| + \|z-x_1\| \,.
\end{equation*}
The proof is now complete.
\end{proof}


The following is an ersatz of the Pythagorean Theorem, valid in uniformly convex Banach spaces.

\begin{Proposition}[Height bound]
\label{keylemma}
Assume that:
\begin{enumerate}
\item[(A)] $C \subset X$ is compact, connected, $U \subset X$ is open, $0 < r < r_0$, $x \in C$, $B(x,r) \subset U$;
\item[(B)] $\xi$ is a gauge and $\xi(r) \leq \delta_X(1/32)$;
\item[(C)] $C$ is $(\xi,r_0)$ almost minimizing in $U$;
\item[(D)] $C \cap B(x,r)$ is a Lipschitz curve with endpoints $x_0$ and $x_1$, and $C \cap \rmbdry B(x,r) = \{x_0,x_1\}$;
\item[(E)] $L = \rmspan\{x_1-x_0\}$.
\end{enumerate}
It follows that
\begin{enumerate}
\item[(F)] $\|x_1-x_0\| \geq (1-\xi(r)) 2r$;
\item[(G)] For every $z \in C \cap B(x,r)$ one has
\begin{equation*}
\rmdist(z , x+L) \leq 16 ( \delta_X^{-1} \circ \xi)(r) \,.
\end{equation*}
\end{enumerate}
\end{Proposition}

\begin{Remark}
Under the same assumptions one can in fact show that
\begin{equation*}
\frac{1}{r} \rmdist_\calH \big[ C \cap B(x,r) , (x+L) \cap B(x,r) \big] \leq 80 (\delta_X^{-1} \circ \xi)(r) \,,
\end{equation*}
but only the weaker version (G) will be used in the proof of Theorem \ref{mainregth}.
\end{Remark}

\begin{proof}[Proof of Proposition \ref{keylemma}.]
We start by observing that
\begin{equation}
\label{eq.15}
2r \leq \calH^1(C \cap B(x,r)) \leq (1 + \xi(r)) \|x_1-x_0 \| \,.
\end{equation}
The first inequality results from \ref{local.topology}(D) as in the proof of \ref{density.first}(C), and the second inequality follows from the almost minimizing property of $C$ together with the fact that
\begin{equation*}
C' = \left( C \setminus B(x,r) \right) \cup \blseg x_0 , x_1 \brseg
\end{equation*}
is a competitor for $C$ in $B(x,r)$, according to \ref{prop.comp}. This proves conclusion (F).
\par
Let $z \in C \cap B(x,r)$ and define
\begin{equation*}
h_z = \rmdist(z , x_0 + L) \,.
\end{equation*}
Notice that $\|x_0-z\| + \|z-x_1\| \leq 4r$, name $\Gamma$ the Lipschitz curve $C \cap B(x,r)$, and write $\Gamma_0$ (resp. $\Gamma_1$) for the subcurve of $\Gamma$ with endpoints $x_0$ and $z$ (resp. $z$ and $x_1$). It ensues from \ref{excess.length} that
\begin{equation*}
\begin{split}
\|x_0-x_1\| \left( 1 + \delta_X \left( \frac{h_z}{8r} \right) \right) & \leq \|x_0-z\| + \|z-x_1\| \\
& \leq \calH^1(\Gamma_0) + \calH^1(\Gamma_1) \\
& = \calH^1(C \cap B(x,r)) \\
& \leq (1 + \xi(r)) \|x_1 -x_0\|\,
\end{split}
\end{equation*}
where the last inequality follows from \eqref{eq.15}. Therefore
\begin{equation*}
\delta_X \left( \frac{h_z}{8r} \right) \leq \xi(r) \,,
\end{equation*}
and in turn,
\begin{equation}
\label{eq.16}
\frac{h_z}{r} \leq 8 (\delta_X^{-1} \circ \xi)(r) \,,
\end{equation}
recall \ref{UR}. 
\par 
We abbreviate $\eta = 8r (\delta_X^{-1} \circ \xi)(r)$. So far we showed that given $z \in C \cap B(x,r)$, there is $v_z \in L$ such that $\|z - (v_z+x_0)\| \leq \eta$. As $x \in C \cap B(x,r)$, there exists $v_x \in L$ such that $\|x - (v_x+x_0) \| \leq \eta$. Therefore $\|z - (v_z-v_x+x)\| \leq \|z-(v_z+x_0)\|+\|x-(v_x+x_0)\| \leq 2\eta$ and the proof of (G) is complete.
\end{proof}

In the following we use the terminology {\em universal constant} for real numbers that do not depend on the data ($X$, $C$, $\xi$ etc).

\begin{Theorem}[$C^1$ regularity]
\label{mainregth}
There are universal constants $\boldeta > 0$ and $\bC > 0$ with the following property.
Assume that:
\begin{enumerate}
\item[(A)] $C \subset X$ is compact, connected, $U \subset X$ is open, $r_0 > 0$, $x_0 \in C$, $B(x_0,r_0) \subset U$;
\item[(B)] $\xi$ is a gauge and {\bf the gauge $\delta_X^{-1} \circ \xi$ is Dini};
\item[(C)] $C$ is $(\xi,r_0)$ almost minimizing in $U$;
\item[(D)] $\exp[\zeta(r_0)] \leq 1 + \boldeta$ where $\zeta$ is the mean slope of $\xi$;
\item[(E)] $(\delta_X^{-1} \circ \xi)(r_0) \leq \boldeta$;
\item[(F)] $\calH^1(C \cap B(x_0,r_0)) \leq (1+\boldeta)2r_0$.
\end{enumerate}
It follows that $C \cap B(x_0,\boldeta r_0)$ is a $C^1$ curve $\Gamma$. Furthermore if $\gamma$ is an arclength parametrization of $\Gamma$ then 
\begin{equation*}
\rmosc(\gamma' ; \eta) \leq \bC \omega( \bC \eta)
\end{equation*}
where $\omega$ is the mean slope of the Dini gauge $\delta_X^{-1} \circ \xi$.
\end{Theorem}

\begin{Remark}
\label{55}
Several comments are in order.
\begin{enumerate}
\item[(A)] Since $\delta_X(\veps) \leq  \veps^2$ for every $0 < \veps < 2$, see \cite[p. 63]{LINDENSTRAUSS.TZAFRIRI.II}, we infer that $t \leq \sqrt{t} \leq  \delta_X^{-1}(t)$ whenever $0 < t \leq 1$. Thus for any gauge $\xi$ and any $t$ such that $\xi(t) \leq 1$ one has
\begin{equation*}
\xi(t) \leq \delta_X^{-1}(\xi(t)) \,,
\end{equation*}
 and it immediately follows from the definition that $\xi$ is a Dini gauge whenever $\delta_X^{-1} \circ \xi$ is Dini. In particular hypothesis (D) makes sense. In our statement $\zeta$ is the mean slope of $\xi$ and $\omega$ is the mean slope of $\delta_X^{-1} \circ \xi$.
\item[(B)] Under the assumptions of the Theorem, if $0 < \tau < 1$, $x \in C \cap B(x_0,\tau r_0)$ and $0 < r \leq (1-\tau)r_0$ then \ref{monotonicity} (almost monotonicity) implies that
\begin{equation*}
\begin{split}
\frac{\calH^1(C \cap B(x,r))}{2r} & \leq \exp[ \zeta((1-\tau)r_0 ] \frac{ \calH^1(C \cap B(x,(1-\tau)r_0))}{2(1-\tau)r_0} \\
& \leq \exp[ \zeta(r_0) ] \frac{\calH^1(C \cap B(x_0,r_0))}{2r_0} \left( \frac{1}{1-\tau} \right) \\
& \leq \frac{(1+\boldeta)^2}{1-\tau} \,.
\end{split}
\end{equation*}
In particular, when $\tau$ is small, a version of hypothesis (F) holds (with a slightly worse constant that $\boldeta$) for $x$ close to $x_0$ and $r$ sligthly smaller than $r_0$. We will refer to this observation in the core of the proof.
\item[(C)] It is useful to notice that the Theorem applies at all. In fact, if $C \subset X$ is compact, connected and $\calH^1(C) < \infty$, and if $\xi$ is a Dini gauge, then for $\calH^1$ almost every $x_0 \in C$ there exists $r_0=r_0(x)>0$ such that assumptions (D), (E) and (F) are verified. Of course assumptions (D) and (E) are satisfied for $r_0$ small enough independently of $x_0$, whereas assumption (F) follows from the rectifiability of $C$ \cite[Theorem 4.4.8.]{AMBROSIO.TILLI} which implies that $\Theta^1(\calH^1 \hel C , x_0) = 1$ for $\calH^1$ almost every $x_0 \in C$, see e.g. \cite{KIR.94}. Thus the only nontrivial assumption of the Theorem, apart from the almost minimizing property of $C$, is that $\delta_X^{-1} \circ \xi$ be Dini.
\item[(D)] Let us spell out the kind of regularity obtained in case $\xi(r) \leq C r^\alpha$, $0 < \alpha \leq 1$, and $X = \bL_p$, $1 < p < \infty$. Here we consider an $\bL_p$ space relative to any measure space, and we point out that the $(\xi,r_0)$ almost minimizing property is verified by a solution of the variational problem $(\calP_{F,w})$ (see near the end of section \ref{sec.3}) provided the weight $w$ is H\"older continuous of exponent $\alpha$ (see \ref{lenghtminimiz}). One infers from \cite[p.63]{LINDENSTRAUSS.TZAFRIRI.II} that $\delta_{\bL_p}^{-1}(\veps) \leq C_p \veps^{1/\max\{2,p\}}$. Thus the gauge $\delta_{\bL_p}^{-1} \circ \xi$ is geometric and it follows that $\mathring{\Gamma}$ is $C^{1,\alpha/\max\{2,p\}}$.
\end{enumerate}
\end{Remark}

\begin{proof}[Proof of Theorem \ref{mainregth}.]
We say a pair $(x,r) \in C \times (\R^+\setminus \{0\})$ is {\em good} if $C \cap B(x,r)$ is a Lipschitz curve $\Gamma_{x,r}$ and if $C \cap \rmbdry B(x,r)$ consists exactly of the endpoints of $\Gamma_{x,r}$. 
\par 
{\em{\sc Claim \#1.} For every $x \in C \cap B(x_0,r_0/6)$ and every $0 < r \leq (5/6)r_0$, there exists $r/144 \leq \rho \leq r/6$ such that $(x,\rho)$ is good.}
\par 
We first notice that
\begin{equation*}
\calH^1(C \cap B(x,(5/6)r_0)) \leq \calH^1( C \cap B(x_0,r_0)) \leq (1+1/6)2r_0 \,,
\end{equation*}
according to hypothesis (F). Thus we infer from hypothesis (D) and \ref{monotonicity} (almost monotonicity) that
\begin{multline*}
\frac{\calH^1(C \cap B(x,r))}{2r} \leq \exp[\zeta(r)] \frac{\calH^1(C \cap B(x,r))}{2r} \\ \leq \exp[ \zeta((5/6)r_0) ] \frac{ \calH^1( C \cap B(x,(5/6)r_0))}{2(5/6)r_0} < \left( 1 + \frac{1}{24} \right) \left( \frac{7}{5} \right) = \frac{3}{2} - \frac{1}{24} \,.
\end{multline*}
We next define an $\calL^1$ measurable set $G = [0,r] \cap \{ \rho : \rmcard C \cap \rmbdry B(x,\rho) = 2 \}$, recall \ref{eilenberg}. It follows from \ref{local.topology}(D) and Eilenberg's inequality that
\begin{multline*}
2 \calL^1(G) + 3 ( r - \calL^1(G)) \leq \int_0^r (\rmcard C \cap \rmbdry B(x,\rho)) d\calL^1(\rho)\\ \leq \calH^1(C \cap B(x,r)) < \left( \frac{3}{2} - \frac{1}{24} \right) 2r \,.
\end{multline*}
It readily follows that $\calL^1(G) > r/12$. Pick $r' \in G$ with $r' \geq r/12$ and apply \ref{lip.reg} (with $\tau=1/6$).\cqfd
\par 
We apply {\sc Claim \#1} to $x=x_0$ and $r=(5/6)r_0$ to find some $\rho_0$ such that
\begin{equation*}
\frac{1}{144} \left( \frac{5r_0}{6} \right) \leq \rho_0 \leq \frac{1}{6} \left( \frac{5r_0}{6} \right) 
\end{equation*}
and $C \cap B(x_0,\rho_0)$ is a Lipschitz curve $\Gamma_0$.
\par 
We define $r_j = 72^{-j}$, $j \in \N$, and for the remaining part of this proof we assume $j \geq j_0$ where $j_0$ is chosen sufficiently large for $r_{j_0} \leq (5/6)r_0$. For such $r_j$ and $x \in C \cap B(x_0,\rho_0)$, {\sc Claim \#1} applies to yield $\rho_{x,j}$ such that
\begin{equation}
\label{eq.220}
\frac{r_{j+1}}{2} = \frac{r_j}{144} \leq \rho_{x,j} \leq \frac{r_j}{6} = 12 r_{j+1} 
\end{equation}
and $C \cap B(x,\rho_{x,j})$ is a Lipschitz curve $\Gamma_{x,j}$ whose two endpoints coincide with $C \cap \rmbdry B(x,\rho_{x,j})$. We easily infer that
\begin{equation}
\label{eq.21}
3 \leq \frac{\rho_{x,j}}{\rho_{x,j+1}} \leq 1728 \,.
\end{equation}
\par 
We parametrize $\Gamma_0$ by arclength $\gamma : [a,b] \to X$. Corresponding to each good pair $(x,\rho_{x,j})$ obtained above, we notice $\Gamma_{x,j} \subset \Gamma_0$, and therefore 
\begin{equation*}
\Gamma_{x,j} = \gamma ([s_x + h^-_{x,j},s_x + h^+_{x,j}])
\end{equation*}
 where $a < s_x < b$ is so that $x = \gamma(s_x)$ and $h^-_{x,j} < 0 < h^+_{x,j}$.
 \par 
 {\em{\sc Claim \#2.} One has
 \begin{equation}
 \label{eq.200}
 \rho_{x,j} \leq | h^\pm_{x,j}| \leq (1+2\xi(r_j))\rho_{x,j}
 \end{equation}
 and
 \begin{equation}
 \label{eq.201}
 3/2 \leq \frac{|h^\pm_{x,j}|}{|h^\pm_{x,j+1}|} \leq 3456 \,. 
 \end{equation}
 }
 \par 
The first of the four (set of) inequalities simply follows from $\rmLip \gamma \leq 1$,
\begin{equation*}
\rho_{x,j} = \| \gamma(s_x+h^\pm_{x,j})-\gamma(s_x)\| \leq |h^\pm_{x,j} |\,.
\end{equation*}
 \par 
 We next infer from the almost minimizing property of $C$ and the fact that $(x,\rho_{x,j})$ is a good pair,
 \begin{multline*}
| h^-_{x,j}| + |h^+_{x,j}| = \calH^1(\Gamma^-_{x,j}) + \calH^1(\Gamma^+_{x,j}) = \calH^1(\Gamma_{x,j}) \\ = \calH^1(C \cap B(x,\rho_{x,j})) \leq (1+ \xi(\rho_{x,j})) 2 \rho_{x,j} \,.
 \end{multline*}
 The second inequality ensues easily from this and from the first inequality.
 \par 
 The third and fourth inequalities are consequences of \eqref{eq.21} and of $|h^\pm_{x,j}| \leq 2 \rho_{x,j}$ which itself follows from the second inequality and $\xi(r_j) \leq 1/2$.\cqfd
\par 
Next we apply \ref{keylemma} to each good pair $(x,\rho_{x,j})$, $x \in C \cap B(x_0,\rho_0)$. This provides us with a 1 dimensional linear subspace $L_{x,j} \subset X$ such that
\begin{equation*}
\max \{ \rmdist( z , x + L_{x,j} ) : z \in \Gamma_{x,j} \} \leq \bveps(r_j)\rho_{x,j} \,,
\end{equation*}
where we have abbreviated
\begin{equation*}
\bveps(r) = 16(\delta_X^{-1} \circ \xi)(r) \,.
\end{equation*}
It is useful to recall that $\xi \leq \bveps$, cf. \ref{55}. Associated with each $h \in [h^-_{x,j},h^-_{x,j+1}] \cup [h^+_{x,j+1},h^+_{x,j}]$ we choose $v_{x,h,j} \in L_{x,j}$ such that
\begin{equation}
\label{eq.202}
\| \gamma(s_x+h)-x - v_{x,h,j}\| \leq \bveps(r_j) \rho_{x,j} \,.
\end{equation}
We choose a unit vector $w_{x,j} \in X$ spanning $L_{x,j}$, and $t_{x,h,j} \in \R$ such that $v_{x,h,j} =t_{x,h,j} w_{x,j}$. Replacing $w_{x,j}$ by $-w_{x,j}$ if necessary, we  also assume that $t_{x,h^+_{x,j},j} \geq 0$.
In the remaining part of this proof we will also use the following abbreviations:
\begin{gather*}
v^\pm_{x,j} = v_{x,h^\pm_{x,j},j} \\
t^\pm_{x,j} =t_{x,h^\pm_{x,j},j} \,.
\end{gather*}
\par

\begin{center}
\begin{figure}[!h]
\scalebox{1} 
{
\begin{pspicture}(0,-3.5319905)(13.5,3.5319905)
\definecolor{color190}{rgb}{0.6,0.6,0.6}
\rput{-30.0}(0.9297837,3.47){\psellipse[linewidth=0.04,dimen=outer](6.94,0.0)(3.8,2.87)}
\usefont{T1}{ptm}{m}{n}
\rput(8.53,3.235){$B(x,\rho_{x,j})$}
\psdots[dotsize=0.12](6.96,-0.11)
\psline[linewidth=0.03cm](1.04,-0.11)(11.92,-0.09)
\psline[linewidth=0.02cm](1.04,0.89)(12.2,0.89)
\psline[linewidth=0.02cm](2.22,-1.11)(12.3,-1.13)
\usefont{T1}{ptm}{m}{n}
\rput(11.85,0.175){$L_{x,j}$}
\psline[linewidth=0.04cm,arrowsize=0.05291667cm 2.0,arrowlength=1.4,arrowinset=0.4,tbarsize=0.07055555cm 5.0]{<-|}(12.08,-1.73)(11.28,-1.73)
\usefont{T1}{ptm}{m}{n}
\rput(11.79,-2.145){$w_{x,j}$}
\usefont{T1}{ptm}{m}{n}
\rput(11.96,-0.645){$C$}
\psdots[dotsize=0.12](3.68,-0.57)
\psdots[dotsize=0.12](10.52,-0.75)
\psdots[dotsize=0.12](10.22,-0.09)
\usefont{T1}{ptm}{m}{n}
\rput(11.17,1.715){$v_{x,j}^+=t_{x,j}^+w_{x,j}$}
\psline[linewidth=0.02cm,linecolor=color190,arrowsize=0.05291667cm 2.0,arrowlength=1.4,arrowinset=0.4]{->}(10.46,1.35)(10.24,0.09)
\usefont{T1}{ptm}{m}{n}
\rput(8.42,-2.005){$\gamma(s_x+h_{x,j}^+)$}
\psline[linewidth=0.02cm,linecolor=color190,arrowsize=0.05291667cm 2.0,arrowlength=1.4,arrowinset=0.4]{->}(9.26,-1.67)(10.38,-0.93)
\psline[linewidth=0.02cm,tbarsize=0.07055555cm 5.0]{|*-|*}(0.86,0.89)(0.88,-0.13)
\usefont{T1}{ptm}{m}{n}
\rput(1.72,0.375){$\bveps(r_j)\rho_{x,j}$}
\usefont{T1}{ptm}{m}{n}
\rput(7.22,0.415){$x=\gamma(s_x)$}
\psdots[dotsize=0.12](3.2,-0.11)
\usefont{T1}{ptm}{m}{n}
\rput(2.29,1.815){$v_{x,j}^-=t_{x,j}^-w_{x,j}$}
\psline[linewidth=0.02cm,linecolor=color190,arrowsize=0.05291667cm 2.0,arrowlength=1.4,arrowinset=0.4]{->}(1.66,1.49)(3.06,0.03)
\usefont{T1}{ptm}{m}{n}
\rput(3.22,-2.105){$\gamma(s_x+h_{x,j}^-)$}
\psline[linewidth=0.02cm,linecolor=color190,arrowsize=0.05291667cm 2.0,arrowlength=1.4,arrowinset=0.4]{->}(3.18,-1.85)(3.68,-0.79)
\psline[linewidth=0.04cm,arrowsize=0.05291667cm 2.0,arrowlength=1.4,arrowinset=0.4]{<->}(0.86,0.87)(0.88,-0.15)
\pscustom[linewidth=0.04]
{
\newpath
\moveto(1.76,-0.85)
\lineto(1.84,-0.79)
\curveto(1.88,-0.76)(1.935,-0.71)(1.95,-0.69)
\curveto(1.965,-0.67)(2.005,-0.635)(2.03,-0.62)
\curveto(2.055,-0.605)(2.11,-0.59)(2.14,-0.59)
\curveto(2.17,-0.59)(2.23,-0.59)(2.26,-0.59)
\curveto(2.29,-0.59)(2.355,-0.6)(2.39,-0.61)
\curveto(2.425,-0.62)(2.48,-0.645)(2.5,-0.66)
\curveto(2.52,-0.675)(2.57,-0.705)(2.6,-0.72)
\curveto(2.63,-0.735)(2.68,-0.76)(2.7,-0.77)
\curveto(2.72,-0.78)(2.755,-0.8)(2.77,-0.81)
\curveto(2.785,-0.82)(2.825,-0.835)(2.85,-0.84)
\curveto(2.875,-0.845)(2.93,-0.85)(2.96,-0.85)
\curveto(2.99,-0.85)(3.045,-0.85)(3.07,-0.85)
\curveto(3.095,-0.85)(3.14,-0.845)(3.16,-0.84)
\curveto(3.18,-0.835)(3.225,-0.82)(3.25,-0.81)
\curveto(3.275,-0.8)(3.32,-0.78)(3.34,-0.77)
\curveto(3.36,-0.76)(3.4,-0.735)(3.42,-0.72)
\curveto(3.44,-0.705)(3.48,-0.68)(3.5,-0.67)
\curveto(3.52,-0.66)(3.56,-0.64)(3.58,-0.63)
\curveto(3.6,-0.62)(3.635,-0.6)(3.65,-0.59)
\curveto(3.665,-0.58)(3.7,-0.555)(3.72,-0.54)
\curveto(3.74,-0.525)(3.775,-0.5)(3.79,-0.49)
\curveto(3.805,-0.48)(3.83,-0.455)(3.84,-0.44)
\curveto(3.85,-0.425)(3.875,-0.395)(3.89,-0.38)
\curveto(3.905,-0.365)(3.94,-0.335)(3.96,-0.32)
\curveto(3.98,-0.305)(4.015,-0.275)(4.03,-0.26)
\curveto(4.045,-0.245)(4.08,-0.205)(4.1,-0.18)
\curveto(4.12,-0.155)(4.16,-0.11)(4.18,-0.09)
\curveto(4.2,-0.07)(4.23,-0.03)(4.24,-0.01)
\curveto(4.25,0.01)(4.28,0.05)(4.3,0.07)
\curveto(4.32,0.09)(4.355,0.12)(4.37,0.13)
\curveto(4.385,0.14)(4.415,0.165)(4.43,0.18)
\curveto(4.445,0.195)(4.48,0.225)(4.5,0.24)
\curveto(4.52,0.255)(4.56,0.28)(4.58,0.29)
\curveto(4.6,0.3)(4.64,0.325)(4.66,0.34)
\curveto(4.68,0.355)(4.715,0.38)(4.73,0.39)
\curveto(4.745,0.4)(4.78,0.415)(4.8,0.42)
\curveto(4.82,0.425)(4.86,0.435)(4.88,0.44)
\curveto(4.9,0.445)(4.955,0.455)(4.99,0.46)
\curveto(5.025,0.465)(5.08,0.465)(5.1,0.46)
\curveto(5.12,0.455)(5.165,0.445)(5.19,0.44)
\curveto(5.215,0.435)(5.26,0.42)(5.28,0.41)
\curveto(5.3,0.4)(5.34,0.38)(5.36,0.37)
\curveto(5.38,0.36)(5.435,0.335)(5.47,0.32)
\curveto(5.505,0.305)(5.575,0.275)(5.61,0.26)
\curveto(5.645,0.245)(5.695,0.22)(5.71,0.21)
\curveto(5.725,0.2)(5.76,0.18)(5.78,0.17)
\curveto(5.8,0.16)(5.84,0.14)(5.86,0.13)
\curveto(5.88,0.12)(5.915,0.1)(5.93,0.09)
\curveto(5.945,0.08)(5.98,0.065)(6.0,0.06)
\curveto(6.02,0.055)(6.07,0.05)(6.1,0.05)
\curveto(6.13,0.05)(6.185,0.05)(6.21,0.05)
\curveto(6.235,0.05)(6.29,0.05)(6.32,0.05)
\curveto(6.35,0.05)(6.41,0.06)(6.44,0.07)
\curveto(6.47,0.08)(6.53,0.09)(6.56,0.09)
\curveto(6.59,0.09)(6.64,0.08)(6.66,0.07)
\curveto(6.68,0.06)(6.72,0.045)(6.74,0.04)
\curveto(6.76,0.035)(6.795,0.015)(6.81,0.0)
\curveto(6.825,-0.015)(6.86,-0.04)(6.88,-0.05)
\curveto(6.9,-0.06)(6.94,-0.09)(6.96,-0.11)
\curveto(6.98,-0.13)(7.015,-0.165)(7.03,-0.18)
\curveto(7.045,-0.195)(7.075,-0.225)(7.09,-0.24)
\curveto(7.105,-0.255)(7.145,-0.295)(7.17,-0.32)
\curveto(7.195,-0.345)(7.26,-0.39)(7.3,-0.41)
\curveto(7.34,-0.43)(7.41,-0.46)(7.44,-0.47)
\curveto(7.47,-0.48)(7.52,-0.5)(7.54,-0.51)
\curveto(7.56,-0.52)(7.6,-0.535)(7.62,-0.54)
\curveto(7.64,-0.545)(7.675,-0.56)(7.69,-0.57)
\curveto(7.705,-0.58)(7.745,-0.59)(7.77,-0.59)
\curveto(7.795,-0.59)(7.85,-0.59)(7.88,-0.59)
\curveto(7.91,-0.59)(7.965,-0.59)(7.99,-0.59)
\curveto(8.015,-0.59)(8.06,-0.585)(8.08,-0.58)
\curveto(8.1,-0.575)(8.14,-0.565)(8.16,-0.56)
\curveto(8.18,-0.555)(8.215,-0.535)(8.23,-0.52)
\curveto(8.245,-0.505)(8.285,-0.475)(8.31,-0.46)
\curveto(8.335,-0.445)(8.38,-0.42)(8.4,-0.41)
\curveto(8.42,-0.4)(8.47,-0.385)(8.5,-0.38)
\curveto(8.53,-0.375)(8.595,-0.365)(8.63,-0.36)
\curveto(8.665,-0.355)(8.73,-0.35)(8.76,-0.35)
\curveto(8.79,-0.35)(8.85,-0.355)(8.88,-0.36)
\curveto(8.91,-0.365)(8.965,-0.375)(8.99,-0.38)
\curveto(9.015,-0.385)(9.065,-0.4)(9.09,-0.41)
\curveto(9.115,-0.42)(9.16,-0.44)(9.18,-0.45)
\curveto(9.2,-0.46)(9.24,-0.48)(9.26,-0.49)
\curveto(9.28,-0.5)(9.315,-0.525)(9.33,-0.54)
\curveto(9.345,-0.555)(9.38,-0.585)(9.4,-0.6)
\curveto(9.42,-0.615)(9.46,-0.64)(9.48,-0.65)
\curveto(9.5,-0.66)(9.54,-0.675)(9.56,-0.68)
\curveto(9.58,-0.685)(9.62,-0.695)(9.64,-0.7)
\curveto(9.66,-0.705)(9.705,-0.715)(9.73,-0.72)
\curveto(9.755,-0.725)(9.81,-0.73)(9.84,-0.73)
\curveto(9.87,-0.73)(9.925,-0.73)(9.95,-0.73)
\curveto(9.975,-0.73)(10.025,-0.73)(10.05,-0.73)
\curveto(10.075,-0.73)(10.125,-0.73)(10.15,-0.73)
\curveto(10.175,-0.73)(10.225,-0.735)(10.25,-0.74)
\curveto(10.275,-0.745)(10.325,-0.75)(10.35,-0.75)
\curveto(10.375,-0.75)(10.44,-0.75)(10.48,-0.75)
\curveto(10.52,-0.75)(10.58,-0.745)(10.6,-0.74)
\curveto(10.62,-0.735)(10.66,-0.725)(10.68,-0.72)
\curveto(10.7,-0.715)(10.745,-0.71)(10.77,-0.71)
\curveto(10.795,-0.71)(10.845,-0.715)(10.87,-0.72)
\curveto(10.895,-0.725)(10.94,-0.74)(10.96,-0.75)
\curveto(10.98,-0.76)(11.025,-0.775)(11.05,-0.78)
\curveto(11.075,-0.785)(11.115,-0.8)(11.13,-0.81)
\curveto(11.145,-0.82)(11.185,-0.83)(11.21,-0.83)
\curveto(11.235,-0.83)(11.28,-0.84)(11.3,-0.85)
\curveto(11.32,-0.86)(11.36,-0.875)(11.38,-0.88)
\curveto(11.4,-0.885)(11.445,-0.895)(11.47,-0.9)
\curveto(11.495,-0.905)(11.545,-0.91)(11.57,-0.91)
\curveto(11.595,-0.91)(11.655,-0.91)(11.69,-0.91)
\curveto(11.725,-0.91)(11.79,-0.91)(11.82,-0.91)
\curveto(11.85,-0.91)(11.91,-0.91)(11.94,-0.91)
\curveto(11.97,-0.91)(12.035,-0.91)(12.07,-0.91)
\curveto(12.105,-0.91)(12.165,-0.91)(12.19,-0.91)
\curveto(12.215,-0.91)(12.255,-0.905)(12.3,-0.89)
}
\usefont{T1}{ptm}{m}{n}
\rput(4.14,0.435){$\Gamma_{x,j}$}
\end{pspicture} 
}
\caption{Notation of the proof of differentiability.}
\end{figure}
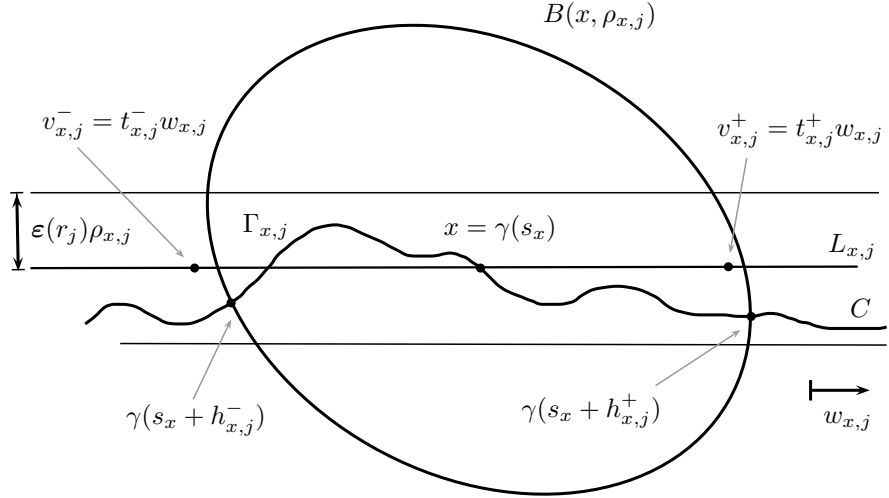
\end{center}

{\em{\sc Claim \#3.}The following hold:
\begin{equation}
\label{eq.203}
\left| |t^\pm_{x,j} | - \rho_{x,j} \right|\leq \bveps(r_j) \rho_{x,j} \,,
\end{equation}
and $t^+_{x,j} > 0$ as well as $t^-_{x,j} < 0$.
}
\par 
Since $\| \gamma(s_x+h^\pm_{x,j}) - x\| = \rho_{x,j}$ the first inequality is an immediate consequence of \eqref{eq.202}. In order to determine the signs of the $t^\pm_{x,j}$ we proceed as follows. As $t^+_{x,j} \geq 0$ (by choice of $w_{x,j}$) we infer from the first conclusion of the claim that $t^+_{x,j} > 0$. 
\par 
We now show that $t^-_{x,j} < 0$. We infer from \eqref{eq.202} that
\begin{equation}
\label{eq.24}
\| \gamma(s_x + h^+_{x,j}) - \gamma(s_x + h^-_{x,j}) - (t^+_{x,j} - t^-_{x,j}) w_{x,j} \| \leq 2 \bveps(r_j) \rho_{x,j} \,.
\end{equation}
It follows from \ref{keylemma}(F) that
\begin{equation}
\label{eq.25}
\| \gamma(s_x+h^+_{x,j}) - \gamma(s_x+h^-_{x,j}) \| \geq 2 \rho_{x,j} (1 - \xi(r_j)) \,.
\end{equation}
Thus,
\begin{equation}
\label{eq.26}
| t^+_{x,j} - t^-_{x,j} | \geq 2 \rho_{x,j} (1 - 3 \bveps(r_j)) \,,
\end{equation}
according to \eqref{eq.24} and \eqref{eq.25}.
If $t^+_{x,j}$ and $t^-_{x,j}$ had the same sign it would follow from the first conclusion of this claim that
\begin{equation*}
| t^+_{x,j} - t^-_{x,j} | \leq 2 \bveps(r_j) \rho_{x,j} \,.
\end{equation*}
Plugging this into \eqref{eq.26} would yield
\begin{equation*}
2(1-3\bveps(r_j)) \leq 2 \bveps(r_j) \,,
\end{equation*}
in contradiction with hypothesis (E).\cqfd
\par 
We now introduce a notation for the difference quotients of $\gamma$. Let $s \in [a,b]$ and $h \in \R \setminus \{0\}$ such that $s+h \in [a,b]$. We define
\begin{equation*}
\triangle_\gamma(s,h) = \frac{\gamma(s+h)-\gamma(s)}{h} \,.
\end{equation*}
We will prove that $\triangle_\gamma(s_x,h)$ and $\triangle_\gamma(s_x,h^+_{x,j})$ are close, for $h \in [h^-_{x,j},h^-_{x,j+1}] \cup [h^+_{x,j+1},h^+_{x,j}]$. To this end we will observe these two vectors are close to positive multiples of $w_{x,j}$, and they both have length nearly equal to $1$.
\par 
{\em{\sc Claim \#4.} For every $x \in C \cap B(x_0,\boldeta r)$, every $j \geq j_0$, and every $h \in [h^-_{x,j},h^-_{x,j+1}] \cup [h^+_{x,j+1},h^+_{x,j}]$ one has
\begin{equation*}
1 - C_0 \xi(r_j) \leq \| \triangle_\gamma(s_x , h ) \| \leq 1 \,,
\end{equation*}
for some universal constant $C_0 > 0$. 
}
\par 
The second inequality simply follows from the fact that $\rmLip \gamma \leq 1$. In order to establish the first inequality we abbreviate $x_j^- = \gamma(s_x+h^-_{x,j})$, $y=\gamma(s_x+h)$, and $x_j^+=\gamma(s_x+h^+_{x,j})$. We will also denote by $\Gamma_{p,q}$ the portion of the curve $\Gamma_{x,j}$ with endpoints $p,q \in \Gamma_{x,j}$. We first show that there exists a universal $\eta_0 > 0$ such that
\begin{equation}
\label{eq.204}
\eta_0 |h| \leq \|\gamma(s_x+h) - \gamma(s_x) \| = \|y-x\| 
\end{equation}
for every $h\in [h^-_{x,j},h^-_{x,j+1}] \cup [h^+_{x,j+1},h^+_{x,j}]$. We note that for such $h$, 
\begin{equation*}
\calH^1(\Gamma_{x,y}) = |h| \geq \rho_{x,j+1} \geq \frac{\rho_{x,j}}{1728} \,,
\end{equation*}
where we used the fact that $\gamma$ is parametrized by arclength, inequalities \eqref{eq.200} (applied to $j+1$) and \eqref{eq.21}. Now assume if possible that $\|x-y\| < \eta_0 |h|$ for some small $\eta_0$ to be determined momentarily. We would then infer that
\begin{equation*}
\calH^1(\Gamma_{y,x^+_j}) \geq \|y-x^+_j\| \geq \rho_{x,j} - \|x-y\| \geq (1-2\eta_0) \rho_{x,j} \,,
\end{equation*}
according to \ref{Ahlfors} and \eqref{eq.200}. In turn,
\begin{equation*}
\calH^1(\Gamma_{x,x^+_j}) = \calH^1(\Gamma_{x,y}) + \calH^1(\Gamma_{y,x^+_j}) \geq \rho_{x,j} \left( 1 + \frac{1}{1728} - 2 \eta_0 \right) \geq \rho_{x,j} \left( 1 + \frac{1}{1729} \right)
\end{equation*}
if $\eta_0$ is small enough. Of course the same estimate holds with $x^-_j$ in place of $x^+_j$. Therefore,
\begin{equation}
\label{eq.208}
\begin{split}
2 \rho_{x,j} \left(1 + \frac{1}{1729} \right) & \leq \calH^1(\Gamma_{x^-_j,x}) + \calH^1(\Gamma_{x,x^+_j}) \\
& = \calH^1(C \cap B(x,\rho_{x,j})) \\
& \leq \left( \frac{(1+\boldeta)^2}{1-\boldeta} \right) 2 \rho_{x,j} \,,
\end{split}
\end{equation}
according to Remark \ref{55}(B) applied with $\tau = \boldeta$ and $r=\rho_{x,j}$. One can choose $\boldeta$ small enough to yield a contradiction, thereby proving the validity of \eqref{eq.204}.
\par 
We next improve on \eqref{eq.204} by showing that there exists a universal $C_0 > 0$ such that
\begin{equation}
\label{eq.205}
\frac{|h|}{1 + C_0 \xi(r_j)} \leq \| \gamma(s_x+h) - \gamma(s_x) \| = \|y-x\| \,.
\end{equation}
We define $0 \leq \eta(h)$ by the the following first equation
\begin{equation*}
\| y-x \| = \frac{|h|}{1+\eta(h)} = \frac{\calH^1(\Gamma_{x,y})}{1+\eta(h)} \,
\end{equation*}
(the second equation ensues from the fact that $\gamma$ is an arclength parametrization) and we seek an upper bound for $\eta(h)$. We define set $S' \subset B(x,\rho_{x,j})$ by
\begin{equation*}
S' = \Gamma_{x^-_j,x} \cup \blseg x , y \brseg \cup \Gamma_{y,x^+_j} \,.
\end{equation*}
We notice that $S'$ is the image of a (possibly non injective) Lipschitz map $[-1,1] \to B(x,\rho_{x,j})$ that sends $-1$ to $x^-_j$ and $+1$ to $x^+_j$. It is now easy to check that 
\begin{equation*}
C' = (C \setminus B(x,\rho_{x,j})) \cup S'
\end{equation*}
is a competitor for $C$ in $B(x,\rho_{x,j})$. Therefore
\begin{equation}
\label{eq.206}
\calH^1(C \cap B(x,\rho_{x,j})) \leq (1 + \xi(r_j)) \calH^1(C' \cap B(x,\rho_{x,j})) \,.
\end{equation}
Furthermore,
\begin{equation}
\label{eq.207}
\calH^1(C' \cap B(x,\rho_{x,j})) \leq \calH^1(\Gamma_{x^-_j,x}) + \|x-y\| + \calH^1(\Gamma_{x,x^+_j}) \,.
\end{equation}
Since also
\begin{equation*}
\calH^1(C \cap B(x,\rho_{x,j})) = \calH^1(\Gamma_{x,j}) = \calH^1(\Gamma_{x^-_j,x}) + \calH^1(\Gamma_{x,y}) + \calH^1(\Gamma_{x,x^+_j}) \,,
\end{equation*}
it follows from the definition of $\eta(h)$ that
\begin{multline*}
\calH^1(C \cap B(x,\rho_{x,j})) - \calH^1(C' \cap B(x,\rho_{x,j})) = \calH^1(\Gamma_{x,y}) - \|x-y\| \\\geq \eta(h) \|y-x\|  \geq \eta(h)\eta_0 |h| \geq  \eta(h)\eta_0 |h^\pm_{x,j+1}| \geq \eta(h)\eta_0 \rho_{x,j+1}
\end{multline*}
according to \eqref{eq.204} and \eqref{eq.200} (applied to $j+1$). Plugging this into \eqref{eq.206} we obtain
\begin{equation*}
\eta(h) \eta_0 \rho_{x,j+1} \leq \xi(r_j) \calH^1(C' \cap B(x,\rho_{x,j}) \,.
\end{equation*}
Finally, we notice that \eqref{eq.207} and \eqref{eq.208} imply that $\calH^1(C' \cap B(x,\rho_{x,j})) \leq \calH^1(C \cap B(x,\rho_{x,j})) \leq 3 \rho_{x,j}$ (if $\boldeta$ is small enough). Thus,
\begin{equation*}
\eta(h) \leq 3 \frac{\rho_{x,j}}{\rho_{x,j+1}} \frac{1}{\eta_0} \xi(r_j) \,,
\end{equation*}
which proves \eqref{eq.205} in view of \eqref{eq.21}.\cqfd
\par
{\em{\sc Claim \#5.} Let $x \in C \cap B(x_0,\boldeta r)$ and $j \geq j_0$. If $h^-_{x,j} \leq h \leq h^-_{x,j+1}$ then $t_{x,h,j} < 0$ and if $h^+_{x,j+1} \leq h \leq h^+_{x,j}$ then $t_{x,h,j} > 0$.
}
\par 
We prove it in case $h^+_{x,j+1} \leq h \leq h^+_{x,j}$, the other case is analogous. We infer from \eqref{eq.202} that
\begin{equation*}
\| \gamma(s_x+h^+_{x,j}) - \gamma(s_x+h) - (t^+_{x,j}-t_{x,h,j})w_{x,j}\| \leq 2 \bveps(r_j) \rho_{x,j}\,,
\end{equation*}
and in turn from \eqref{eq.200} (applied to both $j$ and $j+1$) that
\begin{equation*}
\begin{split}
| t^+_{x,j}-t_{x,h,j}| & \leq \| \gamma(s_x+h^+_{x,j}) - \gamma(s_x+h) \| + 2 \bveps(r_j) \rho_{x,j} \\
& \leq h^+_{x,j} - h + 2 \bveps(r_j) \rho_{x,j} \\
& \leq h^+_{x,j} - h^+_{x,j+1} + 2 \bveps(r_j) \rho_{x,j} \\
& \leq (1+2\bveps(r_j))\rho_{x,j} - \rho_{x,j+1} + 2 \bveps(r_j) \rho_{x,j} \,.
\end{split}
\end{equation*}
Assuming if possible that $t_{x,h} \leq 0$ we would infer from \eqref{eq.203} that
\begin{equation*}
\rho_{x,j}(1-\bveps(r_j)) + |t_{x,h,j}| \leq t^+_{x,j} + |t_{x,h,j}| = |t^+_{x,j}-t_{x,h,j}| \,.
\end{equation*}
Thus,
\begin{equation*}
|t_{x,h,j}| \leq 5 \bveps(r_j) \rho_{x,j} - \rho_{x,j+1} \leq (5.1728\, \bveps(r_j) -1) \rho_{x,j+1} \leq (5.16.1728 \,\boldeta -1) \rho_{x,j+1}  \,.
\end{equation*}
Since the right member of the inequality is negative if $\boldeta$ is chosen small enough, we obtain the sought for contradiction.\cqfd
\par 
We are now ready to finish off the proof of the theorem. We start by fixing $x \in B(x_0,\boldeta r_0)$ and we will show that $\gamma$ is differentiable at $s_x$. To this end we fix first $j \geq j_0$ and $h \in [h^-_{x,j},h^-_{x,j+1}] \cup [h^+_{x,j+1},h^+_{x,j}]$. Dividing \eqref{eq.202} by $|h|$, and referring to \eqref{eq.200} and \eqref{eq.21}, we obtain
\begin{multline}
\label{eq.210}
\left| \| \triangle_\gamma(s_x,h) \| - \left| \frac{t_{x,h,j}}{h} \right| \right| \\
\leq \left\| \triangle_\gamma(s_x,h) - \frac{t_{x,h,j}}{h} w_{x,j} \right\| \leq \bveps(r_j) \frac{\rho_{x,j}}{|h|} \leq \bveps(r_j) \frac{\rho_{x,j}}{|h^\pm_{x,j+1}|} \leq \bveps(r_j) \frac{\rho_{x,j}}{\rho_{x,j+1}} \\ \leq 1728\, \bveps(r_j) \,.  
\end{multline}
We notice that $t_{x,h,j}/h$ and $t^+_{x,j}/h^+_{x,j}$ are both positive according to {\sc Claims \#3} and \#5. It therefore follows from {\sc Claim \#4} that
\begin{equation*}
\begin{split}
\left| \frac{t_{x,h,j}}{h} - \frac{t^+_{x,j}}{h^+_{x,j}} \right| & \leq \left| \left| \frac{t_{x,h,j}}{h} \right| - \| \triangle_\gamma(s_x,h) \| \right| + \left| \| \triangle_\gamma(s_x,h) \| - \| \triangle_\gamma(s_x,h^+_{x,j}) \| \right| \\ & \qquad+ \left| \| \triangle_\gamma(s_x,h^+_{x,j}) \| - \left| \frac{t^+_{x,j}}{h^+_{x,j}} \right| \right| \\
& \leq 2.1728\, \bveps(r_j) + C_0 \xi(r_j) \\
& = C_0' \bveps(r_j)
\end{split}
\end{equation*}
for some universal constant $C_0' > 0$. In view of \eqref{eq.210} we thus obtain
\begin{equation}
\label{eq.211}
\begin{split}
\| \triangle_\gamma(s_x,h^+_{x,j}) - \triangle_\gamma(s_x,h) \| & \leq \left\| \triangle_\gamma(s_x,h^+_{x,j}) - \frac{t^+_{x,j}}{h^+_{x,j}} w_{x,j} \right\| \\ & \qquad+ \left\| \left( \frac{t^+_{x,j}}{h^+_{x,j}} - \frac{t_{x,h,j}}{h} \right) w_{x,j} \right\| + \left\| \frac{t_{x,h,j}}{h} w_{x,j} - \triangle_\gamma(s_x,h) \right\| \\
& \leq 2.1728\, \bveps(r_j) + C_0' \bveps(r_j) \\
& = C_0'' \bveps(r_j) \,.
\end{split}
\end{equation}
Applying temporarily this inequality to $h=h^+_{x,j+1}$ we see that
\begin{equation*}
\| \triangle_\gamma(s_x,h^+_{x,j+1}) - \triangle_\gamma(s_x,h^+_{x,j}) \| \leq C_0'' \bveps(r_j) \,.
\end{equation*}
Thus for any $k \geq j_0$ and any $l \geq 1$ we have
\begin{multline}
\label{eq.212}
\| \triangle_\gamma(s_x,h^+_{x,k+l}) - \triangle_\gamma(s_x,h^+_{x,k}) \| \leq \sum_{j=k}^{k+l-1} \| \triangle_\gamma(s_x,h^+_{x,j+1}) - \triangle_\gamma(s_x,x^+_{x,j}) \| \\ \leq (16.72/71)C_0'' \omega(r_{k-1}) =: C_0''' \omega(r_{k-1}).
\end{multline}
according to \ref{gauges}. Therefore $\{ \triangle_\gamma(s_x,h^+_{x,j}) \}_j$ is a Cauchy sequence, whence also convergent, and we denote its limit by $\gamma'(s_x)$. In order to verify that $\gamma'(s_x)$ is the derivative of $\gamma$ at $s_x$ we combine \eqref{eq.211} with \eqref{eq.212} in which we let $l \to \infty$: If $h \in [h^-_{x,j},h^-_{x,j+1}] \cup [h^+_{x,j+1},h^+_{x,j}]$ then
\begin{equation}
\label{eq.213}
\begin{split}
\|\gamma'(s_x) - \triangle_\gamma(s_x,h) \| & \leq \| \gamma'(s_x) - \triangle_\gamma(s_x,h^+_{x,j}) \| + \| \triangle_\gamma(s_x,h^+_{x,j}) - \triangle_\gamma(s_x,h) \| \\
& \leq C_0'' \bveps(r_j) + C_0''' \omega(r_{j-1}) \\
& \leq 2C_0''' \omega(r_{j-1}) \\ 
& = C^{iv}_0 \omega(r_{j-1}).
\end{split}
\end{equation}
We notice that \eqref{eq.220}, \eqref{eq.21} and \eqref{eq.200} imply that $r_{j-1} \leq 144.1728\,|h|$. Thus \eqref{eq.213} finally yields
\begin{equation*}
\|\gamma'(s_x) - \triangle_\gamma(s_x,h) \| \leq C^v_0 \omega(C^v_0 |h|)\,,
\end{equation*}
thereby establishing the differentiability of $\gamma$ at $h$.
\par 
Finally, if $x_1,x_2 \in C\cap B(x_0,\boldeta r_0)$ we let $s_1 = s_{x_1}$ and $s_2 = s_{x_2}$ and $h = \calH^1(\Gamma_{x_1,x_2}) = |s_1 - s_2|$ (since $\gamma$ is parametrized by arclength). Upon noticing that $\triangle_\gamma(s_1,h) = \triangle_\gamma(s_2,-h)$ we infer from the above inequality that
\begin{multline*}
\| \gamma'(s_1)-\gamma'(s_2)\| \leq \| \gamma'(s_1) - \triangle_\gamma(s_1,h)\| + \| \triangle_\gamma(s_2,-h) - \gamma'(s_2)\| \\ \leq 2 C^v_0 \omega(C_0^v|h|) \leq \bC \omega(\bC |s_1-s_2|)
\end{multline*}
for some universal constant $\bC > 0$, and the proof of the theorem is complete.
\end{proof}

\section{Application to the quasihyperbolic distance}
\label{sec.6}

\begin{Empty}[Local hypotheses about the ambient Banach space]
In \ref{61} $X$ is an arbitrary Banach space, in \ref{62} $X$ is a reflexive Banach space, and in \ref{63} $X$ is uniformly rotund and $\delta_X^{-1}$ is Dini.
\end{Empty}

\begin{Empty}[The quasihyperbolic distance]
\label{61}
In this section we assume $D \subset X$ is a nonempty open subset with the following property: For any pair $x,y \in D$ there exists a curve $\Gamma \subset D$ with endpoints $x$ and $y$, and $\calH^1(\Gamma) < \infty$. We now proceed to define on $D$ a new metric $d_q$, so-called the {\em quasihyperbolic distance} of $D$. We start by abbreviating
\begin{equation*}
h : D \to \R : x \mapsto \rmdist(x , X \setminus D) \,.
\end{equation*}
Since $D$ is open and bounded we notice that $0 < h(x) < \infty$, $x \in D$. It is also helpful to note, for further purposes, that $1/h$ is locally Lipschitzian. In fact, if $\eta > 0$ and $D_\eta := D \cap \{ x : h(x) > \eta \}$ then $\rmLip (1/h) \restriction_{D_\eta} \leq \eta^{-2}$, as follows from
\begin{equation*}
\left| \frac{1}{h(x)} - \frac{1}{h(y)} \right| = \left| \frac{ h(y)-h(x)}{h(y)h(x)} \right| \leq \frac{| h(y) - h(x) |}{\eta^2} \leq \frac{\|y-x\|}{\eta^2} \,,
\end{equation*}
$x,y \in D_\eta$. For $x_0,y_0 \in D$ we now define
\begin{multline*}
d_q(x_0,y_0) = \inf \bigg\{ \int_\Gamma \frac{1}{h} d\calH^1 : \Gamma \subset D \text{ is a curve of finite length} \\\text{ with endpoints $x_0$ and $y_0$} \bigg\} \,.
\end{multline*}
It is obvious that for each given competitor $\Gamma$ one has $1/h \in L_1(X,\calH^1 \hel \Gamma)$; in particular $d_q(x_0,y_0) < \infty$. There are two points to this section:
\begin{enumerate}
\item[(A)] The infimum in the definition of $d_q(x_0,y_0)$ is achieved by some curve $\Gamma$ provided $X$ is a separable reflexive Banach space and $D$ is convex;
\item[(B)] If $\Gamma$ is a curve that achieves the infimum in the definition of $d_q(x_0,y_0)$ then $\Gamma \setminus \{x_0,y_0\}$ is a $C^1$ open curve provided $X$ is uniformly rotund and $\delta_X^{-1}$ is Dini.
\end{enumerate}
In case $X = \ell_2^n$ is a finite dimensional Euclidean space, both results have been obtained by G. Martin \cite{MAR.85}. J. V\"ais\"al\"a proves the existence part (A), \cite{VAI.05}, and obtains the regularity part in case $X = \ell_2$ is Hilbert, \cite{VAI.07}. It should be noted that our method shows a quasihyperbolic geodesic in $\ell_2$ is $C^{1,1/2}$ in the interior, when in fact the stronger $C^{1,1}$ regularity entails from its quasihyperbolic property, see \cite{VAI.07}.
\end{Empty}
 
\begin{Empty}[Existence of quasihyperbolic geodesics]
\label{62}
We assume $X$ is reflexive.
As mentioned before, the existence part (A) above was proved by J. V\"ais\"al\"a. We now briefly comment on the proof. In order that  \ref{compactness} apply to a minimizing sequence $\{\Gamma_n\}$, we ought to show that 

\begin{itemize}
\item[(1)] the weight $w:X\to (0,+\infty]$ defined by $w(x)=1/h(x)$ if $x\in D$ and  $w(x)=+\infty$ otherwise,  is weakly* lower semicontinuous, i.e. weakly lower semicontinuous since $X$ is reflexive;
\item[(2)] there exists $R>0$ such that $\Gamma_n\subset B(x_0,R)$ for every $n$. 
\end{itemize}

The Hahn-Banach Theorem implies that $\rmclos D = \cap \calE$ where $\calE$ is the collection of translates $z+E_0$ of closed half spaces $E_0 \subset X$ such that $D \subset z + E_0$. Therefore $h(x) = \inf \{ \rmdist(x,E) : E \in \calE \}$, and establishing (1) boils down to showing that each $\psi_E : x \mapsto \rmdist(x,E)$ is weakly upper semicontinuous. In fact, as $H = \rmbdry E$ is a hyperplane, there exists a linear form $x^* : X \to \R$ such that $\rmdist(x,\rmbdry E) = | \la x, x^* \ra |$. Since this function of $x$ is Lipschitzian, one has $x^* \in X^*$, and the weak continuity of $\psi_E$ now readily follows. To establish (2) we assume that 
$M:=\sup_n\int_{\Gamma_n} w\, d\calH^1 < \infty$, and we write $L_n:=\calH^1(\Gamma_n)$. Considering an arclength parametrization $\gamma_n:[0,L_n]\to X$ of $\Gamma_n$ with $\gamma_n(0)=x_0$, we estimate
\begin{multline*}
\int_{\Gamma_n} w\, d\calH^1=\int_{0}^{L_n} \frac{1}{h(\gamma_n(t))}\, dt\geq \int_{0}^{L_n} \frac{1}{h(x_0)+\|\gamma_n(t)-\gamma_n(0)\|}\, dt\\
\geq \int_{0}^{L_n} \frac{1}{h(x_0)+t}\, dt=\log\left(1+\frac{L_n}{h(x_0)}\right)\,.
\end{multline*}
Whence $L_n\leq h(x_0)e^M$. Since $\|x-x_0\|\leq L_n$ for every $x\in \Gamma_n$, item (2) is proved. See \cite{VAI.05} for details.
\par 
We now apply \ref{subsequence} and \ref{compactness} to obtain a connected compact set $C \subset \rmclos D$ and a  (not relabeled)  subsequence
still denoted $\{\Gamma_n\}$, 
such that $\rmdist_\calH^*(\Gamma_n,C) \to 0$ and $\int_C w\, d\calH^1 \leq d_q(x_0,y_0)$. Recall \ref{prelim.curves} that $C$ contains a curve $\Gamma$ with endpoints $x_0$ and $y_0$; thus $\int_\Gamma w\, d\calH^1 \leq d_q(x_0,y_0)$. It remains to establish that $\Gamma \cap \rmbdry D = \emptyset$. For each $0 < 2\eta < \min\{h(x_0),h(y_0)\}$ we note that if $\Gamma \cap (D \setminus D_\eta) \neq \emptyset$ then $\calH^1(\Gamma \cap (D_\eta \setminus D_{2\eta})) \geq \eta$ and therefore $\int_{\Gamma \cap (D_\eta \setminus D_{2\eta})} w\, d\calH^1 \geq \frac{1}{2}$. As $\int_\Gamma w\, d\calH^1 < \infty$, $\Gamma$ can therefore meet at most finitely many annuli $D_{2^{-j}} \setminus D_{2^{-j+1}}$.
\end{Empty}

\begin{Empty}[Regularity of quasihyperbolic geodesics]
\label{63}
We assume $X$ is uniformly rotund and $\delta_X^{-1}$ is Dini.
Letting $\Gamma$ be a quasihyperbolic geodesic with endpoints $x_0$ and $y_0$, we recall from the argument above that $\Gamma \subset U = D_\eta$ for some $\eta > 0$, and also that $\rmLip w \restriction_{D_\eta} \leq \eta^{-2}$. It follows from \ref{lenghtminimiz} that \ref{mainregth} will apply with $\xi(r) \leq Cr$ at points $x \in \Gamma$ such that $\Theta^1(\calH^1 \hel \Gamma , x) = 1$. This is the case of each $x \in \mathring{\Gamma}$, according to \ref{toporeg}. Thus $\Gamma$ is, near $x$, a $C^{1,\omega}$ curve where $\omega$ is the mean slope of $\delta_X^{-1} \circ \xi$. Since $\xi(r) \leq Cr$, $\omega$ is asymptotic to the mean slope of $\delta_X^{-1}$. See also \ref{55}(D).
\end{Empty}

\section{Differentiability of almost minimizing curves in 2 dimensional rotund spaces}
\label{sec.8}

\begin{Empty}[Local hypotheses about the ambient Banach space]
In this section $X$ is 2 dimensional and (uniformly) rotund. We also assume that its norm $x \mapsto \|x\|$ is $C^2$ smooth on $X \setminus \{0\}$. We let $S_X = X \cap \{ v : \|v\|=1 \}$ be the unit circle.
\end{Empty}

\begin{Empty}[Local modulus of rotundity]
Here we localize (in direction) the definition of modulus of uniform rotundity of $X$. Specifically, for $v \in S_X$ and $0 < \veps \leq 1$, we define
\begin{equation*}
\delta_X(v;\veps) = \inf \left\{ 1 - \left \| \frac{v + (v+h)}{2} \right\| : h \in X ,  \|v+h\|\leq 1 \, \text{ and } \|h\|\geq\veps \right\} \,.
\end{equation*}
Our first aim is to show that the rotundity and $C^2$ smoothness of the norm imply $\delta_X(v;\veps)$ has the best possible behavior (i.e. is asymptotic to $\veps^2$) except possibly for a closed nowhere dense set of directions $v$.
\end{Empty}

\begin{Theorem}
\label{83}
There exists $G \subset S_X$ with the following properties.
\begin{enumerate}
\item[(A)] $S_X \setminus G$ is closed and has empty interior in $S_X$;
\item[(B)] For every $v_0 \in G$ there exists $\rho > 0$ and $C > 0$ such that for each $v \in S_X$, if $\|v-v_0\|< \rho$ then $\delta_X^{-1}(v;\veps) \leq C \sqrt{\veps}$ for every $0 < \veps \leq 1$.
\end{enumerate}
\end{Theorem}

\begin{proof}
We denote the norm of $X$ as $f : X \to \R : v \mapsto \|v\|$. Given $v \in S_X$ we choose a unit vector $e_v \in X$ that generates the tangent line $T_vS_X$. Thus $v,e_v$ is a basis of $X$ and we denote as $v^*,e_v^*$ the corresponding dual basis (i.e. $x = \la x , v^* \ra v + \la x , e_v^* \ra e_v$ for each $x \in X$). We observe that $Df(v) = v^*$ and that $D^2f(v)$ is a positive semidefinite bilinear form on $X$ (owing to the convexity of $f$),
\begin{equation}
\label{eq.400}
\|v+h\| = 1 + \la h , v^* \ra + \frac{1}{2}D^2f(v)(h,h) + o(\|h\|^2)
\end{equation}
whenever $h \in X$, where $o(\|h\|^2)$ is {\em little o} of $\|h\|^2$ uniformly in $v$. 
\par 
Upon noticing that the function 
\begin{equation*}
S_X \to \R : v \mapsto D^2f(v)(e_v,e_v)
\end{equation*}
is continuous, we immediately infer that
\begin{equation*}
G = S_X \cap \{ v : D^2f(v)(e_v,e_v) > 0 \}
\end{equation*}
is an open subset of $S_X$ and that
\begin{equation*}
B := S_X \setminus G = S_X \cap \{ v : D^2f(v)(e_v,e_v) = 0 \} \,.
\end{equation*}
Assume if possible that $B$ contains a nonempty open interval $V \subset S_X$. For each $v \in V$ one has $D^2f(v)(e_v,e_v)=0$ and $D^2f(v)(v,v)=0$ (because the norm $f$ is homogeneous of degree 1). Since $D^2f(v)$ is also symmetric, this clearly implies that $D^2f(v)=0$. The same argument applied to points of $r.V$, $r > 0$, shows that $D^2f$ vanishes identically in the open connected cone $C = X \cap \{ rv : r > 0 \text{ and } v \in V \}$. Therefore $f \restriction_C$ is the restriction to $C$ of a linear function $l : X \to \R$, which in turn shows that $V=C \cap \{l=1\}$ is a line segment, in contradiction with the rotundity of $f$. This proves (A)
\par 
We now turn to proving (B). Let $V \subset G$ be an open interval containing $v_0$. With $v \in V$ we associate $a(v) = \frac{1}{2} D^2f(v)(e_v,e_v)$, so that $a(v) > 0$. Since $v \mapsto a(v)$ is continuous, there is no restriction to assume that $V$ is small enough for $a(v) \geq a > 0$ uniformly in $v\in V$. If $t \in \R$ then \eqref{eq.400} implies
\begin{equation*}
\|v+te_v\| = 1 + \la te_v,v^* \ra + \frac{1}{2} D^2f(v)(te_v,te_v) + o(t^2) = 1 + a(v)t^2 + o(t^2) \,.
\end{equation*}
We let $h_t \in X$ be so that $v+h_t \in S_X$,
\begin{equation*}
v+h_t = \frac{v+te_v}{\|v+te_v\|} \,,
\end{equation*}
and we define $\veps_t = \|h_t\|$. Since
\begin{equation}
\label{eq.401}
h_t = \left( \frac{t}{1+a(v)t^2+o(t^2)} \right) e_v - \left( \frac{a(v)t^2+o(t^2)}{1+a(v)t^2+o(t^2)} \right) v
\end{equation}
we readily verify that
\begin{equation}
\label{eq.402}
\lim_{t \to 0} \frac{\veps_t}{t} = 1 \text{ uniformly in } v\in V \,.
\end{equation}
Furthermore, \eqref{eq.400} shows that
\begin{equation}
\label{eq.403}
\left\| v + \frac{h_t}{2} \right\| = 1 + \frac{1}{2} \la h_t , v^* \ra + \frac{1}{4} \left( \frac{1}{2} D^2f(v)(h_t,h_t) \right) + o(t^2) \,,
\end{equation}
and it follows from \eqref{eq.401} that
\begin{equation*}
\begin{split}
\la h_t , v^* \ra & = \frac{t}{1+a(v)t^2+o(t^2)} \la e_v,v^* \ra - \frac{a(v)t^2+o(t^2)}{1+a(v)t^2+o(t^2)} \la v,v^* \ra \\
& = - \frac{a(v)t^2+o(t^2)}{1+a(v)t^2+o(t^2)}\\
& = - a(v)t^2 + o(t^2) \,,
\end{split}
\end{equation*}
as well as
\begin{equation*}
\begin{split}
\frac{1}{2} D^2f(v)(h_t,h_t) & = \frac{t^2}{(1+a(v)t^2+o(t^2))^2} \left( \frac{1}{2} D^2f(v)(e_v,e_v) \right) \\
& \qquad + \left( \frac{a(v)t^2 + o(t^2)}{1+a(v)t^2+o(t^2)} \right)^2 \left( \frac{1}{2} D^2f(v)(v,v) \right) \\
& \qquad - 2 \left( \frac{t(a(v)t^2+o(t^2))}{(1+a(v)t^2+o(t^2))^2} \right) \left( \frac{1}{2} D^2f(v)(v,e_v) \right)\\
& = \frac{a(v)t^2}{(1+a(v)t^2+o(t^2))^2} + O(t^3) \\
& = a(v)t^2 + o(t^2) \,.
\end{split}
\end{equation*}
Plugging these into \eqref{eq.403} yields
\begin{equation*}
1 - \left\| v + \frac{h_t}{2} \right\| = \frac{a(v)}{4}t^2 + o(t^2) \,.
\end{equation*}
In view of \eqref{eq.402} this means that one can find $\varepsilon_0>0$ such that 
\begin{equation}\label{bientop}
1 - \left\| v + \frac{h_t}{2} \right\| \geq \frac{a(v)}{5}\|h_t\|^2 \quad \forall \|h_t\|\leq \varepsilon_0\,.
\end{equation}
Taking if necessary a smaller $\varepsilon_0$, we can also assume that any $h\in X$ satisfying both $v+h \in S_X$ and $\|h\|\leq \varepsilon_0$ is of the form $h_t$ for some $t$.

We let now $ \varepsilon >0$ be given and consider an arbitrary $v+h\in S_X$ with $\|h\|\geq \varepsilon$. We distinguish two cases. Firstly  if $\|h\|\leq \varepsilon_0$ then \eqref{bientop} says 
$$1 - \left\| v + \frac{h}{2} \right\| \geq \frac{a(v)}{5}\varepsilon^2\,.$$
Secondly if $\|h\|\geq \varepsilon_0$ then 
$$1 - \left\| v + \frac{h}{2} \right\| \geq M\varepsilon^2, $$
where $M$ is defined by 
$$M:=\min \left\{ \frac{1 - \left\| v + \frac{h}{2} \right\|}{\|h\|^2} \; : \; \{v,v+h\} \in S_X  \, \text{ and } \|h\|\geq\varepsilon_0 \right\}, $$
which is positive due to the uniform rotundity of the norm. We just have proved, owing also to Remark \ref{stupidremark}, that
\begin{equation*}
\delta_X(v;\veps) \geq C_0 \veps^2 \quad \forall v\in V,
\end{equation*}
with $C_0=\min(M,\frac{a}{5})$ and Conclusion (B) now easily follows.
\end{proof}

\begin{Empty}[Excess of length of a nonstraight path]
\label{84}
Here we notice that \ref{excess.length} can be improved in the following way. In case $x_0,x_1$ and $z$ are the vertices of a nondegenerate triangle in $X$, one has
\begin{equation*}
\|x_0-z\| + \|z-x_1\| \geq \|x_0-x_1\| \left( 1 + \delta_X \left( v; \frac{ \rmdist(z, x_0 + \rmspan \{x_1-x_0\}) } { 2 (\|x_0-z\| + \|z-x_1\| ) } \right) \right)\,.
\end{equation*}
where
\begin{equation*}
v = \frac{x_1-x_0}{\|x_1-x_0\|} \,.
\end{equation*}
The proof is {\em exactly} the same as that of \ref{excess.length} once one recognizes that $\delta_X(\veps)$ can be replaced by $\delta_X(v;\veps)$ in \eqref{eq.5.bis}.
\end{Empty}

\begin{Empty}[Height bound]
\label{85}
Under our assumptions on $X$, the height bound \ref{keylemma}(G) can be improved to
\begin{equation*}
\rmdist(z,x+L) \leq 16 . \delta_X^{-1}(v;\xi(r))
\end{equation*}
where $v=(x_1-x_0)\|x_1-x_0\|^{-1}$. Here $\delta^{-1}_X(v;\xi(r))$ denotes the value at $\xi(r)$ of the reciprocal of the function $\delta_X(v;\cdot)$. The proof, based on \ref{84}, is identical.
\end{Empty}

\begin{Empty}[Tangent lines]
\label{86}
We now indicate how to apply the previous ``localized'' observations to the study of regularity of almost minimizing curves, using the notion of {\em tangent measure}. We point out that everything we do in this number makes sense in a finite dimensional rotund Banach space. We will often refer to \cite{DEP.01} and \cite{PRE.87}, but only to {\em elementary} results in these papers, in particular those that {\em do not} depend on the Euclidean structure of the ambient space $\R^n$ considered there. 
\par 
Given $C \subset X$ a set which is $(\xi,r_0)$ almost minimizing in $U \subset X$ with respect to a Dini gauge $\xi$, we consider the finite Borel measure $\mu = \calH^1 \hel C$ in $U$. Given $x \in U$ and $r > 0$ we next consider the finite Borel measure $r^{-1} (T_{x,r})_*\mu$ on $(U-x)/r$, where $T_{x,r}(y)=(y-x)/r$. A 1 dimensional tangent measure of $\mu$ at $x$ is, by definition, a weak* limit of some sequence $\{r_j^{-1}(T_{x,r_j})_*\mu\}$, where $r_j \downarrow 0$. The collection of those is denoted $\rmTan^{(1)}(\mu,x)$. We gather useful information about tangent measures.
\begin{enumerate}
\item[(A)] If $x \in C \cap U$ then $\rmTan^{(1)}(\mu,x) \neq \emptyset$;
\item[(B)] If $x \in C \cap U$, $\Theta^1(\calH^1 \hel C , x)=1$, and $\nu \in \rmTan^{(1)}(\mu,x)$, then $\nu = \calH^1 \hel W$ for some line $W\in \bG(X,1)$;
\item[(C)] If $x \in C \cap U$ and $\Theta^1(\calH^1 \hel C,x) =1$, recall \ref{lip.reg.bis} that $x$ is a regular point of $C$: If $\gamma$ is an arclength parametrization of some $C \cap B(x,r)$ with $\gamma(0)=x$, then $\gamma$ is differentiable at 0 if and only if $\rmTan^{(1)}(\mu,x)$ is singletonic;
\item[(D)] If $x \in C \cap U$, $\Theta^1(\calH^1 \hel C , x)=1$, then the collection of tangent lines $\bG(X,1) \cap \{ W: \calH^1 \hel W \in \rmTan^{(1)}(\mu,x) \}$ is connected.
\end{enumerate}
\par

\begin{proof}[Proof of {\rm (A)}] The proof of (A) depends on our assumption $\rmdim X < \infty$ through the application of a compactness Theorem for Radon measures, see e.g. \cite[Proposition~5.4]{DEP.01}. The relevant almost monotonicity property of $\mu$ is in \ref{monotonicity}.
\end{proof}
\par 

\begin{proof}[Proof of {\rm (B)}]
Let $r_j \downarrow 0$ be such that $\nu$ is the weak* limit of the sequence $\{\mu_j  \}$ where $\mu_j = r_j^{-1} (T_{x,r_j})_*\mu$. Abbreviate $C_j = (C-x)/r_j$. Using the translation invariance, and behavior under homothethy, of the Hausdorff measure, one immediately checks that $C_j$ is $(\xi_j,r_0/r_j)$ almost minimizing in $U_{x,r_j}=(U-x)/r_j$, where $\xi_j(t)=\xi(tr_j)$, and that $\mu_j = \calH^1 \hel C_j$. In the vocabulary of \cite{DEP.01}, each $\mu_j$ is 1 concentrated, according to \ref{density.first}(C). Notice that \ref{monotonicity} does {\em not} show $\mu_j$ is $(\xi_j,1)$ almost monotone in the sense of \cite{DEP.01} (because the condition is verified only at those points of the support of $\mu_j$, not all points of $U_{x,r_j}$): We will emphasize this by saying that $\mu_j$ is $(\xi_j,1)$ almost monotone {\em on its support}. Carefully reading the proofs of \cite[4.2 and 4.1]{DEP.01} reveals that 
\begin{enumerate}
\item[(1)] For every $\lambda > 0$ and every $\delta > 0$ there exists $j_0$ such that $C_j \cap B(0,\lambda) \subset B(\rmsupp \nu,\delta)$ for every $j\geq j_0$;
\item[(2)] $\Theta^1(\nu,z) \geq 1$ for every $z \in \rmsupp \nu$.
\end{enumerate}
We put $Z = \rmsupp \nu$, $Z_\lambda = Z \cap B(0,\lambda)$ and $F_\lambda = Z_\lambda \cap \rmbdry B(0,\lambda)$ for each $\lambda > 0$. 
Select a subsequence $\{j_k\}$ satisfying $\sum_k2^k\xi(2^k r_{j_k}) <1/6$. By the assumption on the density and \ref{local.topology} (D), we may assume that 
$$\frac{\calH^1(C_{j_k} \cap B(0,2^k))}{2^{k+1}}-1<\frac{1}{6}\,,$$
and that $\rmcard C_{j_k}\cap  \rmbdry B(0,2^k)\geq 2$ for every $k\geq 1$. Setting 
$$G_k:=\{\lambda\in[0,2^k] :\rmcard C_{j_k}\cap  \rmbdry B(0,\lambda)\geq 3 \} \,,$$
we infer from \ref{eilenberg} that $\calL^1(G_k)<2^{k}/3$ . 
Then we can find for each integer $k\geq 1$ a radius $\lambda_k\in(2^{k+1}/3,2^k)$ such that $\rmcard C_{j_k}\cap  \rmbdry B(0,\lambda_k) =2$. Writing $\{x^1_k,x^2_k\}:= C_{j_k}\cap  \rmbdry B(0,\lambda_k)$, the set $C_{j_k}^\prime:=(C_{j_k}\setminus  B(0,\lambda_k))\cup[0,x^1_k]\cup[0,x^2_k]$
is a competitor for $C_{j_k}$. From the almost minimality of $C_{j_k}$ we infer that 
$$ \calH^1(C_{j_k} \cap B(0,\lambda_k))\leq 2\lambda_k(1+\xi(\lambda_kr_{j_k}))\leq 2^{k+1}+ 2^{k+1}\xi(2^kr_{j_k})\,.$$
Setting $H_k:=G_k\cap [0,\lambda_k]$, we deduce from \ref{eilenberg} that $\calL^1(H_k)\leq 2^{k+1}\xi(2^kr_{j_k})$. 
Therefore $\calL^1(H)<1/3$ where $H:=\cup_k H_k$. 
Consequently, for each integer $k\ge 1$ we can find a radius $R_k\in(\lambda_k-1/3,\lambda_k)\setminus G$. We shall keep in mind that $R_k\in(2^{k-1},2^k)$ by construction. 

Let us now fix an arbitrary integer $k\geq 1$. The way we have selected the radius $R_k$ ensures that $\rmcard C_{j_h}\cap  \rmbdry B(0,R_k)=2$ for every  $h\geq k$.  We write 
$\{y^1_{k,h},y^2_{k,h}\}:= C_{j_h}\cap  \rmbdry B(0,R_k)$. Noticing that the set $C_{j_h}^{\prime\prime}:=(C_{j_h}\setminus  B(0,R_k))\cup[y^1_{k,h},y^2_{k,h}]$ is a competitor for $C_{j_h}$, the almost minimality of $C_{j_h}$ yields
$$\calH^1(C_{j_h} \cap B(0,R_k))\leq (1+\xi(R_kr_{j_h}))\|y^1_{k,h}-y^2_{k,h}\| \,.$$
On the other hand, $\lim_h\calH^1(C_{j_h} \cap B(0,R_k))=2R_k={\rm diam}\, B(0,R_k)$ by the density assumption.  By the inequality above, it then follows that 
$\lim_h\|y^1_{k,h}-y^2_{k,h}\|=2R_k$. Referring to (1), we now choose a sequence $\delta_h \downarrow 0$ such that for every $h\geq k$ there exist 
$z^1_{k,h},z^2_{k,h} \in B(0,R_k+\delta_h)$ with $\|y^i_{k,h}-z^i_{k,h}\|\leq \delta_h$, $i=1,2$. Taking a subsequence if necessary, $z^i_{k,h} \to z^i_k$ as $h \to \infty$,  $z^i_k \in F_{R_k}$, and $\|z_k^1-z_k^2\|=2R_k$. According to \ref{excess.length}, the rotundity of the norm implies that the line segment $[z_k^1,z_k^2]$ is a diameter of $B(0,R_k)$, i.e. $0\in [z_k^1,z_k^2]$. We claim that $ [z_k^1,z_k^2]\subset Z_{R_k}$. To prove the claim, we first notice that $C_{j_h}\cap B(0,R_k)$ contains a curve $\Gamma_h$ whose endpoints are $\{y^1_{k,h},y^2_{k,h}\}$, see \ref{local.topology} (E). By the Blashke Selection Principle, up to a further subsequence, $\Gamma_h\to K_*$ in the Hausdorff distance 
for some compact connected set $K_*\subset B(0,R_k)$ containing $z^1_k$ and $z_k^2$. In particular $\calH^1(K_*)\geq \|z_k^1-z_k^2\|=2R_k$. On the other hand, by lower semicontinuity of $\calH^1$ with respect to Hausdorff convergence, we have 
$$\calH^1(K_*)\leq \liminf_{h\to \infty}  \calH^1(\Gamma_h)\leq \lim_{h\to\infty} \calH^1(C_{j_h} \cap B(0,R_k)) =2R_k\,.$$
Hence $\calH^1(K_*)=2R_k$, and the rotundity of the norm yields $K_*=[z_k^1,z_k^2]$. In view of (1), we have thus proved the claim. It now follows from 
(2) and \cite[2.10.19(3)]{GMT} that $\nu\restr B(0,R_k) \geq \calH^1\restr  [z_k^1,z_k^2]$. Finally, the monotonicity stated in \ref{monotonicity} and the density assumption classicaly implies that $\nu(B(0,\lambda))=2\lambda$ for every $\lambda>0$. Therefore $\nu(B(0,R_k)) =\calH^1([z_k^1,z_k^2])$, whence $\nu\restr B(0,R_k) = \calH^1\restr  [z_k^1,z_k^2]$. From the arbitrariness of $k$ we conclude that $Z$ is a line through the origin and $\nu=\calH^1\restr Z$, which completes the proof of (B). 

\end{proof}

\par 
We leave the easy proof of (C) to the reader. It relies on the local convergence of $C_j$ to $Z$ in Hausdorff distance, as was already used in the proof of (B).
\par 
Conclusion (D) is our principal tool in this section. It is a consequence of \cite[Theorem 2.6]{PRE.87}  as we now explain.

\begin{proof}[Proof of {\rm (D)}] Let us recall that a tangent measure of $\mu$ at the point $x$ in the sense of \cite{PRE.87} is a weak* limit of some sequence $\{c_j(T_{x,r_j})_*\mu\}$, where $r_j \downarrow 0$ and $c_j>0$. Following the notations of \cite{PRE.87}, we write $\rmTan(\mu,x)$ the collection of all tangent measures of $\mu$ at $x$. The set $\rmTan(\mu,x)$ is endowed with the topology induced by the metric 
$${\mathbf D}(\nu_1,\nu_2):= \sum_{p\in\mathbb{N}}2^{-p}\min(1,F_p(\nu_1,\nu_2))\,,$$
where
$$F_p(\nu_1,\nu_2):=\sup\left\{\left|\int_{\R^n}f\,d\nu_1-\int_{\R^n}f\,d\nu_2\right| : \rmsupp f \subset B(0,p)\,,\;f\geq 0\,,\;{\rm Lip}(f)\leq 1\right\}$$
(recall that the weak* convergence of Radon measures coincides with the convergence with respect to  ${\mathbf D}$). 
Now we observe that \ref{Ahlfors} together with \ref{monotonicity} shows that for $r>0$ small enough,  
$$\kappa^{-1}r\leq \frac{\mu(B(x,r))}{r}\leq \kappa r \,,$$
for some constant $\kappa>0$. It  implies that $\rmTan(\mu,x)=\{c\nu : c>0\,,\; \nu \in \rmTan^{(1)}(\mu,x)\}$. In turn, we infer from (B) that 
 $\rmTan(\mu,x)=\{c \calH^1 \hel W : c>0\,,\; W\in \bG(X,1)\} $. 
According to \cite[Theorem 2.6]{PRE.87}, we can conclude that $\rmTan(\mu,x)$ is connected. It is now elementary to check 
that the map $\nu=c  \calH^1 \hel W \in \rmTan(\mu,x)\mapsto W\in \bG(X,1)$ is continuous, and (D) follows. 
%
\end{proof}

\end{Empty}

\begin{Theorem}[Differentiable regularity]
\label{87}
We recall our general assumption for this section that $X$ is a 2 dimensional rotund Banach space. Assume that
\begin{enumerate}
\item[(A)] $C \subset X$ is compact, connected, $U \subset X$ is open, $x_0 \in C \cap U$, $r_0 > 0$, $B(x_0,r_0) \subset U$;
\item[(B)] $\xi$ is a gauge and $\sqrt{\xi}$ is Dini;
\item[(C)] $C$ is $(\xi,r_0)$ almost minimizing in $U$;
\item[(D)] $\Theta^1(\calH^1 \hel C ,x_0)=1$.
\end{enumerate}
It follows that there exists $r > 0$ such that $C \cap B(x_0,r)$ is a {\em differentiable} curve.
\end{Theorem}

\begin{proof}
We explain how to modify the proof of \ref{mainregth}. Since $x_0$ is a regular point of $C$, recall \ref{def.reg}, we may assume (letting $r_0$ be smaller if necessary) each $x \in C \cap B(x_0,r_0)$ is regular as well, and $C \cap B(x,r_0)$ is a Lipschitz curve, according to \ref{lip.reg.bis}. In order to establish the differentiability of a corresponding arclength parametrization, it suffices to show $\rmTan^{(1)}(\mu,x_0)$ contains excatly one element, according to \ref{86}(C). We let $G$ be the set associated with the norm of $X$ in \ref{83}. Abbreviate $\rmTan(C,x_0) = S_X \cap \{ v : \calH^1 \hel \rmspan\{v\} \in \rmTan^{(1)}(\calH^1 \hel C , x_0)\}$. We will consider the following alternative: {\bf} Either $\rmTan(C,x_0) \cap  G = \emptyset$, {\bf or} $\rmTan(C,x_0) \cap G \neq \emptyset$. In the first case, the connectedness property stated in \ref{86}(D) and the fact that $S_X \setminus G$ has empty interior imply easily that $\rmTan(C,x_0)$ is a singleton and the proof is complete. In the second case we pick $\rmspan\{v_0\} = W_0 \in \bG(X,1)$ such that $v_0 \in G$, $\calH^1 \hel W_0$ is a tangent measure to $\calH^1 \hel C$ at $x_0$, and we choose $\eta > 0$ so that if $v \in S_X$ and $\|v-v_0\| < \eta$ then $v \in G$. The proof of \ref{mainregth} remains unchanged until the end of the proof of {\sc Claim \#2}. The application of \ref{keylemma} is replaced by an application of \ref{85} together with \ref{83}(B). To each $r > 0$ such that $(x_0,r)$ is a good pair we associate the unit vector $v_r$ generating the line that joins the endpoints of $\Gamma_{x_0,r}$. We may choose $r_1 > 0$ small so that $\| v_{r_1} - v_0 \| < \eta/2$. We proceed through the remaining part of the proof of \ref{mainregth} until near \eqref{eq.211}. In our new notations, this shows that $\|v_{r_{j+1}} - v_{r_j} \| \leq C \sqrt{\xi(r_j)}$. However for the computations to be valid, one must make sure at each stage of the iteration that the unit vectors $v_{r_j}$ still belong to $G$, in fact we ought to guarantee that $\|v_{r_j}-v_0\| < \eta$  in order that \ref{83}(B) applies. This of course can be enforced by choosing $r_1$ so small that $C \sum_{j=1}^\infty \sqrt{r_j} < \eta/2$.
\end{proof}


\bibliographystyle{amsplain}
\bibliography{thdp}

\providecommand{\bysame}{\leavevmode\hbox to3em{\hrulefill}\thinspace}
\providecommand{\MR}{\relax\ifhmode\unskip\space\fi MR }
\providecommand{\MRhref}[2]{%
  \href{http://www.ams.org/mathscinet-getitem?mr=#1}{#2}
}
\providecommand{\href}[2]{#2}
\begin{thebibliography}{10}

\bibitem{ALM.76}
F.J. Almgren, \emph{Existence and regularity almost everywhere of solutions to
  elliptic variational problems with constraints}, Memoirs of the AMS, no. 165,
  American Math. Soc., 1976.

\bibitem{AMB.KIR.00}
L.~Ambrosio and B.~Kirchheim, \emph{Currents in metric spaces}, Acta Math.
  \textbf{185} (2000), no.~1, 1--80.

\bibitem{AMBROSIO.TILLI}
Luigi Ambrosio and Paolo Tilli, \emph{Topics on analysis in metric spaces},
  Oxford Lecture Series in Mathematics and its Applications, vol.~25, Oxford
  University Press, Oxford, 2004. \MR{2039660 (2004k:28001)}

\bibitem{COHN}
D.L. Cohn, \emph{Measure theory}, Birkh\"auser, 1980.

\bibitem{d3}
Guy David, \emph{H\"older regularity of two-dimensional almost-minimal sets in
  {$\Bbb R^n$}}, Ann. Fac. Sci. Toulouse Math. (6) \textbf{18} (2009), no.~1,
  65--246. \MR{MR2518104}

\bibitem{DEP.01}
Th. {De Pauw}, \emph{Nearly flat almost monotone measures are big pieces of
  {L}ipschitz graphs}, J. of Geom. Anal. \textbf{12} (2002), no.~1, 29--61.

\bibitem{DEVILLE.GODEFROY.ZIZLER}
Robert Deville, Gilles Godefroy, and V{\'a}clav Zizler, \emph{Smoothness and
  renormings in {B}anach spaces}, Pitman Monographs and Surveys in Pure and
  Applied Mathematics, vol.~64, Longman Scientific \& Technical, Harlow, 1993.
  \MR{1211634 (94d:46012)}

\bibitem{GMT}
Herbert Federer, \emph{Geometric {M}easure {T}heory}, Die grundlehren der
  mathematischen wissenschaften, vol. 153, Springer-Verlag, New York, 1969.

\bibitem{GRO.83}
Mikhael Gromov, \emph{Filling {R}iemannian manifolds}, J. Differential Geom.
  \textbf{18} (1983), no.~1, 1--147. \MR{697984 (85h:53029)}

\bibitem{KIR.94}
Bernd Kirchheim, \emph{Rectifiable metric spaces: local structure and
  regularity of the {H}ausdorff measure}, Proc. Amer. Math. Soc. \textbf{121}
  (1994), no.~1, 113--123.

\bibitem{LINDENSTRAUSS.TZAFRIRI.II}
Joram Lindenstrauss and Lior Tzafriri, \emph{Classical {B}anach spaces. {II.
  Function spaces}}, Ergebnisse der Mathematik und ihrer Grenzgebiete [Results
  in Mathematics and Related Areas], vol.~97, Springer-Verlag, Berlin, 1979,
  Function spaces. \MR{540367 (81c:46001)}

\bibitem{MAR.85}
Gaven~J. Martin, \emph{Quasiconformal and bi-{L}ipschitz homeomorphisms,
  uniform domains and the quasihyperbolic metric}, Trans. Amer. Math. Soc.
  \textbf{292} (1985), no.~1, 169--191. \MR{805959 (87a:30037)}

\bibitem{MEI.09}
Thomas Meinguet, \emph{{$(\bold M,cr^\gamma,\delta)$}-minimizing curve
  regularity}, Bull. Belg. Math. Soc. Simon Stevin \textbf{16} (2009), no.~4,
  577--591. \MR{2583547 (2011a:49093)}

\bibitem{MOR.94}
F.~Morgan, \emph{$(\mathbf{M},\veps,\delta)$-minimal curve regularity}, Proc.
  Amer. Math. Soc. \textbf{120} (1994), 677--686.

\bibitem{PRE.87}
D.~Preiss, \emph{Geometry of measures in $\mathbf{R}^n$: Distribution,
  rectifiability, and densities}, Ann. of Math. (2) \textbf{125} (1987),
  537--643.

\bibitem{VAI.05}
Jussi V{\"a}is{\"a}l{\"a}, \emph{Quasihyperbolic geodesics in convex domains},
  Results Math. \textbf{48} (2005), no.~1-2, 184--195. \MR{2181248
  (2006h:30035)}

\bibitem{VAI.07}
\bysame, \emph{Quasihyperbolic geometry of domains in {H}ilbert spaces}, Ann.
  Acad. Sci. Fenn. Math. \textbf{32} (2007), no.~2, 559--578. \MR{2337495
  (2008d:30041)}

\end{thebibliography}


\end{document}